\documentclass[10pt,reqno]{amsart}

\usepackage[margin=1in]{geometry}
\usepackage{amsmath,amssymb,amsthm,mathtools}
\usepackage{mathrsfs}
\usepackage{enumitem}
\usepackage{microtype}
\usepackage{graphicx}
\usepackage{float}
\usepackage{tikz}
\usetikzlibrary{arrows.meta,positioning}
\numberwithin{equation}{section}
\usepackage{hyperref}

\usepackage{sidecap}

\newtheorem{theorem}{Theorem}[section]
\newtheorem{proposition}{Proposition}[section]
\newtheorem{lemma}{Lemma}[section]
\newtheorem{corollary}{Corollary}[section]
\newtheorem{remark}{Remark}[section]
\theoremstyle{definition}
\newtheorem{definition}{Definition}[section]

\newcommand{\R}{\mathbb R}
\newcommand{\dd}{\mathrm d}
\newcommand{\dx}{\mathrm{d}x}

\newcommand{\dt}{\mathrm{d}t}
\newcommand{\ds}{\mathrm{d}s}
\newcommand{\dy}{\mathrm{d}y}

\newcommand{\dv}{\textnormal{d}v}

\newcommand{\e}{\varepsilon}
\newcommand{\one}{\mathbf 1}
\newcommand{\p}{\partial}
\newcommand{\la}{\langle}
\newcommand{\ra}{\rangle}
\newcommand{\ov}{\overline}

\newcommand{\tnd}{\textnormal{d}}
\newcommand{\ddt}{\frac{\textnormal{d}}{\textnormal{d}t}}
\newcommand{\ve}{\varepsilon}
\newcommand{\weakto}{\rightharpoonup}
\newcommand{\weakstarto}{\overset{\star}{\rightharpoonup}}
\newcommand{\bq}{\begin{equation}}
\newcommand{\eq}{\end{equation}}
\newcommand{\lt}{\left}
\newcommand{\rt}{\right}

\newcommand{\calD}{\mathcal D}

\newcommand{\calG}{\mathcal G}

\newcommand{\calM}{\mathcal M}

\newcommand{\calO}{\mathcal O}

\newcommand{\calU}{\mathcal U}
\newcommand{\calV}{\mathcal V}

\newcommand{\frkB}{\mathfrak B}

\newcommand{\frkM}{\mathfrak M}

\newcommand{\frkQ}{\mathfrak Q}

\newcommand{\scrA}{\mathscr A}
\newcommand{\scrB}{\mathscr B}

\newcommand{\scrD}{\mathscr D}
\newcommand{\scrE}{\mathscr E}

\newcommand{\scrH}{\mathscr H}

\newcommand{\scrQ}{\mathscr Q}

\newcommand{\scrV}{\mathscr V}

\newcommand{\frka}{\mathfrak a}
\newcommand{\frkb}{\mathfrak b}

\newcommand{\frkl}{\mathfrak l}

\newcommand{\frkr}{\mathfrak r}

\newcommand{\bbN}{\mathbb N}

\usepackage[colorinlistoftodos]{todonotes}

\usepackage{cancel}
\usepackage{soul}

\title[BGK equation with physical collision frequency in a slab]{Global large-data solutions for the BGK equation with physical collision frequency in a slab}

\author[Choi]{Young-Pil Choi}
\address[Young-Pil Choi]{\newline Department of Mathematics\newline
Yonsei University, 50 Yonsei-Ro, Seodaemun-Gu, Seoul 03722, Republic of Korea}
\email{ypchoi@yonsei.ac.kr}

 \author[Song]{Sihyun Song}
 \address[Sihyun Song]{\newline Department of Mathematics\newline
 Yonsei University, 50 Yonsei-Ro, Seodaemun-Gu, Seoul 03722, Republic of Korea}
 \email{ssong@yonsei.ac.kr}

\begin{document}

\allowdisplaybreaks

\date{\today}


\begin{abstract}
We establish a global large-data existence theory for the Bhatnagar--Gross--Krook equation in a slab with the physical collision frequency given by the local density. We consider both inflow and Maxwell boundary conditions and assume only finite mass, kinetic energy, and entropy of the initial data, together with the corresponding inflow flux bounds. The data need not be small or pointwise bounded, and we impose neither positive lower bounds on the density or temperature nor prescribed velocity moments above order two. To our knowledge, this is the first nonperturbative global existence result for this initial and boundary value problem under these assumptions. The central obstacle is the simultaneous loss of energy compactness and integrability of the unrenormalized collision operator. Finite energy alone does not give uniform integrability of the second velocity moment, while the density dependent collision frequency prevents separate local $L^1$ control of the gain and loss terms. We overcome these difficulties through a compensated estimate uniform in the truncation of the collision frequency, a de la Vall\'ee Poussin weight determined by the data, and two successive renormalizations using entropy production. This yields the compactness of the macroscopic energy, local Maxwellians, and nonlinear collision term, and hence global renormalized solutions.
\end{abstract}

\maketitle

\tableofcontents

%
%
%
%
%
%
%
%
%
%
%

\section{Introduction}

%
%
%
%
%
%
%
%
%
%
%

\subsection{The BGK equation in a slab}

The BGK model, introduced by Bhatnagar, Gross, and Krook \cite{BGK54}, replaces the Boltzmann collision operator by relaxation towards a local Maxwellian while preserving the basic collision invariants. The choice of collision frequency is analytically decisive. A constant frequency leads to a substantially milder relaxation problem, whereas the physical choice $\nu=\rho_f$ makes the collision rate itself depend on the unknown solution and restores the quadratic density scaling associated with binary collisions. Our aim is to develop a global large-data theory for this physical collision frequency under only the natural mass, energy, and entropy conditions, allowing vacuum and without imposing any prescribed velocity moment above the energy.

We consider one spatial dimension and three-dimensional velocities. More precisely, let
\[
\calO=I\times\R^3, \quad I=\lt(-\frac12,\frac12\rt),
\]
and write $(x,v)=(x,v_1,v_2,v_3)\in\calO$. The equation takes the form
\bq\label{eq/ bgk}
\begin{cases}
\p_t f+v_1\p_x f=\rho_f(M[f]-f ), & (t,x,v)\in(0,\infty)\times\calO,\\
f(0,x,v)=f_0(x,v).
\end{cases}
\eq
Here the density, bulk velocity, and temperature are defined by
\[
\rho_f=\int_{\R^3}f\,\dv, \quad \rho_f u_f=\int_{\R^3}v f\,\dv, \quad 3\rho_f\theta_f =\int_{\R^3}|v-u_f|^2f\,\dv,
\]
and the associated local Maxwellian is
\[
M[f] =\frac{\rho_f}{(2\pi\theta_f)^{\frac32}} \exp\lt(-\frac{|v-u_f|^2}{2\theta_f}\rt).
\]
On the vacuum set $\{\rho_f=0\}$, we set $u_f=\theta_f=0$ and $M[f]=0$.

By construction, $M[f]$ has the same mass, momentum, and energy as $f$:
\bq\label{eq/ M mom}
\int_{\R^3} (M[f]-f )
\begin{pmatrix}
1\\
v\\
|v|^2
\end{pmatrix}
\dv=0.
\eq
Thus the relaxation operator preserves the collision invariants locally. The corresponding global balance laws depend on the boundary condition imposed at $x=\pm\frac12$, which will be specified below.

%
%
%
%
%
%
%
%
%
%
%

\subsection{The physical collision frequency and the Boltzmann closure}

The appearance of the local density $\rho_f$ in the collision frequency can be understood directly from the gain--loss structure of the Boltzmann collision operator. At a fixed space-time point, write
\[
Q_B(f,f)=Q_B^+(f,f)-\nu_f(v)f(v), \quad \nu_f(v) = \iint_{\R^3 \times \mathbb S^2} B(v-v_*,\sigma)f(v_*)\,\dd\sigma \dv_*.
\]
For cutoff Maxwell molecules,
\[
B(z,\sigma)=b(\cos\theta), \quad \cos\theta=\frac{z\cdot\sigma}{|z|},
\]
and the angular integral
\[
\kappa:=\int_{\mathbb S^2}b(e\cdot\sigma)\,\dd\sigma
\]
is independent of $e\in\mathbb S^2$. Hence
\[
\nu_f(v)=\kappa\rho_f, \quad Q_B^-(f,f)=\kappa\rho_f f.
\]

The classical BGK approximation may then be viewed as a Maxwellian closure of the gain term. Indeed, for each collision invariant $\psi(v)\in\{1,v,|v|^2\}$,
\[
\int_{\R^3}\psi(v)Q_B^+(f,f)\,\dv = \kappa\rho_f\int_{\R^3}\psi(v)f(v)\,\dv.
\]
Since the local Maxwellian $M[f]$ has the same mass, momentum, and energy as $f$, the Maxwellian gain $\kappa\rho_f M[f]$ has exactly the same collision-invariant moments as $Q_B^+(f,f)$. Replacing the Boltzmann gain by this local equilibrium distribution thus gives
\[
Q_B^+(f,f)\ \rightsquigarrow\ \kappa\rho_f M[f], \quad Q_{\rm BGK}(f) =\kappa\rho_f(M[f]-f ).
\]
After the normalization $\kappa=1$, this is precisely the collision operator in \eqref{eq/ bgk}. This is a relaxation closure of the Boltzmann collision operator rather than a rigorous scaling limit; see \cite{BGK54,BHP21,BPS23}.

The density factor is thus not merely an additional variable coefficient. It preserves the quadratic density homogeneity of binary collisions: for every $a>0$,
\[
M[af]=aM[f], \quad Q_{\rm BGK}(af)=a^2Q_{\rm BGK}(f).
\]
In contrast, a constant-frequency BGK operator has only linear homogeneity in the density. The simple choice $\nu_f=\rho_f$ is specific to the Maxwell-molecule setting; more general Boltzmann kernels lead to velocity- and temperature-dependent collision frequencies.
\subsection{Previous results and the large-data problem}

The mathematical theory of BGK equations has a long history. Early existence results and continuity properties of the local Maxwellian were established in \cite{PG83,GP89}. For the constant-frequency Cauchy problem, global solutions were constructed in \cite{Per89} for data with finite mass, energy, and entropy, together with a spatial second-moment condition. A key ingredient was a local space-time third-velocity-moment estimate generated by free-transport dispersion, rather than imposed on the initial data, which provided the compactness needed to identify the temperature. The corresponding Dirichlet boundary-value problem was studied in \cite{PD91}. Weighted $L^\infty$ estimates, uniqueness, and further Cauchy theories were subsequently developed in \cite{PP93,Mis96,ZH07,CZ16}.

The physical collision frequency $\nu=\rho_f$ leads to a substantially different large-data problem. The distinction between the original density dependent rate and a constant rate was emphasized in the particle approximation literature \cite{BHP21,BP23}; in particular, \cite{BP23} points out the lack of a constructive existence and uniqueness theory for the former. More recently, \cite[Section 1.1]{EMY26} explicitly identifies well-posedness for the original self-consistent model with $\lambda(\rho)=\rho$ as an open problem. Theorem \ref{thm1} below addresses the global existence part of this problem in a slab with inflow or Maxwell reflection; uniqueness and more general collision frequencies are outside the scope of the result. The difficulty is not merely the presence of an unbounded coefficient: $\rho_f$ multiplies a local Maxwellian whose density, velocity, and temperature are themselves determined by the solution, and the available mass and energy bounds do not directly provide the compactness needed to pass to this nonlinear collision term. In particular, it appears that the aforementioned third-velocity-moment method does not apply in this setting; those arguments rely on the fact that the collision operator of \eqref{eq/ bgk} has bounded second moments, and this is unclear in the setting where the collision frequency is given by the local density. Thus, the compactness of the temperature cannot be obtained in this way.

Under additional structural or perturbative assumptions, however, several theories for nonconstant collision frequencies are available. Near-equilibrium results were obtained in \cite{Bel03,Yun10}. The large-amplitude theory of \cite{BKLY26} allows substantial pointwise perturbations, but still relies on weighted $L^\infty$ control, a positive lower bound along free transport, matching conserved quantities, and a small relative entropy. A recent whole-space theory near vacuum treats more general collision frequencies by combining weighted smallness with dispersive estimates \cite{PY26}. These regimes are fundamentally different from an existence theory for arbitrary data with finite energy and entropy.

The boundary theory in the large-data regime is more limited. Stationary BGK problems on an interval or slab, including large boundary data and nonequilibrium steady states, were studied in \cite{Uka92,EM21,BY16,BY24}. For evolutionary problems in bounded domains, global perturbative solutions with diffuse reflection were obtained in \cite{CKP26}, while local well-posedness near equilibrium with inflow boundary data was established in \cite{LZ26}. To the best of our knowledge, the available boundary results do not provide a global nonperturbative existence theory for the self-consistent BGK equation with collision frequency $\rho_f$ under only the natural assumptions of finite mass, kinetic energy, and entropy.

Some recent large-data results concern related, but different, BGK-type models. The density dependent equation in \cite{EMY26} relaxes towards a fixed Maxwellian and does not contain the full self-consistent mass--momentum--energy structure of \eqref{eq/ bgk}, while \cite{DJL24} concerns a kinetic relaxation model associated with hyperbolic conservation laws. These results do not address the compactness problem arising from the local Maxwellian in \eqref{eq/ bgk}.

The main new ingredient is the estimate in Section \ref{sec/ comp}, which controls the thermal energy weighted by the collision frequency with a constant independent of the frequency truncation. This estimate propagates a superlinear energy weight chosen from the data and yields compactness of the macroscopic fields even when the density and temperature may degenerate. Entropy production then provides the additional control needed to recover the density factor in the renormalized collision term.

%
%
%
%
%
%
%
%
%
%
%

\subsection{Main results}

We now state our main result. We use the notation $\la v\ra := (1+|v|^2)^{\frac12}$, and adopt the convention $0|\log 0|=0$. The initial datum is assumed to satisfy
\bq\label{eq/ data}
f_0\ge 0,\quad \iint_{\calO} f_0(x,v) (\la v\ra^2+|\log f_0(x,v)| )\,\dx\dv<\infty.
\eq
Thus only finite mass, kinetic energy, and entropy are required. The datum may be arbitrarily large, may contain vacuum, and need not satisfy an $L^p$ bound with $p>1$ or any velocity moment condition above order two.
 
Since the spatial domain $I$ is bounded, an additional spatial moment assumption is not needed. In the inflow problem, the corresponding integrability conditions are imposed on the incoming flux with respect to the natural boundary measure $|v_1|\,\dt\dv$.

We consider the following two boundary conditions.

\begin{enumerate}
\item[\textbf{(In)}]
The inflow boundary condition. Given
\[
g_l:[0,\infty)\times\R^3_+\to[0,\infty), \quad g_r:[0,\infty)\times\R^3_-\to[0,\infty),
\]
where $\R^3_\pm:=\{v\in\R^3:\pm v_1>0\}$, we assume
\bq\label{eq/ in data}
\begin{aligned}
&\iint_{(0,\infty)\times\{v_1>0\}} g_l(t,v)(\la v\ra^2+|\log g_l(t,v)| ) v_1\,\dv\dt<\infty,\\
&\iint_{(0,\infty)\times\{v_1<0\}} g_r(t,v)(\la v\ra^2+|\log g_r(t,v)| ) |v_1|\,\dv\dt<\infty.
\end{aligned}
\eq
The traces
\[
f_l:=f|_{x=-\frac12}, \quad f_r:=f|_{x=\frac12}
\]
are required to satisfy
\bq\label{eq/ inflow}
\begin{cases}
f_l(t,v)=g_l(t,v), & v_1>0,\\
f_r(t,v)=g_r(t,v), & v_1<0.
\end{cases}
\eq

\item[\textbf{(Max)}]
The Maxwell boundary condition. We define the outgoing mass fluxes by
\[
J_l(t):=\int_{v_1<0}f_l(t,v)|v_1|\,\dv, \quad J_r(t):=\int_{v_1>0}f_r(t,v)v_1\,\dv.
\]
Let $\frkM(v)$ be a stationary Maxwellian centered at the origin and normalized by
\bq\label{eq/ M norm}
\int_{v_1>0}\frkM(v)v_1\,\dv = \int_{v_1<0}\frkM(v)|v_1|\,\dv =1.
\eq
For accommodation coefficients $\alpha_l,\alpha_r\in(0,1]$, we impose
\bq\label{eq/ maxwell}
\begin{cases}
f_l(t,v) =(1-\alpha_l)f_l(t,\scrV v) +\alpha_l\frkM(v)J_l(t), & v_1>0,\\[4pt]
f_r(t,v) =(1-\alpha_r)f_r(t,\scrV v) +\alpha_r\frkM(v)J_r(t), & v_1<0,
\end{cases}
\eq
where
\[
\scrV(v_1,v_2,v_3):=(-v_1,v_2,v_3).
\]
\end{enumerate}

\begin{remark}
If one of $\alpha_l,\alpha_r$ vanishes, reflection across the corresponding specular boundary reduces the problem to a doubled slab with positive accommodation at the remaining boundaries. If both vanish, the reflected extension becomes periodic. In fact, even without such a reflection procedure, a purely specular boundary may be handled in a straightforward manner by following for instance \cite{Mis00}. For clarity, the results below are stated for $\alpha_l,\alpha_r>0$.
\end{remark}

Throughout the paper, boundary traces are understood in the sense of the Green formula \eqref{eq/ green} below. They agree with the transport traces in \cite{Mis00,Mis10}; values on the grazing set $\{v_1=0\}$ are immaterial.

To formulate the notion of solution, for a subset $X\subset\R^d$ we denote by $\calD(X)$ the space of smooth compactly supported test functions on $X$, and by $\calD'(X)$ its dual. We also set
\[
\frkB:= \lt\{ \beta\in C^1([0,\infty))\,:\, \beta(0)=0,\, |\beta(z)|\le C_\beta(1+z),\, |\beta'(z)|\le\frac{C_\beta}{1+z} \text{ for some }C_\beta>0 \rt\}.
\]

\begin{definition}
Let $\scrQ$ be a given source term. A nonnegative function $f\in L^\infty((0,\infty);L^1(\calO))$ is called a \emph{renormalized solution} of
\[
\begin{cases}
\p_t f+v_1\p_x f=\scrQ &\text{in }(0,\infty)\times\calO,\\
f|_{t=0}=f_0,
\end{cases}
\]
with boundary condition \textbf{(In)} or \textbf{(Max)} if there exist trace functions $f_l,f_r$ satisfying \eqref{eq/ inflow} or \eqref{eq/ maxwell} almost everywhere, respectively, such that for every $t>0$, $\varphi\in\calD([0,t]\times\ov{\calO})$, and $\beta\in\frkB$,
\bq\label{eq/ green}
\begin{aligned}
&\iint_{\calO}\beta(f(t))\varphi(t)\,\dx\dv -\iint_{\calO}\beta(f_0)\varphi(0)\,\dx\dv\\
&\quad +\int_0^t\int_{\R^3} v_1\beta(f_r(s,v))\varphi(s,\frac12,v)\,\dv\ds  -\int_0^t\int_{\R^3} v_1\beta(f_l(s,v))\varphi(s,-\frac12,v)\,\dv\ds\\
&\quad \quad = \int_0^t\iint_{\calO} \lt[ \beta(f)(\p_s+v_1\p_x)\varphi +\beta'(f)\scrQ\,\varphi \rt]\dx\dv\ds,
\end{aligned}
\eq
and all the terms in \eqref{eq/ green} are well defined.
\end{definition}

If \eqref{eq/ green} holds with $\beta(z)=z$, we simply call $f$ a \emph{weak solution}.

\begin{theorem}\label{thm1}
Assume that $f_0$ satisfies \eqref{eq/ data}. Consider either the inflow condition \textbf{(In)} with \eqref{eq/ in data}, or the Maxwell condition \textbf{(Max)} with $\alpha_l,\alpha_r>0$. Then there exists a nonnegative function
\[
f\in  C([0,\infty);L^1(\calO,\la v \ra^2 \dx\dv)) 
\]
with traces satisfying the prescribed boundary condition almost everywhere and such that for every $\beta\in\frkB$:
\[
\begin{cases}
\p_t\beta(f)+v_1\p_x\beta(f) = \beta'(f)\rho_f(M[f]-f ) &\text{in }\calD'((0,\infty)\times\calO),\\
\beta(f)|_{t=0}=\beta(f_0).
\end{cases}
\]
In particular, $f$ is a global renormalized solution of \eqref{eq/ bgk}. It also has the following properties. 
\begin{enumerate}
\item[\textbf{(In)}] The solution satisfies the Boltzmann entropy inequality \eqref{eq/ ent in}. Moreover, for every $t\ge0$ and $\psi\in\{1,v_1,v_2,v_3,|v|^2\}$,
\[
\iint_{\calO}\psi(v)f(t)\,\dx\dv+\int_0^t\int_{\R^3}\psi(v)v_1(f_r-f_l)\,\dv\ds=\iint_{\calO}\psi(v)f_0\,\dx\dv.
\]
\item[\textbf{(Max)}] The total mass is conserved:
\[
\iint_{\calO}f(t)\,\dx\dv=\iint_{\calO}f_0\,\dx\dv,\quad t\ge0.
\]
The solution satisfies the relative entropy inequality \eqref{eq/ ent max}, including the collision dissipation and the Darroz\`es--Guiraud boundary dissipation.
\end{enumerate}
Both entropy inequalities hold for every $t\ge0$. These properties are proved in Proposition \ref{prop/ bal}.
\end{theorem}

\begin{remark}
For every $\beta\in\frkB$, the renormalized source term satisfies
\[
\beta'(f)\rho_f(M[f]-f) \in  L^1_{\rm loc}([0,\infty)\times \ov{\calO}) .
\]
Indeed, by splitting the domain into the regions $\{M[f]\ge 3f\}$ and $\{M[f]<3f\}$, the integral over the first region is controlled by the kinetic entropy dissipation in \eqref{eq/ diss def}. This yields
\[
\frac{\rho_f|M[f]-f|}{1+f}\in  L^1_{\rm loc}([0,\infty)\times\ov{\calO}),
\]
and the conclusion follows from $(1+f)|\beta'(f)|\le C_\beta$.
\end{remark}

%
%
%
%
%
%
%
%
%
%
%

\subsection{Main ideas of the proof}

The proof has two main compactness difficulties. First, finite kinetic energy does not by itself control the large-velocity tails uniformly along an approximation. This is important since velocity averaging gives compactness only for bounded velocity weights, while the temperature entering the local Maxwellian depends on the full kinetic energy. Second, the physical collision frequency introduces an additional factor of $\rho_f$; thus the gain and loss terms cannot be handled separately by a direct $L^1$ compactness argument. We address these difficulties by combining a compensated estimate in the slab with a tailored de la Vall\'ee Poussin argument and a two-stage renormalization procedure. The last step follows the classical compactness framework for renormalized kinetic equations \cite{DL89B,Lio94}.

In the arguments below, we approximate the physical collision frequency by the bounded rates
\[
\lambda_\e(z):=\frac{z}{1+\e z}.
\]
For each fixed $\e>0$, the boundedness of $\lambda_\e$ allows us to construct global weak solutions of the approximate problem and to justify the intermediate renormalization arguments; see Appendix \ref{app/ approx}. The key point is that the a priori estimates needed to pass to the limit $\e\to0$ are uniform in $\|\lambda_\e\|_{L^\infty}$. In particular, their essential structural input is the bound $0\le \lambda_\e(z)\le z$. 

\subsubsection{Compensated estimates in the slab}

We begin with the one-dimensional balance laws
\[
\p_t\rho+\p_xj_1=0, \quad \p_tj_1+\p_xe_1=0,
\]
where
\[
j_1=\int_{\R^3}v_1f\,\dv, \quad e_1=\int_{\R^3}v_1^2f\,\dv.
\]
Following the Bony-type approach \cite{Bon87,Cer92,BCP06,Wyn23} and the related compensated-integrability viewpoint \cite{Ser18}, we introduce
\[
\scrA(t,x) := \int_{-\frac12}^x\rho(t,y)\,\dy - \lt(x+\frac12\rt)\int_I\rho(t,y)\,\dy, \quad \scrB(t) := \int_I\scrA(t,x)j_1(t,x)\,\dx.
\]
The choice
\[
\scrA(t,-\frac12)=\scrA(t,\frac12)=0
\]
eliminates the momentum boundary terms, and investigating the evolution of $\scrB$ leads to the estimate
\[
\int_0^T\int_I (\rho e_1-j_1^2 ) \,\dx\dt \le C_T.
\]

This normal-direction estimate can be converted into control of the full temperature. Writing $\ov v=(v_2,v_3)$, $\ov u=(u_2,u_3)$, the isotropy of the local Maxwellian gives
\[
\int_{\R^3} |\ov v|^2(M[f]-f )\,\dv = \frac{\rho e_1-j_1^2}{\rho} - \rho\theta.
\]
Testing with $|\ov v|^2$ and using the boundary flux estimates gives, for the truncated collision frequencies
\[
\lambda_\e(z)=\frac{z}{1+\e z},
\]
the uniform bound
\bq\label{eq/ intro therm}
\sup_{\e>0} \int_0^T\int_I \lambda_\e(\rho_{f_\e}) \rho_{f_\e}\theta_{f_\e} \,\dx\dt \le C_T.
\eq

\subsubsection{Energy tails and macroscopic compactness}

In the whole-space theory for constant collision frequency, transport dispersion yields local space-time control of the third velocity moment within the finite-energy framework \cite{Per89}. This mechanism does not directly extend to the present slab, where the transport operator involves only $v_1$: large transverse velocities need not produce rapid motion towards the boundary.

We instead use \eqref{eq/ intro therm} to propagate a convex energy weight with constants independent of the frequency truncation. More precisely, we construct a convex function $\Phi$ satisfying
\[
\frac{\Phi(z)}z\xrightarrow[z\to\infty]{}\infty,\quad 0\le\Phi''(z)\le\frac Cz.
\]
The function is chosen from the tails of the approximate initial data and, in case \textbf{(In)}, from the incoming energy tails.

The curvature bound is compatible with the Gaussian structure of the local Maxwellian. If $X_\e = 1+|u_{f_\e}|^2+3\theta_{f_\e}$, then a direct Gaussian computation gives
\[
\int_{\R^3} (\la v\ra^2-X_\e)^2M[f_\e]\,\dv = \rho_{f_\e} ( 4\theta_{f_\e}|u_{f_\e}|^2 +6\theta_{f_\e}^2  ).
\]
Together with Jensen's inequality, this yields
\[
\int_{\R^3} \Phi(\la v\ra^2) (M[f_\e]-f_\e )\,\dv \le C\rho_{f_\e}\theta_{f_\e}.
\]
The bound \eqref{eq/ intro therm} then allows us to propagate the weighted energy:
\[
\sup_{\e>0} \operatorname*{ess\,sup}_{0<t<T} \iint_{\calO} \Phi(\la v\ra^2)f_\e(t)\,\dx\dv \le C_T,
\]
together with the corresponding outgoing boundary estimates. Since $\frac{\Phi(z)}z\to\infty$, we obtain
\bq\label{eq/ intro tail}
\lim_{R\to\infty} \sup_{\e>0} \operatorname*{ess\,sup}_{0<t<T} \iint_{I\times\{|v|>R\}} \la v\ra^2f_\e(t)\,\dx\dv =0.
\eq

Velocity averaging \cite{GLPS88,DLM91} gives compactness for bounded velocity averages. Combining it with \eqref{eq/ intro tail} yields strong convergence of the full second moment. The density, momentum, and temperature can then be identified, and
\[
M[f_\e] \longrightarrow  M[f] \quad\text{strongly in } L^1( (0,T)\times\calO, \la v\ra^2\,\dt\dx\dv  ).
\]

\subsubsection{Renormalized collision compactness}

The strong convergence of the local Maxwellians is not sufficient to pass directly to the physical collision term due to the additional density factor. We first use
\[
\beta_\delta(z) := \frac{z}{1+\delta z}
\]
and pass to the limit $\e\to0$ for fixed $\delta$. The corresponding renormalized gain and loss terms are treated separately.

The second renormalization is $\gamma(z):=\log(1+z)$. The entropy production
\[
D_\e[f_\e] := \lambda_\e(\rho_{f_\e}) (f_\e-M[f_\e] ) (\log f_\e-\log M[f_\e] )
\]
controls the gain through the elementary splitting
\[
Q_{\e,\delta}^+ \le 2Q_{\e,\delta}^- + \frac{2}{\log2}D_\e[f_\e].
\]
If
\[
D_\e[f_\e]\,\dt\dx\dv \weakstarto \nu = F_\nu\,\dt\dx\dv+\nu_s,
\]
the limiting renormalized collision terms satisfy
\[
0\le\frkQ_\delta^-\le\rho_f, \quad \frkQ_\delta^+ \le 2\frkQ_\delta^- + \frac{2}{\log2}F_\nu.
\]
This provides the local uniform integrability needed to let $\delta\to0$. In particular, the singular part of the limiting entropy-production measure need not be shown to vanish. We obtain
\[
\p_t\gamma(f)+v_1\p_x\gamma(f) = \gamma'(f)\rho_f(M[f]-f).
\]

\subsubsection{Boundary traces}

The interior compactness argument is common to the two boundary conditions, but their trace limits are different. For \textbf{(In)}, the outgoing energy and entropy estimates give weak $L^1$ compactness of the traces directly. For \textbf{(Max)}, the corresponding compactness follows from the Darroz\`es--Guiraud information through a biting argument, after which the incoming traces are recovered from the Maxwell boundary law; see \cite{Mis10}.

Once $\gamma(f)$ and its traces are identified, the remaining renormalizations follow by applying the transport chain rule \cite{DL89O,Mis10} to
\[
\Gamma(s)=\beta(e^s-1), \quad \beta\in\frkB.
\]
The same trace theory, together with the uniform energy and entropy bounds, gives
\[
f\in C([0,T];L^1(\calO)),
\]
and then interpolating with the energy tail estimate in \eqref{eq/ intro tail}, we get
\begin{equation*}
    f\in C([0,T];L^1(\calO,\la v \ra^2 \dx\dv)).
\end{equation*}

The rest of this paper is organized as follows. In Section \ref{sec/ comp}, we establish the compensated estimate and the thermal bound weighted by the collision frequency. Section \ref{sec/ tails} develops the uniform estimates on the kinetic energy tails. The interior limit and the inflow boundary argument are presented in Section \ref{sec/ inflow limit}, while the Maxwell boundary traces are treated in Section \ref{sec/ maxwell limit}. In Appendix \ref{app/ approx}, we construct solutions for each truncated collision frequency. Finally, Appendix \ref{app/ mild} derives the mild formulation along the last free flight segment.

%
%
%
%
%
%
%
%
%
%
%

\section{Compensated estimates in the slab}\label{sec/ comp}

We begin with the auxiliary BGK equation
\bq\label{eq/ lambda}
\begin{cases}
\p_t f+v_1\p_x f =\lambda(\rho_f)(M[f]-f ) &\text{in }(0,\infty)\times\calO,\\
f|_{t=0}=f_0.
\end{cases}
\eq
On the boundary, we impose either \textbf{(In)} or \textbf{(Max)}. Throughout this section, $\lambda:[0,\infty)\to[0,\infty)$ is assumed to be bounded and measurable unless otherwise stated. The boundedness of $\lambda$ is used to guarantee the weighted integrability of the collision term required in the weak formulation and in the velocity truncation arguments below. The additional structural condition $0\le\lambda(z)\le z$ will only be imposed when deriving the density weighted thermal estimate in Section \ref{subsec/ therm}.

This distinction is useful for the approximation introduced later. Indeed, the truncated collision frequencies
\[
\lambda_\e(z):=\frac{z}{1+\e z}
\]
are bounded for every fixed $\e>0$ and satisfy
\[
0\le\lambda_\e(z)\le z, \quad \|\lambda_\e\|_{L^\infty}=\frac1\e.
\]
The boundary and compensated estimates established below are independent of $\|\lambda\|_{L^\infty}$. This uniformity is the essential point: Proposition \ref{prop/ rho2} below controls $\lambda_\e(\rho_{f_\e})\rho_{f_\e}\theta_{f_\e}$ even as $\|\lambda_\e\|_{L^\infty}\to\infty$, and supplies the estimate used to propagate the energy weight in Section \ref{sec/ tails}.

For brevity, we write
\[
\rho=\rho_f,\quad u=u_f,\quad \theta=\theta_f,\quad j=\rho u=(j_1,j_2,j_3),
\]
and set
\[
e_1:=\int_{\R^3}v_1^2f\,\dv.
\]

We shall repeatedly use the following consequence of the boundedness of $\lambda$:
\bq\label{eq/ Q wt}
\int_0^T\iint_{\calO} \la v\ra^2 \bigl|\lambda(\rho)(M[f]-f)\bigr| \,\dx\dv\dt \le 2\|\lambda\|_{L^\infty} \int_0^T\iint_{\calO} \la v\ra^2f\,\dx\dv\dt <\infty.
\eq
Here we used the moment identities of the local Maxwellian.

%
%
%
%
%
%
%
%
%
%
%

\subsection{Boundary fluxes and macroscopic balance laws}

We first record the boundary estimates needed throughout this section.

\begin{lemma}\label{lem/ j bounds}
Let $f\ge0$ be a weak solution of \eqref{eq/ lambda} on $[0,T]$, satisfying either \textbf{(In)} or \textbf{(Max)}, and assume
\[
f\in C([0,T];L^1(\calO)) \cap L^\infty(0,T; L^1(\calO,\la v\ra^2\,\dx\dv) )
\]
In case \textbf{(Max)}, set
\[
K(t):=\iint_{\calO}xv_1f(t,x,v)\,\dx\dv, \quad c_{\frkM,1}:= \int_{v_1<0}v_1^2\frkM(v)\,\dv>0.
\]
Then, for every $t\in[0,T]$,
\bq\label{eq/ J est}
\frac{c_{\frkM,1}}2 \int_0^t(\alpha_lJ_l+\alpha_rJ_r )\,\ds \le K(0)-K(t) +\int_0^t\iint_{\calO}v_1^2f\,\dx\dv\ds.
\eq
For either boundary condition, the traces satisfy
\bq\label{eq/ tr wt}
\int_0^T\int_{\R^3} \la v\ra^2(f_l+f_r)|v_1|\,\dv\dt<\infty.
\eq
More precisely,
\bq\label{eq/ bdy E}
\alpha_l \int_0^T\int_{v_1<0} \la v\ra^2 f_l|v_1|\,\dv\dt +\alpha_r \int_0^T\int_{v_1>0} \la v\ra^2 f_rv_1\,\dv\dt \le C_T,
\eq
where $\alpha_l=\alpha_r=1$ in case \textbf{(In)}. The constant $C_T$ depends only on the bulk second-moment bound and the prescribed boundary data; in the Maxwell case it may also depend on $\alpha_l,\alpha_r$ and $\frkM$. It is independent of $\|\lambda\|_{L^\infty}$.
\end{lemma}

\begin{proof}
We use a standard multiplier argument; see \cite{AM94,Mis00,Mis00T,Mis10}. Let $0\le\chi_R\le1$ be a family of smooth radial velocity cutoffs such that $\chi_R\uparrow1$ as $R\to\infty$. The bound \eqref{eq/ Q wt} justifies the passage to the limit in all volume terms arising from the velocity truncation.

Consider first \textbf{(Max)}. Testing the Green formula with $xv_1\chi_R(v)$ and passing to the limit $R\to\infty$, using the momentum cancellation in \eqref{eq/ M mom}, gives, initially for almost every $t\in(0,T)$,
\bq\label{eq/ xv1}
K(t)-K(0) +\frac12\int_0^t\int_{\R^3} v_1^2(f_l+f_r)\,\dv\ds = \int_0^t\iint_{\calO}v_1^2f\,\dx\dv\ds.
\eq
We also note that $K\in C([0,T])$. Indeed, for every $s,t\in[0,T]$ and $R>0$,
\begin{equation} \label{eq/ v1 cont}
\begin{aligned}
|K(t)-K(s)| &\le \frac12\iint_{\calO}|v_1|\,|f(t)-f(s)|\,\dx\dv \\
&\le \frac{R}{2}\|f(t)-f(s)\|_{L^1(\calO)} +\frac1R\sup_{0\le\tau\le T} \iint_{\calO}|v|^2f(\tau)\,\dx\dv.
\end{aligned}
\end{equation}
Fixing $R\gg 1$ large and then using the continuity of $f$ in $L^1(\calO)$ yields the claimed continuity. Since the remaining terms in \eqref{eq/ xv1} are absolutely continuous in time, the identity extends to every $t\in[0,T]$.

The Maxwell boundary condition yields
\begin{align*}
\frac{c_{\frkM,1}}2 \int_0^t(\alpha_lJ_l+\alpha_rJ_r)\,\ds ={}& K(0)-K(t) +\int_0^t\iint_{\calO}v_1^2f\,\dx\dv\ds\\
&- \lt(1-\frac{\alpha_l}{2}\rt) \int_0^t\int_{v_1<0}v_1^2f_l\,\dv\ds - \lt(1-\frac{\alpha_r}{2}\rt) \int_0^t\int_{v_1>0}v_1^2f_r\,\dv\ds.
\end{align*}
Discarding the last two nonpositive terms proves \eqref{eq/ J est}.

We next test with $\la v\ra^2\chi_R(v)$. In case \textbf{(In)}, the collision contribution vanishes in the limit due to \eqref{eq/ M mom}, and the Green formula gives
\begin{align*}
&\iint_{\calO}\la v\ra^2f(t)\,\dx\dv +\int_0^t\int_{v_1<0} \la v\ra^2f_l|v_1|\,\dv\ds +\int_0^t\int_{v_1>0} \la v\ra^2f_rv_1\,\dv\ds\\
&\quad = \iint_{\calO}\la v\ra^2f_0\,\dx\dv +\int_0^t\int_{v_1>0} \la v\ra^2g_lv_1\,\dv\ds +\int_0^t\int_{v_1<0} \la v\ra^2g_r|v_1|\,\dv\ds.
\end{align*}
The right-hand side is finite by \eqref{eq/ in data}.

For \textbf{(Max)}, the specular part preserves the radial weight $\la v\ra^2$. Set
\[
c_{\frkM,2} := \int_{v_1>0}\la v\ra^2\frkM(v)v_1\,\dv.
\]
The same velocity-truncation argument gives
\begin{align*}
&\iint_{\calO}\la v\ra^2f(t)\,\dx\dv +\alpha_l\int_0^t\int_{v_1<0} \la v\ra^2f_l|v_1|\,\dv\ds +\alpha_r\int_0^t\int_{v_1>0} \la v\ra^2f_rv_1\,\dv\ds\\
&\quad = \iint_{\calO}\la v\ra^2f_0\,\dx\dv +c_{\frkM,2} \int_0^t(\alpha_lJ_l+\alpha_rJ_r)\,\ds.
\end{align*}
Together with \eqref{eq/ J est}, this proves \eqref{eq/ bdy E}. Since $\alpha_l,\alpha_r>0$, the Maxwell boundary condition also controls the corresponding incoming weighted fluxes through the reflected outgoing part and the Maxwellian contribution. This gives \eqref{eq/ tr wt}. The inflow case follows directly from the above energy balance and \eqref{eq/ in data}. This completes the proof.
\end{proof}

We next record the macroscopic balance laws. They follow from the collision invariants, while Lemma \ref{lem/ j bounds} supplies the trace integrability needed to justify the boundary terms.

\begin{lemma}
Under the assumptions of Lemma \ref{lem/ j bounds}, the macroscopic moments satisfy
\bq\label{eq/ mom}
\p_t\rho+\p_xj_1=0, \quad \p_tj_1+\p_xe_1=0 \quad\text{in }\calD'((0,T)\times I).
\eq
Moreover, define
\[
j_{f_r,1}(t):=\int_{\R^3}v_1f_r(t,v)\,\dv, \quad j_{f_l,1}(t):=\int_{\R^3}v_1f_l(t,v)\,\dv.
\]
Then, for every $\varphi\in\calD((0,T)\times\ov I)$,
\bq\label{eq/ mass bdy}
\begin{aligned}
&\int_0^T\int_I (\rho\p_t\varphi+j_1\p_x\varphi )\,\dx\dt =
\begin{cases}
\displaystyle \int_0^T \Big[ \varphi(t,\frac12)j_{f_r,1}(t) -\varphi(t,-\frac12)j_{f_l,1}(t) \Big]\,\dt, &\text{in case \textbf{(In)}},\\[8pt]
0, &\text{in case \textbf{(Max)}}.
\end{cases}
\end{aligned}
\eq
\end{lemma}
\begin{proof}
We first derive the continuity equation together with its boundary contribution. Let $\chi_R\in C_c^\infty(\R^3)$ be a radial cutoff such that $\one_{\{|v|\le R\}}\le\chi_R\le\one_{\{|v|\le2R\}}$. For $\varphi\in\calD((0,T)\times\ov I)$, use $\varphi_R(t,x,v):=\varphi(t,x)\chi_R(v)$ as a test function in the Green formula for \eqref{eq/ lambda}. Since $\varphi$ is compactly supported in time, the initial and terminal terms vanish, and we obtain
\bq\label{eq/ chir}
\begin{aligned}
&\int_0^T \int_{\R^3} v_1f_r(t,v)\varphi(t,\frac12)\chi_R(v)\,\dv\dt - \int_0^T \int_{\R^3} v_1f_l(t,v)\varphi(t,-\frac12)\chi_R(v)\,\dv\dt\\
&= \int_0^T \iint_{\calO} f(\p_t\varphi+v_1\p_x\varphi)\chi_R \,\dx\dv\dt + \int_0^T \iint_{\calO} \lambda(\rho)(M[f]-f)\varphi\chi_R \,\dx\dv\dt.
\end{aligned}
\eq

We next pass to the limit $R\to\infty$. The boundary terms are well defined since \eqref{eq/ tr wt} implies
\[
\int_0^T|j_{f_r,1}(t)|\,\dt \le \int_0^T\int_{\R^3}|v_1|f_r\,\dv\dt<\infty \quad \text{and} \quad \int_0^T|j_{f_l,1}(t)|\,\dt \le \int_0^T\int_{\R^3}|v_1|f_l\,\dv\dt<\infty.
\]
Hence the boundary integrals in \eqref{eq/ chir} converge by the dominated convergence theorem.

For the bulk transport term, we use $|v_1|\le\la v\ra$ together with the finite second moment of $f$. Thus
\[
\int_0^T\iint_{\calO} f(\p_t\varphi+v_1\p_x\varphi)\chi_R \,\dx\dv\dt \longrightarrow  \int_0^T\int_I (\rho\p_t\varphi+j_1\p_x\varphi ) \,\dx\dt.
\]

It remains to consider the collision term. The bound \eqref{eq/ Q wt} provides an integrable majorant. Hence the dominated convergence theorem, together with the mass cancellation in \eqref{eq/ M mom}, gives
\[
\int_0^T\iint_{\calO} \lambda(\rho)(M[f]-f)\varphi\chi_R \,\dx\dv\dt \longrightarrow  \int_0^T\int_I \lambda(\rho)\varphi \lt( \int_{\R^3}(M[f]-f)\,\dv \rt)\dx\dt =0.
\]
Passing to the limit in \eqref{eq/ chir}, we obtain
\bq\label{eq/ mass green}
\int_0^T \varphi(t,\frac12)j_{f_r,1}(t)\,\dt - \int_0^T \varphi(t,-\frac12)j_{f_l,1}(t)\,\dt = \int_0^T\int_I (\rho\p_t\varphi+j_1\p_x\varphi ) \,\dx\dt.
\eq
In particular, taking $\varphi\in\calD((0,T)\times I)$ gives
\[
\p_t\rho+\p_xj_1=0 \quad\text{in }\calD'((0,T)\times I).
\]
Equation \eqref{eq/ mass green} also proves \eqref{eq/ mass bdy} in case \textbf{(In)}.

For the Maxwell boundary condition, the net mass flux vanishes at each endpoint. We verify this at $x=\frac12$. For almost every $t\in(0,T)$,
\[
\begin{aligned}
j_{f_r,1}(t) &= \int_{v_1>0}v_1f_r(t,v)\,\dv + \int_{v_1<0}v_1f_r(t,v)\,\dv\\
&= J_r(t) + \int_{v_1<0} v_1 \lt[ (1-\alpha_r)f_r(t,\scrV v) +\alpha_r\frkM(v)J_r(t) \rt] \dv.
\end{aligned}
\]
Under the change of variables $v\mapsto\scrV v$,
\[
\int_{v_1<0} v_1f_r(t,\scrV v)\,\dv = -\int_{v_1>0}v_1f_r(t,v)\,\dv = -J_r(t),
\]
while the normalization \eqref{eq/ M norm} gives
\[
\int_{v_1<0}v_1\frkM(v)\,\dv=-1.
\]
Consequently,
\[
j_{f_r,1}(t) = J_r(t) -(1-\alpha_r)J_r(t) -\alpha_rJ_r(t) =0.
\]
The same computation at $x=-\frac12$ gives $j_{f_l,1}(t)=0$. Substituting these identities into \eqref{eq/ mass green} proves \eqref{eq/ mass bdy} for \textbf{(Max)}.

We finally derive the normal momentum equation. Let now $\varphi\in\calD((0,T)\times I)$ and use $\psi_R(t,x,v) := \varphi(t,x)v_1\chi_R(v)$ as a test function. Since $\varphi$ is compactly supported in the interior of $I$, the boundary terms vanish. The Green formula yields
\[
0 = \int_0^T\iint_{\calO} f \Big( v_1\p_t\varphi +v_1^2\p_x\varphi \Big) \chi_R \,\dx\dv\dt + \int_0^T\iint_{\calO} \lambda(\rho)(M[f]-f) v_1\varphi\chi_R \,\dx\dv\dt.
\]
The first integral converges as $R\to\infty$ to
\[
\int_0^T\int_I ( j_1\p_t\varphi +e_1\p_x\varphi  )\,\dx\dt,
\]
since the second velocity moment is integrable.

For the collision term, \eqref{eq/ Q wt} provides an integrable majorant, while the momentum cancellation in \eqref{eq/ M mom} gives
\[
\int_{\R^3}v_1(M[f]-f)\,\dv=0.
\]
Hence
\[
\int_0^T \iint_{\calO} \lambda(\rho)(M[f]-f) v_1\varphi\chi_R \,\dx\dv\dt \longrightarrow  0,
\]
and this yields
\[
\int_0^T\int_I ( j_1\p_t\varphi +e_1\p_x\varphi  )\,\dx\dt=0
\]
for every $\varphi\in\calD((0,T)\times I)$. This is equivalent to
\[
\p_tj_1+\p_xe_1=0 \quad\text{in }\calD'((0,T)\times I).
\]
This completes the proof.
\end{proof}

%
%
%
%
%
%
%
%
%
%
%

\subsection{A Bony estimate in the slab}

The one-dimensional balance laws in \eqref{eq/ mom} allow us to recover a compensated estimate of Bony type \cite{Bon87,Cer92,BCP06,Wyn23}; see also the compensated-integrability viewpoint in \cite{Ser18}. In the presence of boundaries, it is useful to choose a primitive of the density which vanishes at both endpoints. We set
\bq\label{eq/ scra}
\scrA(t,x) := \int_{-\frac12}^x\rho(t,y)\,\dy - \lt(x+\frac12\rt) \int_I\rho(t,y)\,\dy.
\eq
In particular,
\[
\scrA(t,-\frac12)=\scrA(t,\frac12)=0.
\]
The associated Bony functional is defined by
\[
\scrB(t) := \int_I\scrA(t,x)j_1(t,x)\,\dx.
\]
The vanishing of $\scrA$ at the endpoints eliminates the momentum boundary terms arising from the integration by parts.

\begin{proposition}\label{prop/ bony}
Under the assumptions of Lemma \ref{lem/ j bounds}, there exists $C_T>0$, independent of $\|\lambda\|_{L^\infty}$, such that
\bq\label{eq/ useful}
\int_0^T\int_I (\rho e_1-j_1^2 )\,\dx\dt \le C_T.
\eq
More precisely, for every $t\in[0,T]$, we have
\bq\label{eq/ B prelim}
\int_0^t\int_I (\rho e_1-j_1^2 )\,\dx\ds \le \scrB(t)-\scrB(0)+C_T.
\eq
\end{proposition}

\begin{proof}
We regularize only in time. Choose $0\le\zeta\in C_c^\infty((-2,-1))$, $\int_{\R}\zeta(s)\,\ds=1$, and, for $0<\e<\min\{1,\frac T2\}$, set $\zeta_\e(t):= \frac1\e\zeta\lt(\frac{t}{\e}\rt)$. For $t \in J_\e:=(-\e,T-2\e)$, define the one-sided temporal convolutions
\[
\rho_\e(t) :=\int_0^T\rho(s)\zeta_\e(t-s)\,\ds,
\]
and analogously $j_\e$ and $e_{1\e}$. Since $\operatorname{supp}\zeta_\e(t-\cdot) \subset(t+\e,t+2\e) \subset(0,T)$ for $t\in J_\e$, the balance laws \eqref{eq/ mom} give
\bq\label{eq/ mol cont}
\forall t\in J_\e, \quad \p_t\rho_\e(t) +\p_xj_{\e1}(t)=0, \quad \p_tj_{\e1}(t)+\p_xe_{1\e}(t)=0 \quad\text{in }\calD'(I).
\eq
The temporal smoothing and \eqref{eq/ mol cont} then imply $j_{\e1},e_{1\e} \in C^\infty(J_\e;W^{1,1}(I))$. 

We then define
\[
\scrA_\e(t,x) := \int_{-\frac12}^x\rho_\e(t,y)\,\dy - \lt(x+\frac12\rt) \int_I\rho_\e(t,y)\,\dy \quad \text{and} \quad \scrB_\e(t) := \int_I\scrA_\e(t,x)j_{\e1}(t,x)\,\dx,
\]
and note
\bq\label{eq/ A bdy}
\scrA_\e(t,-\frac12) = \scrA_\e(t,\frac12) =0.
\eq

\medskip \noindent $\bullet$ (The Maxwell boundary condition) We first consider \textbf{(Max)}. For fixed $t\in J_\e$, use $\varphi(s,x)=\zeta_\e(t-s)\psi(x)$ in \eqref{eq/ mass bdy}, where $\psi\in C^\infty(\ov I)$ may be nonzero at either endpoint. Differentiating the temporal convolution gives
\[
\int_I j_{\e1}(t,x)\p_x\psi(x)\,\dx=\int_I\p_t\rho_\e(t,x)\psi(x)\,\dx.
\]
On the other hand, because $j_{\e1}(t)\in W^{1,1}(I)$, the following computations are justified:
\[
\int_I j_{\e1}(t,x)\p_x\psi(x)\,\dx =\int_I\p_t\rho_\e(t,x)\psi(x)\,\dx +j_{\e1}(t,\frac12)\psi(\frac12)-j_{\e1}(t,-\frac12)\psi(-\frac12),
\]
where we used also \eqref{eq/ mol cont}. Comparing the two identities and choosing the endpoint values of $\psi$ independently, we obtain
\[
j_{\e1}(t,-\frac12)=j_{\e1}(t,\frac12)=0.
\]
Consequently,
\[
\p_t\scrA_\e(t,x) = -j_{\e1}(t,x).
\]
Moreover,
\[
\p_x\scrA_\e(t,x) = \rho_\e(t,x)-m_\e(t), \quad m_\e(t):=\int_I\rho_\e(t,y)\,\dy.
\]
Since the Maxwell boundary condition preserves mass, $m_\e(t)=m(0)$ on $J_\e$. Using \eqref{eq/ mol cont} and \eqref{eq/ A bdy}, we obtain
\[
\ddt\scrB_\e(t) = -\int_Ij_{\e1}^2\,\dx -\int_I\scrA_\e\,\p_xe_{1\e}\,\dx = -\int_Ij_{\e1}^2\,\dx +\int_I (\rho_\e-m(0) )e_{1\e}\,\dx.
\]
Hence, for $0\le t<T-2\e$, we have
\bq\label{eq/ B reg}
\int_0^t\int_I (\rho_\e e_{1\e}-j_{\e1}^2 ) \,\dx\ds \le m(0) \int_0^t\int_Ie_{1\e}\,\dx\ds +\scrB_\e(t)-\scrB_\e(0).
\eq

We next pass to the limit $\e\to0$. By the same argument in \eqref{eq/ v1 cont}, we first observe that the assumptions of Lemma \ref{lem/ j bounds} on $f$ imply
\[
\rho,j\in C([0,T];L^1(I)).
\]
Standard properties of one-sided mollification then yield, for every $T'<T$,
\[
\rho_\e\to\rho,\quad j_\e\to j \quad\text{in }C([0,T'];L^1(I)), \quad \text{and} \quad e_{1\e}\to e_1 \quad\text{in }L^1((0,T')\times I).
\]
Consequently,
\[
\scrA_\e\to\scrA \quad\text{uniformly on }[0,T']\times\ov I, \quad \text{and} \quad \scrB_\e\to\scrB \quad\text{in }C([0,T']).
\]

For a fixed $t<T$, we may extract a subsequence such that $\rho_\e$, $j_{\e1}$, and $e_{1\e}$ converge almost everywhere on $(0,t)\times I$. Since
\[
\rho_\e e_{1\e}-j_{\e1}^2 = \rho_\e \int_{\R^3} (v_1-u_{\e1})^2f_\e\,\dv \ge0,
\]
where
\[
f_\e:=f*_t\zeta_\e, \quad u_{\e1}:=\frac{j_{\e1}}{\rho_\e}
\]
with the usual convention on $\{\rho_\e=0\}$, Fatou's lemma and \eqref{eq/ B reg} give
\[
\int_0^t\int_I (\rho e_1-j_1^2)\,\dx\ds \le m(0)\int_0^t\int_Ie_1\,\dx\ds +\scrB(t)-\scrB(0).
\]
This proves \eqref{eq/ B prelim} for \textbf{(Max)}.

\medskip 
\noindent $\bullet$ (The inflow boundary condition) For \textbf{(In)}, the mass in the slab is not constant. We still have $\p_x\scrA_\e = \rho_\e-m_\e(t)$, while
\[
\p_t\scrA_\e(t,x) = -j_{\e1}(t,x) +j_{\e1}(t,-\frac12) - \lt(x+\frac12\rt) \lt( -j_{\e1}(t,\frac12) +j_{\e1}(t,-\frac12)\rt).
\]
Using these identities in the derivative of $\scrB_\e$ gives
\bq\label{eq/ B in}
\begin{aligned}
\int_0^t\int_I (\rho_\e e_{1\e}-j_{\e1}^2 )\,\dx\ds &= \scrB_\e(t)-\scrB_\e(0) - \int_0^t j_{\e1}(s,-\frac12) \lt(\int_Ij_{\e1}(s,x)\,\dx\rt) \ds\\
&\quad- \int_0^t \lt( j_{\e1}(s,\frac12)-j_{\e1}(s,-\frac12)  \rt) \lt( \int_I \lt(x+\frac12\rt)j_{\e1}(s,x)\,\dx \rt)\ds\\
&\quad+ \int_0^t m_\e(s) \lt(\int_Ie_{1\e}(s,x)\,\dx\rt) \ds.
\end{aligned}
\eq

It remains to control the boundary values of $j_{\e1}$. Convolving the Green formula in time and using the absolutely continuous representative of $j_{\e1}(t)$, we have
\[
j_{\e1}(t,-\frac12) = \int_{\R^3}v_1f_{l\e}(t,v)\,\dv, \quad j_{\e1}(t,\frac12) = \int_{\R^3}v_1f_{r\e}(t,v)\,\dv,
\]
where $f_{l\e}:=f_l*_t\zeta_\e$, $f_{r\e}:=f_r*_t\zeta_\e$. Since convolution with $\zeta_\e$ is $L^1$-contractive in time, \eqref{eq/ tr wt} yields
\[
\int_0^{T-2\e} \lt( |j_{\e1}(s,-\frac12)| +|j_{\e1}(s,\frac12)|  \rt) \ds \le C_T.
\]
The bulk moment bounds also give
\[
\sup_{0<\e<\min\{1,\frac T2\}} \lt( \|m_\e\|_{L^\infty(J_\e)} + \|j_{\e1}\|_{L^\infty(J_\e;L^1(I))} + \|e_{1\e}\|_{L^\infty(J_\e;L^1(I))} \rt) \le C_T.
\]
Applying these estimates to \eqref{eq/ B in} yields
\[
\int_0^t\int_I (\rho_\e e_{1\e}-j_{\e1}^2 )\,\dx\ds \le \scrB_\e(t)-\scrB_\e(0)+C_T.
\]
The convergence argument used in the Maxwell case now gives \eqref{eq/ B prelim} for \textbf{(In)}.

Finally, from \eqref{eq/ scra}, $|\scrA(t,x)| \le \|\rho(t)\|_{L^1(I)}$, and thus
\[
|\scrB(t)| \le \|\rho(t)\|_{L^1(I)} \|j_1(t)\|_{L^1(I)} \le C_T.
\]
Letting $t\uparrow T$ in \eqref{eq/ B prelim}, using the continuity of $\scrB$ and the monotone convergence of the nonnegative integral on the left-hand side, proves \eqref{eq/ useful}. This completes the proof.
\end{proof}

%
%
%
%
%
%
%
%
%
%
%

\subsection{Control of the density weighted thermal energy}\label{subsec/ therm}

The Bony estimate obtained above can be converted into control of the density weighted temperature. The independence of the estimate from $\|\lambda\|_{L^\infty}$ will allow us to recover a space-time bound on $\rho_f^2\theta_f$ in the limit of the truncated collision frequencies.

\begin{proposition}\label{prop/ rho2}
Let $f$ satisfy the assumptions of Lemma \ref{lem/ j bounds}, and suppose in addition that
\[
0\le\lambda(z)\le z \quad\text{for all }z\ge0.
\]
Then, for every $T>0$, we have
\bq\label{eq/ therm}
\int_0^T\int_I \lambda(\rho)\rho\theta\,\dx\dt \le C_T,
\eq
where $C_T$ is independent of $\|\lambda\|_{L^\infty}$.
\end{proposition}

\begin{proof}
Write $\ov v:=(v_2,v_3)$, $\ov u:=(u_2,u_3)$, and introduce the transverse pressure
\[
\ov P := \int_{\R^3} |\ov v-\ov u|^2f\,\dv.
\]
The definition of the temperature gives
\bq\label{eq/ tr press}
3\rho\theta = \int_{\R^3}|v-u|^2f\,\dv = \ov P + \int_{\R^3}(v_1-u_1)^2f\,\dv = \ov P + \frac{\rho e_1-j_1^2}{\rho}.
\eq
All quotients are interpreted as zero on the vacuum set.

Since the local Maxwellian is isotropic around $u$, we also have
\[
\int_{\R^3} |\ov v|^2(M[f]-f )\,\dv = 2\rho\theta-\ov P.
\]
Combining this identity with \eqref{eq/ tr press} yields
\bq\label{eq/ tr Q}
\int_{\R^3} |\ov v|^2(M[f]-f )\,\dv = \frac{\rho e_1-j_1^2}{\rho} -\rho\theta.
\eq

Set
\[
\ov E(t) := \iint_{\calO}|\ov v|^2f(t,x,v)\,\dx\dv \quad \text{and} \quad \ov F(t) := \int_{\R^3} |\ov v|^2v_1 (f_r(t,v)-f_l(t,v) )\,\dv.
\]
We remark that \eqref{eq/ tr wt} gives
\bq\label{eq/ tr flux}
\int_0^T|\ov F(t)|\,\dt \le C_T,
\eq
with a constant independent of $\|\lambda\|_{L^\infty}$; moreover, the bound \eqref{eq/ Q wt} shows that
\begin{equation*}
    \int_0^T \iint_{\calO} \lambda(\rho) |\ov v|^2 |M[f] - f| \,\dx\dv\dt \le \|\lambda\|_{L^\infty} C_T.
\end{equation*}
These show that, after an approximation procedure, $|\ov v|^2$ is admissible as a test function in the Green formula of \eqref{eq/ lambda}. In particular, it is deduced that $\ov E(t)$ is absolutely continuous on $[0,T]$ with the identity

\bq\label{eq/ tr E}
\begin{aligned}
\ddt\ov E(t) &= - \ov F(t) + \iint_{\calO} \lambda(\rho) |\ov v|^2 (M[f] - f) \,\dx\dv  \\
&= - \ov F(t) + \int_I \lambda(\rho) \lt( \frac{\rho e_1-j_1^2}{\rho} -\rho\theta \rt) \dx,
\end{aligned}
\eq 
where the last equality is due to \eqref{eq/ tr Q}. Consequently, we find
\bq\label{eq/ lambda rho}
\int_0^T\int_I \lambda(\rho)\rho\theta\,\dx\dt = \int_0^T\int_I \frac{\lambda(\rho)}{\rho} (\rho e_1-j_1^2 )\,\dx\dt + \ov E(0)-\ov E(T) -\int_0^T\ov F(t)\,\dt.
\eq
Since $\lambda(\rho)\le\rho$ and $\rho e_1-j_1^2\ge0$,
\[
\frac{\lambda(\rho)}{\rho} (\rho e_1-j_1^2 ) \le \rho e_1-j_1^2.
\]
Using also $\ov E(T)\ge0$, \eqref{eq/ tr flux}, and Proposition \ref{prop/ bony}, we conclude from \eqref{eq/ lambda rho} that
\[
\int_0^T\int_I \lambda(\rho)\rho\theta\,\dx\dt \le \int_0^T\int_I (\rho e_1-j_1^2 )\,\dx\dt +C_T \le C_T.
\]
This proves \eqref{eq/ therm} and completes the proof.
\end{proof}

%
%
%
%
%
%
%
%
%
%
%

\section{Uniform integrability of the kinetic energy}\label{sec/ tails}

This section establishes uniform integrability in velocity of the bulk kinetic energy and the outgoing energy flux. We first collect the basic estimates for the approximate solutions and then combine the thermal bound of Section \ref{sec/ comp} with a convex weight chosen from the data.

\subsection{Approximate solutions and uniform estimates}\label{sec/ approximate}

For $\e\in(0,1)$ we set
\[
\lambda_\e(z):=\frac{z}{1+\e z}
\]
and consider nonnegative weak solutions $f_\e$ of
\bq\label{eq/ regbgk}
\begin{cases}
\p_t f_\e+v_1\p_x f_\e = \lambda_\e(\rho_{f_\e}) (M[f_\e]-f_\e ) &\text{in }(0,\infty)\times\calO,\\
f_\e|_{t=0}=f_{\e0},
\end{cases}
\eq
with either \textbf{(In)} or \textbf{(Max)} on the boundary.

Throughout this section we work with entropy-admissible weak solutions $f_\e$ satisfying, for every $T>0$,
\[
f_\e\in C([0,T];L^1(\calO)), \quad \la v\ra^2f_\e\in L^\infty(0,T;L^1(\calO)).
\]
By entropy-admissible we mean that these solutions are assumed to satisfy the entropy inequalities \eqref{eq/ H in} and \eqref{eq/ H max} used in Lemma \ref{lem/ every} below. We refer to Proposition \ref{prop/ approx} for the construction of such entropy-admissible solutions. We point out that the existence argument there uses constants depending on $\e$, but the estimates below are uniform as $\e\to0$.

We fix a sequence $\e_n\downarrow0$ and, to simplify the notation, continue to denote its elements by $\e$. All suprema indexed by $\e>0$ below are taken over this fixed sequence in $(0,1)$. The approximate initial data are chosen such that
\bq\label{eq/ init E}
f_{\e0}\longrightarrow  f_0 \quad\text{strongly in } L^1(\calO,\la v\ra^2\,\dx\dv),
\eq
and
\[
\sup_{\e>0} \iint_{\calO}f_{\e0}|\log f_{\e0}|\,\dx\dv<\infty, \quad \iint_{\calO}f_{\e0}\log f_{\e0}\,\dx\dv \longrightarrow  \iint_{\calO}f_0\log f_0\,\dx\dv.
\]
In particular,
\bq\label{eq/ init bdd}
\sup_{\e>0} \iint_{\calO}\la v\ra^2f_{\e0}\,\dx\dv<\infty.
\eq
After passing to a subsequence, we also assume
\[
f_{\e0}\longrightarrow  f_0 \quad\text{a.e. in }\calO.
\]

We first derive the basic energy and entropy estimates for \eqref{eq/ regbgk}.

\begin{lemma}\label{lem/ every}
Let $f_\e$ be the entropy-admissible solution of \eqref{eq/ regbgk} corresponding to $f_{\e0}$. Then, for every $T>0$, there exists $C_T>0$, independent of $\e$, such that the following estimates hold.
\begin{enumerate}

\item[(i)] The bulk energy and entropy satisfy
\bq\label{eq/ bdd}
\sup_{\e>0}\, \operatorname*{ess\,sup}_{0<t<T} \iint_{\calO} f_\e(\la v\ra^2+|\log f_\e| ) \,\dx\dv \le C_T.
\eq

\item[(ii)] The entropy production satisfies
\bq\label{eq/ reg diss}
\begin{aligned}
\sup_{\e>0}\int_0^T\iint_{\calO} &\lambda_\e(\rho_{f_\e}) (f_\e-M[f_\e] ) (\log f_\e-\log M[f_\e] ) \,\dx\dv\dt \le C_T.
\end{aligned}
\eq

\item[(iii)] In case \textbf{(In)}, the outgoing traces satisfy
\bq\label{eq/ out bdd}
\begin{aligned}
&\sup_{\e>0} \int_0^T\int_{v_1<0} f_{\e l} (\la v\ra^2+|\log f_{\e l}| ) |v_1|\,\dv\dt \le C_T,\\
&\sup_{\e>0} \int_0^T\int_{v_1>0} f_{\e r} (\la v\ra^2+|\log f_{\e r}| ) v_1\,\dv\dt \le C_T.
\end{aligned}
\eq

\item[(iv)] In case \textbf{(Max)}, define
\[
J_{\e l}(t) := \int_{v_1<0}f_{\e l}(t,v)|v_1|\,\dv, \quad J_{\e r}(t) := \int_{v_1>0}f_{\e r}(t,v)v_1\,\dv,
\]
and let $h(z)=z\log z$. Define the Darroz\`es--Guiraud information terms by
\[
\scrE_{\e l}(t) := \int_{v_1<0} h(\frac{f_{\e l}}{\frkM})\frkM|v_1|\,\dv -h(J_{\e l}(t)),\quad \scrE_{\e r}(t) := \int_{v_1>0} h(\frac{f_{\e r}}{\frkM})\frkM v_1\,\dv -h(J_{\e r}(t)).
\] 
When the corresponding outgoing flux vanishes, the information is defined to be zero. Then
\bq\label{eq/ darr bdd}
\int_0^T ( \alpha_l\scrE_{\e l}(t) +\alpha_r\scrE_{\e r}(t)  )\,\dt \le C_T.
\eq

\end{enumerate}

\end{lemma}

\begin{proof}
We first consider the inflow boundary condition \textbf{(In)}.

By the collision-invariant identities in \eqref{eq/ M mom}, the collision operator does not contribute to the mass or energy balance. Testing the equation with a standard radial truncation of $\la v\ra^2=1+|v|^2$ and passing to the limit in the truncation parameter, we obtain
\bq\label{eq/ E diff}
\ddt\iint_{\calO} f_\e\la v\ra^2\,\dx\dv = -\int_{\R^3} \la v\ra^2f_{\e r}v_1\,\dv +\int_{\R^3} \la v\ra^2f_{\e l}v_1\,\dv.
\eq
Using the inflow boundary condition, we decompose the two boundary terms as
\[
\begin{aligned}
-\int_{\R^3} \la v\ra^2f_{\e r}v_1\,\dv &= \int_{v_1<0} \la v\ra^2g_r|v_1|\,\dv - \int_{v_1>0} \la v\ra^2f_{\e r}v_1\,\dv,\\
\int_{\R^3} \la v\ra^2f_{\e l}v_1\,\dv &= \int_{v_1>0} \la v\ra^2g_lv_1\,\dv - \int_{v_1<0} \la v\ra^2f_{\e l}|v_1|\,\dv.
\end{aligned}
\]
Integrating \eqref{eq/ E diff} over $(0,t)$ gives
\bq\label{eq/ E in}
\begin{aligned}
&\iint_{\calO} f_\e(t)\la v\ra^2\,\dx\dv +\int_0^t\int_{v_1>0} \la v\ra^2f_{\e r}v_1\,\dv\ds +\int_0^t\int_{v_1<0} \la v\ra^2f_{\e l}|v_1|\,\dv\ds \\
&\quad= \iint_{\calO} f_{\e0}\la v\ra^2\,\dx\dv +\int_0^t\int_{v_1>0} \la v\ra^2g_lv_1\,\dv\ds +\int_0^t\int_{v_1<0} \la v\ra^2g_r|v_1|\,\dv\ds.
\end{aligned}
\eq
The right-hand side is bounded uniformly in $\e$ by \eqref{eq/ init bdd} and \eqref{eq/ in data}. In particular, \eqref{eq/ E in} yields both the uniform bulk second-moment bound and the outgoing energy-flux bound.

We next consider the entropy. Formally multiplying \eqref{eq/ regbgk} by $1+\log f_\e$, we use
\[
\int_{\R^3}(M[f_\e]-f_\e)\,\dv=0 \quad \text{and} \quad \int_{\R^3} \log M[f_\e]\, (M[f_\e]-f_\e)\,\dv=0,
\]
where the latter follows since $\log M[f_\e]$ is a linear combination of $1$, $v$, and $|v|^2$. Hence
\[
\int_{\R^3} (1+\log f_\e) \lambda_\e(\rho_{f_\e}) (M[f_\e]-f_\e)\,\dv = -\int_{\R^3} \lambda_\e(\rho_{f_\e}) (f_\e-M[f_\e] ) (\log f_\e-\log M[f_\e] )\,\dv.
\]
The entropy-admissibility of the approximations then gives
\bq\label{eq/ H in}
\begin{aligned}
&\iint_{\calO} f_\e(t)\log f_\e(t)\,\dx\dv +\int_0^t\int_{v_1>0} f_{\e r}\log f_{\e r}\,v_1\,\dv\ds +\int_0^t\int_{v_1<0} f_{\e l}\log f_{\e l}\,|v_1|\,\dv\ds \\
&\quad+ \int_0^t\iint_{\calO} \lambda_\e(\rho_{f_\e}) (f_\e-M[f_\e] ) (\log f_\e-\log M[f_\e] ) \,\dx\dv\ds \\
&\quad \quad\le \iint_{\calO} f_{\e0}\log f_{\e0}\,\dx\dv +\int_0^t\int_{v_1>0} g_l\log g_l\,v_1\,\dv\ds +\int_0^t\int_{v_1<0} g_r\log g_r\,|v_1|\,\dv\ds.
\end{aligned}
\eq

To control the negative parts of the entropy, let $\log^-z:=\max\{0,-\log z\}$. For every $a>0$ and $z\ge0$,
\bq\label{eq/ H neg}
z \log^-z \le \frac{a}{2}\la v\ra^2 z +e^{-1} e^{-\frac a2 \la v\ra^2}.
\eq
Indeed, this follows by applying $r\log^-r\le e^{-1}$ to $r=z e^{\frac a2\la v\ra^2}$. Since $\int_{\R^3} e^{-\frac a2\la v\ra^2}(1 + |v_1|)\,\dv<\infty$,   the bulk and outgoing negative entropy terms are controlled by \eqref{eq/ E in}. The incoming negative entropy terms are controlled in the same way using \eqref{eq/ in data}. Thus \eqref{eq/ H in}, together with the assumptions on $f_{\e0}$ and the inflow data, yields
\begin{align*}
&\operatorname*{ess\,sup}_{0<t<T} \iint_{\calO} f_\e(t)|\log f_\e(t)|\,\dx\dv + \int_0^T\int_{v_1>0} f_{\e r}|\log f_{\e r}|v_1\,\dv\dt + \int_0^T\int_{v_1<0} f_{\e l}|\log f_{\e l}||v_1|\,\dv\dt\\
&\quad+ \int_0^T\iint_{\calO} \lambda_\e(\rho_{f_\e}) (f_\e-M[f_\e] ) (\log f_\e-\log M[f_\e] ) \,\dx\dv\dt \le C_T.
\end{align*}
Together with \eqref{eq/ E in}, this proves \eqref{eq/ bdd}, \eqref{eq/ reg diss}, and \eqref{eq/ out bdd} in case \textbf{(In)}.

We now turn to the Maxwell boundary condition \textbf{(Max)}. Let $h(z):=z\log z$ and introduce the relative entropy with respect to the fixed boundary Maxwellian $\frkM$,
\[
\scrH_{\frkM}[f_\e] := \iint_{\calO} h(\frac{f_\e}{\frkM})\frkM\,\dx\dv.
\]
At the formal level, multiplication of \eqref{eq/ regbgk} by $h'(\frac{f_\e}{\frkM})$ gives
\bq\label{eq/ test hprime}
\ddt\scrH_{\frkM}[f_\e] = -\int_{\R^3} h(\frac{f_{\e r}}{\frkM})\frkM v_1\,\dv +\int_{\R^3} h(\frac{f_{\e l}}{\frkM})\frkM v_1\,\dv + \iint_{\calO} h'(\frac{f_\e}{\frkM}) \lambda_\e(\rho_{f_\e}) (M[f_\e]-f_\e ) \,\dx\dv.
\eq
The entropy-admissibility assumption means that the integrated form of the resulting inequality is valid for the approximate solutions.

We first identify the bulk dissipation term. Since $h'(z)=1+\log z$, we have
\bq\label{eq/ H diss}
\begin{aligned}
\int_{\R^3} h'(\frac{f_\e}{\frkM}) (M[f_\e]-f_\e )\,\dv &= \int_{\R^3} \lt( 1+\log f_\e-\log\frkM \rt) (M[f_\e]-f_\e )\,\dv  \\
&= \int_{\R^3} \log f_\e (M[f_\e]-f_\e )\,\dv \\
&= \int_{\R^3} ( \log f_\e-\log M[f_\e]  ) (M[f_\e]-f_\e )\,\dv. 
\end{aligned}
\eq
Both cancellations follow from \eqref{eq/ M mom}, since $\log\frkM$ and $\log M[f_\e]$ are linear combinations of the collision invariants.

We next estimate the boundary entropy flux. We give the details at $x=-\frac12$; the computation at $x=\frac12$ is analogous. We use the Darroz\`es--Guiraud argument for the Maxwell boundary condition; see \cite[Lemma 2.2]{Mis10}. Since $\frkM$ is centered, $\frkM(\scrV v)=\frkM(v)$, and the normalization \eqref{eq/ M norm} gives
\[
\int_{v_1>0}\frkM(v)v_1\,\dv = \int_{v_1<0}\frkM(v)|v_1|\,\dv =1.
\]
Using the Maxwell boundary condition for $v_1>0$, we write
\bq\label{eq/ darr first}
\begin{aligned}
\int_{\R^3} h\lt(\frac{f_{\e l}}{\frkM}\rt) \frkM v_1\,\dv &= \int_{v_1>0} h\lt( \frac{ (1-\alpha_l)f_{\e l}(\scrV v) +\alpha_l\frkM(v)J_{\e l}(t) }{\frkM(v)} \rt) \frkM(v)v_1\,\dv  \\
&\quad+ \int_{v_1<0} h\lt(\frac{f_{\e l}}{\frkM}\rt) \frkM v_1\,\dv. 
\end{aligned}
\eq
Since $h$ is convex, we get
\[
\begin{aligned}
h\lt( (1-\alpha_l) \frac{f_{\e l}(\scrV v)}{\frkM(v)} +\alpha_lJ_{\e l}(t) \rt) &\le (1-\alpha_l) h\lt( \frac{f_{\e l}(\scrV v)}{\frkM(v)} \rt) + \alpha_lh(J_{\e l}(t)).
\end{aligned}
\]
Substituting this into \eqref{eq/ darr first} yields
\begin{align*}
\int_{\R^3} h\lt(\frac{f_{\e l}}{\frkM}\rt) \frkM v_1\,\dv &\le (1-\alpha_l) \int_{v_1>0} h\lt( \frac{f_{\e l}(\scrV v)}{\frkM(v)} \rt) \frkM(v)v_1\,\dv\\
&\quad+ \alpha_lh(J_{\e l}(t)) \int_{v_1>0}\frkM(v)v_1\,\dv + \int_{v_1<0} h\lt(\frac{f_{\e l}}{\frkM}\rt) \frkM v_1\,\dv.
\end{align*}
Under the change of variables $v\mapsto\scrV v$, we obtain
\[
\begin{aligned}
&\int_{v_1>0} h\lt( \frac{f_{\e l}(\scrV v)}{\frkM(v)} \rt) \frkM(v)v_1\,\dv = \int_{v_1<0} h\lt( \frac{f_{\e l}(v)}{\frkM(v)} \rt) \frkM(v)|v_1|\,\dv.
\end{aligned}
\]
It follows that
\bq\label{eq/ darr}
\int_{\R^3} h\lt(\frac{f_{\e l}}{\frkM}\rt) \frkM v_1\,\dv \le -\alpha_l \lt[ \int_{v_1<0} h\lt(\frac{f_{\e l}}{\frkM}\rt) \frkM|v_1|\,\dv - h(J_{\e l}(t)) \rt] = -\alpha_l\scrE_{\e l}(t).
\eq

Let us also verify explicitly that $\scrE_{\e l}(t)\ge0$. If $J_{\e l}(t)>0$, then
\[
\scrE_{\e l}(t) = \int_{v_1<0} f_{\e l} \log\frac{f_{\e l}} {J_{\e l}(t)\frkM} |v_1|\,\dv = J_{\e l}(t) \int_{v_1<0} \frac{f_{\e l}} {J_{\e l}(t)\frkM} \log \frac{f_{\e l}} {J_{\e l}(t)\frkM} \frkM|v_1|\,\dv.
\]
Since $\frkM|v_1|\,\dv$ is a probability measure on $\{v_1<0\}$ and
\[
\int_{v_1<0} \frac{f_{\e l}} {J_{\e l}(t)\frkM} \frkM|v_1|\,\dv =1,
\]
Jensen's inequality gives $\scrE_{\e l}(t) \ge J_{\e l}(t)h(1) =0$. When $J_{\e l}(t)=0$, we set $\scrE_{\e l}(t)=0$. The same argument at $x=\frac12$ gives
\[
\int_{\R^3} h(\frac{f_{\e r}}{\frkM})\frkM v_1\,\dv \ge \alpha_r\scrE_{\e r}(t) \ge0.
\]
Notice that the sign is reversed at the right endpoint due to the orientation of the boundary flux.

Combining \eqref{eq/ test hprime}, \eqref{eq/ H diss}, and \eqref{eq/ darr}, together with the analogous estimate at $x=\frac12$, gives
\bq\label{eq/ H max}
\begin{aligned}
&\scrH_{\frkM}[f_\e(t)] +\int_0^t ( \alpha_l\scrE_{\e l}(s) +\alpha_r\scrE_{\e r}(s)  )\,\ds \\
&\quad + \int_0^t\iint_{\calO} \lambda_\e(\rho_{f_\e}) (f_\e-M[f_\e] ) (\log f_\e-\log M[f_\e] ) \,\dx\dv\ds \\
&\quad \quad\le \scrH_{\frkM}[f_{\e0}].
\end{aligned}
\eq
We now extract the second-moment estimate from the relative entropy. Write
\[
\frkM(v)=A e^{-b|v|^2}, \quad A,b>0.
\]
For $0<a<b$, Young's inequality $z\log z\ge yz-e^{y-1}$ for $z\ge0$, $y\in\R$, applied with $z=\frac{f_\e}{\frkM}$, $y=a|v|^2$, gives
\[
\begin{aligned}
h(\frac{f_\e}{\frkM})\frkM \ge a|v|^2f_\e - e^{-1}\frkM e^{a|v|^2} = a|v|^2f_\e - A e^{-1}e^{-(b-a)|v|^2}.
\end{aligned}
\]
Hence
\bq\label{eq/ H coerc}
\scrH_{\frkM}[f_\e(t)] \ge a\iint_{\calO}|v|^2f_\e(t)\,\dx\dv -C_{\frkM,a},
\eq
where $C_{\frkM,a}<\infty$ is independent of $\e$ and $t$.

The initial relative entropy is uniformly bounded from above. Indeed, since $\log\frkM(v)=\log A-b|v|^2$, we have
\[
\scrH_{\frkM}[f_{\e0}] = \iint_{\calO} f_{\e0}\log f_{\e0}\,\dx\dv -(\log A) \iint_{\calO}f_{\e0}\,\dx\dv + b\iint_{\calO}|v|^2f_{\e0}\,\dx\dv,
\]
which is bounded uniformly in $\e$ by the assumptions on the approximate initial data. The Maxwell boundary condition preserves the total mass, and hence
\[
\iint_{\calO}f_\e(t)\,\dx\dv = \iint_{\calO}f_{\e0}\,\dx\dv
\]
for almost every $t>0$. In particular, the mass is bounded uniformly in $\e$. Combining this fact with \eqref{eq/ H max} and \eqref{eq/ H coerc} gives
\bq\label{eq/ E max}
\sup_{\e>0}\, \operatorname*{ess\,sup}_{0<t<T} \iint_{\calO} \la v\ra^2f_\e(t)\,\dx\dv \le C_T.
\eq

We next recover the absolute entropy. Since
\[
\scrH_{\frkM}[f_\e] = \iint_{\calO}f_\e\log f_\e\,\dx\dv - \iint_{\calO}f_\e\log\frkM\,\dx\dv \quad \text{and} \quad \log\frkM=\log A-b|v|^2,
\]
the upper bound furnished by \eqref{eq/ H max}, together with mass conservation, gives a uniform upper bound on the signed entropy
\[
\iint_{\calO}f_\e\log f_\e\,\dx\dv.
\]
On the other hand, \eqref{eq/ H neg} and \eqref{eq/ E max} yield
\[
\sup_{\e>0}\, \operatorname*{ess\,sup}_{0<t<T} \iint_{\calO} f_\e(t)\log^-f_\e(t)\,\dx\dv \le C_T.
\]
It follows that
\[
\sup_{\e>0}\, \operatorname*{ess\,sup}_{0<t<T} \iint_{\calO} f_\e(t)|\log f_\e(t)|\,\dx\dv \le C_T.
\]
This proves \eqref{eq/ bdd} in case \textbf{(Max)}.

Finally, the coercive lower bound \eqref{eq/ H coerc} and the uniform upper bound on $\scrH_{\frkM}[f_{\e0}]$ allow us to return to \eqref{eq/ H max}. Since both $\scrE_{\e l},\scrE_{\e r}$ and
\[
\lambda_\e(\rho_{f_\e}) (f_\e-M[f_\e] ) (\log f_\e-\log M[f_\e] )
\]
are nonnegative, we obtain
\[
\int_0^T ( \alpha_l\scrE_{\e l}(t) +\alpha_r\scrE_{\e r}(t)  )\,\dt \le C_T
\]
and
\[
\int_0^T\iint_{\calO} \lambda_\e(\rho_{f_\e}) (f_\e-M[f_\e] ) (\log f_\e-\log M[f_\e] ) \,\dx\dv\dt \le C_T.
\]
These are precisely \eqref{eq/ darr bdd} and \eqref{eq/ reg diss}. This completes the proof.
\end{proof}

As a consequence of \eqref{eq/ bdd} and Lemma \ref{lem/ j bounds}, the approximate traces also satisfy
\bq\label{eq/ bdy E reg}
\begin{aligned}
&\sup_{\e>0} \alpha_l \int_0^T\int_{v_1<0} \la v\ra^2f_{\e l}|v_1|\,\dv\dt + \sup_{\e>0} \alpha_r \int_0^T\int_{v_1>0} \la v\ra^2f_{\e r}v_1\,\dv\dt \le C_T,
\end{aligned}
\eq
where $\alpha_l=\alpha_r=1$ in case \textbf{(In)}.

For later use, we write
\[
\scrH[f_\e] := \iint_{\calO}f_\e\log f_\e\,\dx\dv, \quad \scrD_\e[f_\e] := \iint_{\calO}D_\e[f_\e]\,\dx\dv,
\]
where
\[
D_\e[f_\e] := \lambda_\e(\rho_{f_\e}) (f_\e-M[f_\e] ) (\log f_\e-\log M[f_\e] ).
\]

%
%
%
%
%
%
%
%
%
%
%

\subsection{A de la Vall\'ee Poussin weight}

The classical de la Vall\'ee Poussin criterion provides a superlinear weight for a uniformly integrable family; see, for instance, \cite{Lau15}. Here we require in addition a quantitative curvature bound on the weight. This additional property will be used to control the Maxwellian Jensen gap in Section \ref{subsec/ propagation}. We give a direct construction.

\begin{lemma}\label{lem/ dlVP}
Let $(f_{\e0})$ be a nonnegative sequence satisfying \eqref{eq/ init E}. Then there exists $\Phi:[0,\infty)\to[0,\infty)$ with the following properties.
\begin{enumerate}

\item[(i)] $\Phi\in C^\infty([0,\infty))$ and $\Phi(0)=\Phi'(0)=0$. 

\item[(ii)] $\Phi$ is convex and $\Phi'$ is increasing and concave. Moreover, for every $z>0$,
\begin{align}
&\frac{\Phi(z)}{z} \le \Phi'(z) \le 2\frac{\Phi(z)}{z}, \label{eq/ Phi p}\\
&0\le \Phi''(z) \le \min\lt\{ \frac{\Phi'(z)}{z}, \frac{C}{z} \rt\} \quad\text{for some }C>0, \label{eq/ Phi pp}\\
&\lim_{z\to\infty}\frac{\Phi(z)}{z} =\infty. \label{eq/ Phi sup}
\end{align}

\item[(iii)] The initial data satisfy
\bq\label{eq/ init tight}
\sup_{\e>0} \iint_{\calO} \Phi(\la v\ra^2)f_{\e0}\,\dx\dv \le1.
\eq

\end{enumerate}

\end{lemma}

\begin{proof}
Since \eqref{eq/ init E} assumes strong convergence in $L^1(\calO,\la v\ra^2\,\dx\dv)$, the corresponding energy tails are uniformly small. More precisely,
\bq\label{eq/ init tail}
\lim_{R\to\infty} \sup_{\e>0} \iint_{I\times\{|v|>R\}} \la v\ra^2 f_{\e0}\,\dx\dv =0.
\eq
Indeed, let $\eta>0$. By \eqref{eq/ init E}, for all sufficiently small $\e$ along the fixed sequence,
\[
\iint_{\calO} \la v\ra^2|f_{\e0}-f_0|\,\dx\dv <\frac{\eta}{2}.
\]
Since $\la v\ra^2f_0\in L^1(\calO)$, we may choose $R>0$ large enough that
\[
\iint_{I\times\{|v|>R\}} \la v\ra^2f_0\,\dx\dv <\frac{\eta}{2}.
\]
Hence, for all sufficiently small $\e$ along the fixed sequence,
\[
\iint_{I\times\{|v|>R\}} \la v\ra^2f_{\e0}\,\dx\dv <\eta.
\]
The finitely many remaining elements are treated by increasing $R$ once more. This proves \eqref{eq/ init tail}.

Set
\[
M_0:= \sup_{\e>0} \iint_{\calO}\la v\ra^2f_{\e0}\,\dx\dv<\infty.
\]
We construct $\Phi'$ as a series. For $k\ge1$, let
\[
\phi_k(z):=\frac{z}{a_k+z},
\]
where the positive sequence $(a_k)$ will be chosen recursively. Using \eqref{eq/ init tail}, choose $R_k$ large enough that
\[
\sup_{\e>0} \iint_{I\times\{|v|>R_k\}} \la v\ra^2f_{\e0}\,\dx\dv <\frac1{2^{k+1}}.
\]
We then choose $a_k$ sufficiently large such that
\begin{equation}\label{eq/ doubling}
a_1\ge2,\quad a_{k+1}\ge2a_k, \quad \text{and} \quad \frac{\la R_k\ra^2}{a_k}M_0<\frac1{2^{k+1}}.
\end{equation}
This yields, for every $\e$,
\bq\label{eq/ phik}
\iint_{\calO} \la v\ra^2 \phi_k(\la v\ra^2) f_{\e0}\,\dx\dv \le \frac{\la R_k\ra^2}{a_k} \iint_{\calO}\la v\ra^2f_{\e0}\,\dx\dv + \iint_{I\times\{|v|>R_k\}} \la v\ra^2f_{\e0}\,\dx\dv <\frac1{2^{k}}.
\eq

Define
\[
\Phi'(z):=\sum_{k\ge1}\phi_k(z), \quad \Phi(z):=\int_0^z\Phi'(s)\,\ds.
\]
The doubling condition on $(a_k)$ guarantees convergence of this series and all its derivatives uniformly on compact subsets of $[0,\infty)$. Indeed, for $m\ge1$, $|\phi_k^{(m)}(z)| \le \frac{m!}{a_k^m}$, while on every bounded interval $0\le z\le L$, $0\le\phi_k(z)\le\frac{L}{a_k}$. Hence $\Phi\in C^\infty([0,\infty))$.

Since each $\phi_k$ is increasing and concave, $\Phi'$ has the same properties. Also, for any fixed $N$,
\[
\liminf_{z\to\infty}\Phi'(z) \ge \lim_{z\to\infty}\sum_{k=1}^N\phi_k(z) =N,
\]
and hence $\Phi'(z)\xrightarrow[z\to\infty]{}\infty$. Together with the comparison between $\frac{\Phi(z)}z$ and $\Phi'(z)$ proved below, this gives the superlinearity needed for velocity tightness.

By monotone convergence and \eqref{eq/ phik}, we estimate
\[
\iint_{\calO} \Phi(\la v\ra^2)f_{\e0}\,\dx\dv \le \iint_{\calO} \la v\ra^2\Phi'(\la v\ra^2) f_{\e0}\,\dx\dv = \sum_{k\ge1} \iint_{\calO} \la v\ra^2\phi_k(\la v\ra^2) f_{\e0}\,\dx\dv \le1.
\]
This proves \eqref{eq/ init tight}.

The monotonicity of $\Phi'$ and the identity $\Phi(0)=0$ give $\Phi(z)\le z\Phi'(z)$. Since $\Phi'$ is concave and $\Phi'(0)=0$,
\[
\frac{\Phi'(s)}{s} \ge \frac{\Phi'(z)}{z}, \quad 0<s<z.
\]
Consequently, we obtain
\[
\Phi(z) = \int_0^z\Phi'(s)\,\ds \ge \frac{\Phi'(z)}{z} \int_0^z s\,\ds = \frac12z\Phi'(z),
\]
which proves \eqref{eq/ Phi p}. In particular, \eqref{eq/ Phi p} and $\Phi'(z)\to\infty$ imply \eqref{eq/ Phi sup}.

It remains to prove the second estimate in \eqref{eq/ Phi pp}. Since
\[
\Phi''(z) = \sum_{k\ge1} \frac{a_k}{(a_k+z)^2},
\]
we split
\bq\label{eq/ split sum}
\Phi''(z) \le \frac1{z^2} \sum_{\{k:a_k\le z\}}a_k + \sum_{\{k:a_k>z\}}\frac1{a_k}.
\eq
If $0<z<a_1$, the first sum is empty and
\[
\Phi''(z) \le \sum_{k\ge1}\frac1{a_k} \le \frac2{a_1} \le \frac2z.
\]
If $z\ge a_1$, let $k_*:=\max\{k:a_k\le z\}$. The doubling property (recalling \eqref{eq/ doubling}) gives
\[
\sum_{k\le k_*}a_k\le2z, \quad \sum_{k>k_*}\frac1{a_k}\le\frac2z.
\]
Inserting these bounds into \eqref{eq/ split sum} yields $\Phi''(z)\le\frac4z$. Finally, the concavity of $\Phi'$ and $\Phi'(0)=0$ imply
\[
\Phi''(z)\le\frac{\Phi'(z)}{z}.
\]
This proves \eqref{eq/ Phi pp}. This completes the proof.
\end{proof}

The same function can be chosen to control the incoming energy tails.

\begin{corollary}\label{cor/ in Phi}
Assume that the inflow data satisfy \eqref{eq/ in data}. Then $\Phi$ in Lemma \ref{lem/ dlVP} can be chosen such that, in addition,
\[
\int_0^\infty\int_{v_1<0} \Phi(\la v\ra^2)g_r|v_1|\,\dv\dt \le1, \quad \int_0^\infty\int_{v_1>0} \Phi(\la v\ra^2)g_lv_1\,\dv\dt \le1.
\]
\end{corollary}

\begin{proof}
Set
\[
M_0:= \sup_{\e>0} \iint_{\calO}\la v\ra^2f_{\e0}\,\dx\dv, \quad  M_l:= \int_0^\infty\int_{v_1>0} \la v\ra^2g_l v_1\,\dv\dt, \quad M_r:= \int_0^\infty\int_{v_1<0} \la v\ra^2g_r |v_1|\,\dv\dt.
\]
These quantities are finite by \eqref{eq/ init bdd} and \eqref{eq/ in data}. Moreover,
\begin{align*}
\lim_{R\to\infty} \int_0^\infty\int_{\substack{v_1>0\\
|v|>R}} \la v\ra^2g_l v_1\,\dv\dt=0,\quad \lim_{R\to\infty} \int_0^\infty\int_{\substack{v_1<0\\
|v|>R}} \la v\ra^2g_r |v_1|\,\dv\dt=0.
\end{align*}

We repeat the construction in Lemma \ref{lem/ dlVP}, choosing the sequence $(R_k)$ such that the three tail estimates
\begin{align*}
&\sup_{\e>0} \iint_{I\times\{|v|>R_k\}} \la v\ra^2f_{\e0}\,\dx\dv<\frac1{2^{k+2}},\\
& \int_0^\infty\int_{\substack{v_1>0\\ |v|>R_k}} \la v\ra^2g_l v_1\,\dv\dt<\frac1{2^{k+2}},\\
&\int_0^\infty\int_{\substack{v_1<0\\ |v|>R_k}} \la v\ra^2g_r |v_1|\,\dv\dt<\frac1{2^{k+2}}
\end{align*}
hold simultaneously.

After $R_k$ has been fixed, choose $a_k$ sufficiently large such that $a_{k+1}\ge2a_k$ and
\[
\frac{\la R_k\ra^2}{a_k} \max\{M_0,M_l,M_r\} <\frac1{2^{k+2}}.
\]
With $\phi_k(z)=\frac{z}{a_k+z}$, the same decomposition as in \eqref{eq/ phik} gives
\begin{align*}
\iint_{\calO} \la v\ra^2\phi_k(\la v\ra^2) f_{\e0}\,\dx\dv &<\frac1{2^{k+1}},\\
\int_0^\infty\int_{v_1>0} \la v\ra^2\phi_k(\la v\ra^2) g_l v_1\,\dv\dt &<\frac1{2^{k+1}},\\
\int_0^\infty\int_{v_1<0} \la v\ra^2\phi_k(\la v\ra^2) g_r |v_1|\,\dv\dt &<\frac1{2^{k+1}}.
\end{align*}
Define $\Phi'$ and $\Phi$ from this common sequence $(a_k)$ exactly as in Lemma \ref{lem/ dlVP}. Since
\[
\Phi(z)\le z\Phi'(z) =z\sum_{k\ge1}\phi_k(z),
\]
monotone convergence gives
\[
\sup_{\e>0} \iint_{\calO}\Phi(\la v\ra^2)f_{\e0}\,\dx\dv \le1,\quad \int_0^\infty\int_{v_1>0} \Phi(\la v\ra^2)g_l v_1\,\dv\dt \le1,\quad \int_0^\infty\int_{v_1<0} \Phi(\la v\ra^2)g_r |v_1|\,\dv\dt \le1.
\]
The remaining properties of $\Phi$ are unchanged, since they depend only on the doubling condition \eqref{eq/ doubling} for $(a_k)$. This completes the proof.
\end{proof}

%
%
%
%
%
%
%
%
%
%
%

\subsection{Propagation of the weighted energy}\label{subsec/ propagation}

We now propagate the de la Vall\'ee Poussin bound along the approximate solutions.

\begin{proposition}\label{prop/ super}
Let $f_\e$ be the entropy-admissible approximate solutions introduced in Section \ref{sec/ approximate}, and let $\Phi$ be given by Lemma \ref{lem/ dlVP}. In case \textbf{(In)}, choose the common function $\Phi$ provided by Corollary \ref{cor/ in Phi}. Then, for every $T>0$,
\bq\label{eq/ propagation}
\begin{aligned}
&\sup_{\e>0} \operatorname*{ess\,sup}_{0<t<T} \iint_{\calO} \Phi(\la v\ra^2)f_\e(t)\,\dx\dv \\ 
&\quad + \sup_{\e>0}\, \alpha_l \int_0^T\int_{v_1<0} \Phi(\la v\ra^2)f_{\e l}|v_1|\,\dv\dt + \sup_{\e>0}\, \alpha_r \int_0^T\int_{v_1>0} \Phi(\la v\ra^2)f_{\e r}v_1\,\dv\dt  \\
&\quad \quad\le C_T, 
\end{aligned}
\eq
where $\alpha_l=\alpha_r=1$ in case \textbf{(In)} and $C_T$ is independent of $\e$.
\end{proposition}

\begin{proof}
We first perform the computation formally with $\Phi(\la v\ra^2)$ as a test function. The justification by truncation is given at the end of the proof. We obtain
\bq\label{eq/ direct}
\begin{aligned}
\ddt \iint_{\calO} \Phi(\la v\ra^2)f_\e\,\dx\dv &= \iint_{\calO} \lambda_\e(\rho_{f_\e}) \Phi(\la v\ra^2) (M[f_\e]-f_\e ) \,\dx\dv\\
&\quad- \int_{\R^3} \Phi(\la v\ra^2)f_{\e r}v_1\,\dv + \int_{\R^3} \Phi(\la v\ra^2)f_{\e l}v_1\,\dv.
\end{aligned}
\eq

\medskip \noindent $\bullet$ (The bulk collision term) It suffices to consider $\rho_{f_\e}>0$, since the estimate is immediate on the vacuum set. Set
\[
X_\e := \int_{\R^3} \la v\ra^2 \frac{M[f_\e]}{\rho_{f_\e}}\,\dv = 1+|u_{f_\e}|^2+3\theta_{f_\e}.
\]
By Jensen's inequality, applied to the probability measure $\frac{f_\e}{\rho_{f_\e}}\,\dv$,
\[
\frac1{\rho_{f_\e}} \int_{\R^3} \Phi(\la v\ra^2)f_\e\,\dv \ge \Phi\lt( \frac1{\rho_{f_\e}} \int_{\R^3}\la v\ra^2f_\e\,\dv \rt) = \Phi(X_\e),
\]
where the last identity follows from the definitions of $u_{f_\e}$ and $\theta_{f_\e}$. Hence
\bq\label{eq/ inter}
\int_{\R^3} \Phi(\la v\ra^2) (M[f_\e]-f_\e )\,\dv \le \rho_{f_\e} \lt[ \int_{\R^3} \Phi(\la v\ra^2) \frac{M[f_\e]}{\rho_{f_\e}}\,\dv - \Phi(X_\e) \rt]. 
\eq

Let $h(z):=z\log z$. By \eqref{eq/ Phi pp}, after enlarging the constant $C$ if necessary,
\[
\Phi''(z)\le\frac{C}{z}=Ch''(z),
\]
and hence $\Phi-Ch$ is concave. Thus, for $y,X>0$,
\[
\Phi(y)-\Phi(X)-\Phi'(X)(y-X) \le C\lt( h(y)-h(X)-h'(X)(y-X) \rt).
\]
Writing $r=\frac yX$, we compute
\[
h(y)-h(X)-h'(X)(y-X) = X(r\log r-r+1 ).
\]
Since $\log r\le r-1$,
\[
r\log r-r+1 \le r(r-1)-r+1 = (r-1)^2,
\]
and hence
\bq\label{eq/ log gap}
h(y)-h(X)-h'(X)(y-X) \le \frac{(y-X)^2}{X}.
\eq

Applying \eqref{eq/ log gap} with $y=\la v\ra^2$ and $X=X_\e$, and using
\[
\int_{\R^3} (\la v\ra^2-X_\e) \frac{M[f_\e]}{\rho_{f_\e}}\,\dv =0,
\]
we obtain
\[
\int_{\R^3} \Phi(\la v\ra^2) \frac{M[f_\e]}{\rho_{f_\e}}\,\dv -\Phi(X_\e) \le \frac{C}{X_\e} \int_{\R^3} (\la v\ra^2-X_\e)^2 \frac{M[f_\e]}{\rho_{f_\e}}\,\dv.
\]

To compute the last integral, we use the change of variables $w=v-u_{f_\e}$. Since $X_\e = 1+|u_{f_\e}|^2+3\theta_{f_\e}$, we have
\[
\la v\ra^2-X_\e = 1+|u_{f_\e}+w|^2 -(1+|u_{f_\e}|^2+3\theta_{f_\e} ) = 2u_{f_\e}\cdot w +|w|^2-3\theta_{f_\e}.
\]
Expanding the square, we obtain
\begin{align*}
\int_{\R^3} (\la v\ra^2-X_\e)^2M[f_\e]\,\dv &= 4\int_{\R^3} (u_{f_\e}\cdot w)^2M[f_\e]\,\dv + \int_{\R^3} (|w|^2-3\theta_{f_\e} )^2 M[f_\e]\,\dv\\
&\quad+ 4\int_{\R^3} (u_{f_\e}\cdot w) (|w|^2-3\theta_{f_\e} ) M[f_\e]\,\dv.
\end{align*}
The last integral vanishes by oddness. For the first term, the standard Gaussian identities
\[
\int_{\R^3}w_iw_jM[f_\e]\,\dv = \rho_{f_\e}\theta_{f_\e}\delta_{ij}
\]
give
\[
\int_{\R^3} (u_{f_\e}\cdot w)^2M[f_\e]\,\dv = \rho_{f_\e}\theta_{f_\e}|u_{f_\e}|^2.
\]
For the second term, we use
\[
\int_{\R^3}|w|^2M[f_\e]\,\dv = 3\rho_{f_\e}\theta_{f_\e} \quad \text{and} \quad \int_{\R^3}|w|^4M[f_\e]\,\dv = 15\rho_{f_\e}\theta_{f_\e}^2.
\]
Indeed,
\[
\int_{\R^3}w_i^4M[f_\e]\,\dv = 3\rho_{f_\e}\theta_{f_\e}^2, \quad \int_{\R^3}w_i^2w_j^2M[f_\e]\,\dv = \rho_{f_\e}\theta_{f_\e}^2 \quad(i\ne j),
\]
and hence
\[
\int_{\R^3}|w|^4M[f_\e]\,\dv = \sum_{i=1}^3 \int_{\R^3}w_i^4M[f_\e]\,\dv + 2\sum_{1\le i<j\le3} \int_{\R^3}w_i^2w_j^2M[f_\e]\,\dv = 15\rho_{f_\e}\theta_{f_\e}^2.
\]
Consequently,
\[
\int_{\R^3} (|w|^2-3\theta_{f_\e} )^2 M[f_\e]\,\dv = \int_{\R^3}|w|^4M[f_\e]\,\dv -6\theta_{f_\e} \int_{\R^3}|w|^2M[f_\e]\,\dv +9\theta_{f_\e}^2 \int_{\R^3}M[f_\e]\,\dv = 6\rho_{f_\e}\theta_{f_\e}^2.
\]
We conclude that
\[
\int_{\R^3} (\la v\ra^2-X_\e)^2M[f_\e]\,\dv = \rho_{f_\e} ( 4\theta_{f_\e}|u_{f_\e}|^2 +6\theta_{f_\e}^2  ).
\]

It follows that
\bq\label{eq/ gaussian gap}
\int_{\R^3} \Phi(\la v\ra^2) \frac{M[f_\e]}{\rho_{f_\e}}\,\dv -\Phi(X_\e) \le C \frac{ 4\theta_{f_\e}|u_{f_\e}|^2 +6\theta_{f_\e}^2 }{ 1+|u_{f_\e}|^2+3\theta_{f_\e} } \le C\theta_{f_\e}. 
\eq
Combining \eqref{eq/ inter} and \eqref{eq/ gaussian gap} yields
\bq\label{eq/ Phi Q}
\int_{\R^3} \Phi(\la v\ra^2) (M[f_\e]-f_\e )\,\dv \le C\rho_{f_\e}\theta_{f_\e}.
\eq
From Proposition \ref{prop/ rho2},
\bq\label{eq/ therm reg}
\int_0^T\int_I \lambda_\e(\rho_{f_\e}) \rho_{f_\e}\theta_{f_\e} \,\dx\dt \le C_T,
\eq
with a constant independent of $\e$; thus \eqref{eq/ Phi Q}--\eqref{eq/ therm reg} show that the time-integrated bulk collision term in \eqref{eq/ direct} satisfies
\begin{equation}\label{eq/ therm cont}
\int_0^T\iint_{\calO} \lambda_\e(\rho_{f_\e})\Phi(\la v\ra^2) (M[f_\e]-f_\e)\,\dx\dv\dt \le C\int_0^T\int_I \lambda_\e(\rho_{f_\e})\rho_{f_\e}\theta_{f_\e}\,\dx\dt \le C_T.
\end{equation}

Integrating \eqref{eq/ direct} in time and using \eqref{eq/ init tight}, \eqref{eq/ therm cont}, we obtain
\bq\label{eq/ Phi est}
\iint_{\calO} \Phi(\la v\ra^2)f_\e(t)\,\dx\dv \le 1+C_T - \int_0^t\int_{\R^3} \Phi(\la v\ra^2)f_{\e r}v_1\,\dv\ds + \int_0^t\int_{\R^3} \Phi(\la v\ra^2)f_{\e l}v_1\,\dv\ds.
\eq

\medskip \noindent $\bullet$ (Boundary terms) For \textbf{(In)}, the boundary conditions give
\[
-\int_{\R^3} \Phi(\la v\ra^2)f_{\e r}v_1\,\dv = \int_{v_1<0} \Phi(\la v\ra^2)g_r|v_1|\,\dv - \int_{v_1>0} \Phi(\la v\ra^2)f_{\e r}v_1\,\dv,
\]
and an analogous identity holds at $x=-\frac12$. Hence Corollary \ref{cor/ in Phi} and \eqref{eq/ Phi est} give
\[
\iint_{\calO} \Phi(\la v\ra^2)f_\e(t)\,\dx\dv +\int_0^t\int_{v_1>0} \Phi(\la v\ra^2)f_{\e r}v_1\,\dv\ds + \int_0^t\int_{v_1<0} \Phi(\la v\ra^2)f_{\e l}|v_1|\,\dv\ds \le C_T.
\]

For \textbf{(Max)}, set
\[
c_{\frkM,\Phi} := \int_{v_1<0} \Phi(\la v\ra^2)\frkM(v)|v_1|\,\dv.
\]
The curvature estimate \eqref{eq/ Phi pp} implies $\Phi(z)\le C(1+z\log(1+z))$, and hence $c_{\frkM,\Phi}<\infty$ by the Gaussian decay of $\frkM$.

Consider the boundary $x=\frac12$. Since $\Phi(\la\scrV v\ra^2)=\Phi(\la v\ra^2)$, the Maxwell boundary condition gives
\begin{align*}
-\int_{\R^3} \Phi(\la v\ra^2)v_1f_{\e r}(v)\,\dv &= -\int_{v_1>0} \Phi(\la v\ra^2)v_1f_{\e r}(v)\,\dv\\
&\quad -\int_{v_1<0} \Phi(\la v\ra^2)v_1 \Big( (1-\alpha_r)f_{\e r}(\scrV v) +\alpha_r\frkM(v)J_{\e r} \Big)\,\dv.
\end{align*}
Using the change of variables $v\mapsto\scrV v$ in the reflected part, we have
\[
\int_{v_1<0} \Phi(\la v\ra^2)v_1f_{\e r}(\scrV v)\,\dv = -\int_{v_1>0} \Phi(\la v\ra^2)v_1f_{\e r}(v)\,\dv.
\]
Moreover,
\[
\int_{v_1<0} \Phi(\la v\ra^2)v_1\frkM(v)\,\dv = -c_{\frkM,\Phi}.
\]
Hence
\bq\label{eq/ Phi r}
\begin{aligned}
-\int_{\R^3} \Phi(\la v\ra^2)v_1f_{\e r}\,\dv &= -\alpha_r \int_{v_1>0} \Phi(\la v\ra^2)v_1f_{\e r}\,\dv +\alpha_rc_{\frkM,\Phi}J_{\e r} \\
&= -\alpha_r \lt[ \int_{v_1>0} \Phi(\la v\ra^2)v_1f_{\e r}\,\dv - c_{\frkM,\Phi}J_{\e r} \rt]. 
\end{aligned}
\eq

The computation at $x=-\frac12$ is analogous and gives
\bq\label{eq/ Phi l}
\int_{\R^3} \Phi(\la v\ra^2)v_1f_{\e l}\,\dv = -\alpha_l \lt[ \int_{v_1<0} \Phi(\la v\ra^2)|v_1|f_{\e l}\,\dv - c_{\frkM,\Phi}J_{\e l} \rt].
\eq
Substituting \eqref{eq/ Phi r} and \eqref{eq/ Phi l} into \eqref{eq/ Phi est}, we obtain
\begin{align*}
&\iint_{\calO} \Phi(\la v\ra^2)f_\e(t)\,\dx\dv +\alpha_l \int_0^t\int_{v_1<0} \Phi(\la v\ra^2)f_{\e l}|v_1|\,\dv\ds + \alpha_r \int_0^t\int_{v_1>0} \Phi(\la v\ra^2)f_{\e r}v_1\,\dv\ds \\
&\quad \le 1+C_T +c_{\frkM,\Phi} \int_0^t ( \alpha_lJ_{\e l} +\alpha_rJ_{\e r}  )\,\ds.
\end{align*}

Lemma \ref{lem/ j bounds}, together with \eqref{eq/ bdd}, controls the last term uniformly in $\e$. This proves the formal estimate \eqref{eq/ propagation}.

\medskip \noindent $\bullet$ (Justification by truncation) For $R>1$, define
\[
\Phi_R(z) :=
\begin{cases}
\Phi(z), & z\le R,\\
\Phi(R)+\Phi'(R)(z-R), & z>R.
\end{cases}
\]
Then $\Phi_R$ is convex and $C^1$, its derivative is locally Lipschitz, and
\[
0\le\Phi_R\uparrow\Phi, \quad 0\le\Phi_R''(z)\le\frac{C}{z} \quad\text{for a.e. }z>0.
\]
Moreover,
\[
\Phi_R(\la v\ra^2) \le C_R\la v\ra^2.
\]
For fixed $R$, use $\Phi_R(\la v\ra^2)\chi_L(v)$ as a test function, with a smooth approximation if necessary. Letting $L\to\infty$ is justified by \eqref{eq/ Q wt} for each fixed $\e$ and by the trace estimate \eqref{eq/ tr wt}. We first pass to the limit $L\to\infty$ with $R$ fixed. The above bulk and boundary estimates then apply to $\Phi_R$, with constants independent of $R$, since $\Phi_R\le\Phi$ and the curvature bound is uniform in $R$. Finally, letting $R\to\infty$ and using monotone convergence in the bulk and in the nonnegative outgoing boundary terms gives \eqref{eq/ propagation}. This completes the proof.
\end{proof}

The above estimate yields the uniform integrability of the kinetic energy required for the compactness argument.

\begin{corollary}
Let $(f_\e)$ be the approximate solutions of \eqref{eq/ regbgk} satisfying the assumptions of Proposition \ref{prop/ super}. For every $T>0$, we have
\bq\label{eq/ E tail}
\lim_{R\to\infty} \sup_{\e>0} \operatorname*{ess\,sup}_{0<t<T} \iint_{I\times\{|v|>R\}} \la v\ra^2f_\e(t,x,v)\,\dx\dv =0.
\eq
The corresponding outgoing boundary tails satisfy
\bq\label{eq/ bdy tail}
\lim_{R\to\infty} \sup_{\e>0} \alpha_l \int_0^T\int_{\substack{v_1<0\\
|v|>R}} \la v\ra^2f_{\e l}|v_1|\,\dv\dt =0,\quad \lim_{R\to\infty} \sup_{\e>0} \alpha_r \int_0^T\int_{\substack{v_1>0\\
|v|>R}} \la v\ra^2f_{\e r}v_1\,\dv\dt =0,
\eq
with $\alpha_l=\alpha_r=1$ in case \textbf{(In)}.
\end{corollary}

\begin{proof}
Since $\Phi$ is convex and $\Phi(0)=0$, the function $z\mapsto\frac{\Phi(z)}z$ is nondecreasing. Hence, on $|v|>R$,
\[
\la v\ra^2 \le \frac{\la R\ra^2}{\Phi(\la R\ra^2)} \Phi(\la v\ra^2).
\]
The prefactor $\frac{\la R\ra^2}{\Phi(\la R\ra^2)}$ tends to zero by \eqref{eq/ Phi sup}. Applying Proposition \ref{prop/ super} gives \eqref{eq/ E tail} and \eqref{eq/ bdy tail}. This completes the proof.
\end{proof}

%
%
%
%
%
%
%
%
%
%
%

\section{Construction of a renormalized solution in the inflow case}\label{sec/ inflow limit}

We first treat the inflow boundary condition. The compactness arguments for the interior terms are common to \textbf{(In)} and \textbf{(Max)}, whereas the passage to the limit in the boundary terms requires different trace estimates. We thus carry out the interior limit and the inflow boundary argument in this section. The modifications required for \textbf{(Max)} will be given in the next section.

By \eqref{eq/ bdd}, the family $(f_\e)$ is uniformly integrable on bounded velocity sets. Moreover, the second-moment bound gives
\[
\sup_{\e>0} \int_0^T\iint_{I\times\{|v|>R\}} f_\e\,\dx\dv\dt \le \frac{C_T}{\la R\ra^2} \xrightarrow[R\to\infty]{}0.
\]
Hence the Dunford--Pettis theorem yields, after passing to a subsequence,
\bq\label{eq/ f weak}
f_\e\weakto f \quad\text{weakly in }L^1((0,T)\times\calO).
\eq
By a diagonal extraction in $T$, we may assume that the same convergence holds weakly in $L^1_{\rm loc}((0,\infty);L^1(\calO))$. In the remainder of the proof, all further extractions are understood to be taken from this subsequence.

For $\delta>0$, define
\[
\beta_\delta(z):=\frac{z}{1+\delta z}.
\]
For every fixed $\e,\delta>0$, the collision frequency $\lambda_\e$ is bounded. The classical renormalization and trace theory for transport equations with integrable source terms \cite{DL89O,Mis00,Mis10} gives
\bq\label{eq/ betadelta}
\p_t\beta_\delta(f_\e) +v_1\p_x\beta_\delta(f_\e) = \lambda_\e(\rho_{f_\e}) \beta_\delta'(f_\e) (M[f_\e]-f_\e )
\eq
in the renormalized sense, with spatial traces
\[
\beta_\delta(f_\e)_l=\beta_\delta(f_{\e l}), \quad \beta_\delta(f_\e)_r=\beta_\delta(f_{\e r}),
\]
and temporal traces
\[
\begin{aligned}
\beta_\delta(f_\e)|_{t=0} =\beta_\delta(f_{\e0}), \quad \beta_\delta(f_\e)|_{t=s}=\beta_\delta(f_\e(s)) \quad (s>0).
\end{aligned}
\]
In particular, for every $t>0$ and $\varphi\in\calD([0,t]\times\ov{\calO})$, the inflow boundary condition gives
\bq\label{eq/ reg green}
\begin{aligned}
&\iint_{\calO} \beta_\delta(f_\e(t))\varphi(t)\,\dx\dv - \iint_{\calO} \beta_\delta(f_{\e0})\varphi(0)\,\dx\dv\\
&\quad+ \int_0^t\int_{v_1>0} \beta_\delta(f_{\e r}) \varphi(s,\frac12,v)v_1\,\dv\ds + \int_0^t\int_{v_1<0} \beta_\delta(g_r) \varphi(s,\frac12,v)v_1\,\dv\ds\\
&\quad- \int_0^t\int_{v_1<0} \beta_\delta(f_{\e l}) \varphi(s,-\frac12,v)v_1\,\dv\ds - \int_0^t\int_{v_1>0} \beta_\delta(g_l) \varphi(s,-\frac12,v)v_1\,\dv\ds\\
&\quad \quad = \int_0^t\iint_{\calO} \beta_\delta(f_\e) (\p_s+v_1\p_x)\varphi\,\dx\dv\ds + \int_0^t\iint_{\calO} Q_{\e,\delta}\varphi\,\dx\dv\ds,
\end{aligned}
\eq
where
\bq\label{eq/ Q reg}
Q_{\e,\delta} := \beta_\delta'(f_\e) \lambda_\e(\rho_{f_\e}) (M[f_\e]-f_\e ).
\eq

We first pass to the limit in the interior terms, then remove the renormalization parameter $\delta$, and finally return to the boundary terms.

%
%
%
%
%
%
%
%
%
%
%

\subsection{Interior compactness and the limit $\e\to0$}\label{subsec/ int lim}

For fixed $\delta>0$,
\[
0\le\beta_\delta(f_\e) \le\min\lt\{f_\e,\frac1\delta\rt\},
\]
thus $(\beta_\delta(f_\e))$ is bounded in $L^1\cap L^\infty((0,T)\times\calO)$.

We first obtain a local $L^1$ bound for the source term in \eqref{eq/ betadelta}. Let $R>0$. Splitting the domain according to whether $M[f_\e]\ge3f_\e$, we have
\begin{align*}
&\iiint_{(0,T)\times I\times B_R} \lambda_\e(\rho_{f_\e}) \beta_\delta'(f_\e) |M[f_\e]-f_\e| \,\dt\dx\dv \\
&\quad \le \frac{1}{\log3} \iiint_{\{(t,x,v)\in(0,T)\times\calO:M[f_\e]\ge3f_\e\}} \lambda_\e(\rho_{f_\e}) (M[f_\e]-f_\e) \log\frac{M[f_\e]}{f_\e} \,\dt\dx\dv \\
&\quad \quad+ 3 \iiint_{(0,T)\times I\times B_R} \lambda_\e(\rho_{f_\e}) \frac{f_\e}{(1+\delta f_\e)^2} \,\dt\dx\dv.
\end{align*}
For the second term, since $\frac{f_\e}{(1+\delta f_\e)^2} \le\frac1\delta$ and $\lambda_\e(\rho_{f_\e})\le\rho_{f_\e}$,
\begin{align*}
\iiint_{(0,T)\times I\times B_R} \lambda_\e(\rho_{f_\e}) \frac{f_\e}{(1+\delta f_\e)^2} \,\dt\dx\dv \le \frac{|B_R|}{\delta} \int_0^T\int_I\rho_{f_\e}\,\dx\dt \le C_{T,R,\delta}.
\end{align*}
Together with \eqref{eq/ reg diss}, this gives
\[
\sup_{\e>0} \iiint_{(0,T)\times I\times B_R} |Q_{\e,\delta}|\,\dt\dx\dv \le C_{T,R,\delta}.
\]

For fixed $\delta>0$, the velocity averaging theorem applied after localization in $(t,x)$ gives relative compactness of
\bq\label{eq/ beta avg}
\lt( \int_{\R^3}\beta_\delta(f_\e)\varphi(v)\,\dv \rt)_\e
\eq
in $L^1_{\rm loc}((0,T)\times I)$ for every $\varphi\in\calD(\R^3)$; see \cite[Section III, Theorem 6]{DLM91}. In the present slab geometry, one may first integrate against $\varphi$ in the transverse velocity variables and then apply the one-dimensional transport result.

To extend the compactness to the whole cylinder, we use
\[
\lt| \int_{\R^3}\beta_\delta(f_\e)\varphi(v)\,\dv \rt| \le \frac{\|\varphi\|_{L^1(\R^3)}}{\delta}.
\]
For fixed $\delta$, the integrals over the time and spatial boundary layers tend to zero uniformly in $\e$ as the measures of these layers tend to zero. An increasing sequence of compact subsets covering $(0,T)\times I$ then gives relative compactness in $L^1((0,T)\times I)$.

We next transfer this compactness to $f_\e$. Since $0\le\beta_\delta(z)\le z$ for $z\ge0$, for every $L>1$ we have
\bq\label{eq/ beta err}
|\beta_\delta(f_\e)-f_\e| = f_\e-\beta_\delta(f_\e) = \frac{\delta f_\e^2}{1+\delta f_\e} \le \delta L f_\e + f_\e\one_{\{f_\e>L\}} \le \delta L f_\e + \frac1{\log L} f_\e|\log f_\e|. 
\eq
Integrating \eqref{eq/ beta err} and using \eqref{eq/ bdd}, we obtain
\[
\sup_{\e>0} \operatorname*{ess\,sup}_{0<t<T} \|\beta_\delta(f_\e(t))-f_\e(t)\|_{L^1(\calO)} \le C_T\lt( \delta L+\frac1{\log L} \rt).
\]
For fixed $L>1$, letting $\delta\to0$ and then $L\to\infty$ yields
\bq\label{eq/ beta err unif}
\lim_{\delta\to0} \sup_{\e>0} \operatorname*{ess\,sup}_{0<t<T} \|\beta_\delta(f_\e(t))-f_\e(t)\|_{L^1(\calO)} =0.
\eq
For fixed $\varphi\in\calD(\R^3)$, \eqref{eq/ beta avg} and \eqref{eq/ beta err unif} imply that
\[
\lt( \int_{\R^3}f_\e\varphi(v)\,\dv \rt)_\e
\]
is relatively compact in $L^1((0,T)\times I)$. Since its weak limit is uniquely determined by \eqref{eq/ f weak}, we conclude that
\bq\label{eq/ avg}
\int_{\R^3}f_\e\varphi(v)\,\dv \longrightarrow  \int_{\R^3}f\varphi(v)\,\dv \quad\text{strongly in }L^1((0,T)\times I)
\eq
for every $\varphi\in\calD(\R^3)$.

The velocity averages in \eqref{eq/ avg} involve compactly supported weights. The uniform integrability obtained in Section \ref{sec/ tails} allows us to recover the full kinetic energy.

\begin{lemma}\label{lem/ mom2}
Let $(f_\e)$ be the approximate solutions of \eqref{eq/ regbgk}, and let $f$ be their weak limit in \eqref{eq/ f weak}. For every $T>0$, 
\bq\label{eq/ mom2 conv}
\int_{\R^3}\la v\ra^2f_\e\,\dv \longrightarrow  \int_{\R^3}\la v\ra^2f\,\dv \quad\text{strongly in }L^1((0,T)\times I).
\eq
\end{lemma}

\begin{proof}
We first show that the limiting function inherits the de la Vall\'ee Poussin bound. For every $L>0$, $\Phi(\la v\ra^2)\one_{\{|v|\le L\}} \in L^\infty(\R^3)$, and thus the weak convergence \eqref{eq/ f weak} gives
\begin{align*}
\iiint_{(0,T)\times I\times\{|v|\le L\}} \Phi(\la v\ra^2)f\,\dt\dx\dv & = \lim_{\e\to0} \iiint_{(0,T)\times I\times\{|v|\le L\}} \Phi(\la v\ra^2)f_\e\,\dt\dx\dv\\
& \le \liminf_{\e\to0} \iiint_{(0,T)\times\calO} \Phi(\la v\ra^2)f_\e\,\dt\dx\dv \le C_T.
\end{align*}
Letting $L\to\infty$ by monotone convergence yields
\bq\label{eq/ f too}
\iiint_{(0,T)\times\calO} \Phi(\la v\ra^2)f\,\dt\dx\dv \le C_T.
\eq

Let $\chi_R\in C_c^\infty(\R^3)$ satisfy $\one_{\{|v|\le R\}}\le\chi_R\le\one_{\{|v|\le2R\}}$. We write
\bq\label{eq/ v2 split}
\int_{\R^3}\la v\ra^2(f_\e-f)\,\dv = \int_{\R^3} \chi_R(v)\la v\ra^2(f_\e-f)\,\dv + \int_{\R^3} (1-\chi_R(v)) \la v\ra^2(f_\e-f)\,\dv.
\eq
For fixed $R$, the first term converges strongly to zero in $L^1((0,T)\times I)$ by \eqref{eq/ avg}. For the second term, the monotonicity of $z\mapsto\frac{\Phi(z)}z$ gives
\[
\lt| \int_{\R^3} (1-\chi_R(v)) \la v\ra^2(f_\e-f)\,\dv \rt| \le \int_{|v|>R} \la v\ra^2(f_\e+f)\,\dv \le \frac{\la R\ra^2}{\Phi(\la R\ra^2)} \int_{\R^3} \Phi(\la v\ra^2)(f_\e+f)\,\dv.
\]
By Proposition \ref{prop/ super} and \eqref{eq/ f too}, the last integral is bounded uniformly in $\e$ in $L^1((0,T)\times I)$, while
\[
\frac{\la R\ra^2}{\Phi(\la R\ra^2)} \xrightarrow[R\to\infty]{}0.
\]
We first choose $R$ large and then let $\e\to0$ in \eqref{eq/ v2 split}. This proves \eqref{eq/ mom2 conv}. This completes the proof.
\end{proof}

We next identify the macroscopic limits.

\begin{lemma}\label{lem/ max L1}
Along a further subsequence, we have
\bq\label{eq/ ae}
\begin{aligned}
\rho_{f_\e} &\longrightarrow \rho_f &&\text{a.e. in }(0,T)\times I,\\
\lambda_\e(\rho_{f_\e}) &\longrightarrow \rho_f &&\text{a.e. in }(0,T)\times I,\\
u_{f_\e} &\longrightarrow  u_f &&\text{a.e. on }\{\rho_f>0\},\\
\theta_{f_\e} &\longrightarrow \theta_f &&\text{a.e. on }\{\rho_f>0\}.
\end{aligned}
\eq
Moreover,
\begin{align}
\lambda_\e(\rho_{f_\e}) &\longrightarrow \rho_f \quad \text{strongly in }L^1((0,T)\times I), \label{eq/ rate str}\\
M[f_\e] &\longrightarrow  M[f] \quad \text{strongly in } L^1((0,T)\times\calO, \la v\ra^2\,\dt\dx\dv). \label{eq/ M str}
\end{align}
\end{lemma}

\begin{proof}
Set
\[
\rho_\e:=\rho_{f_\e}, \quad j_\e:=\int_{\R^3}vf_\e\,\dv, \quad e_\e:=\int_{\R^3}|v|^2f_\e\,\dv,
\]
and define $\rho$, $j$, and $e$ analogously from $f$.

The convergence \eqref{eq/ avg}, combined with the second-moment bounds, gives
\bq\label{eq/ rho j strong}
\rho_\e\to\rho, \quad j_\e\to j \quad\text{strongly in }L^1((0,T)\times I).
\eq
Indeed, after introducing a radial velocity cutoff, the tails are controlled by
\[
\int_{|v|>R}f_\e\,\dv \le \frac1{R^2} \int_{\R^3}|v|^2f_\e\,\dv, \quad \text{and} \quad \int_{|v|>R}|v|f_\e\,\dv \le \frac1R \int_{\R^3}|v|^2f_\e\,\dv,
\]
and the same estimates hold for $f$. Furthermore, \eqref{eq/ mom2 conv} and the first convergence in \eqref{eq/ rho j strong} yield
\[
e_\e \to  e \quad\text{strongly in }L^1((0,T)\times I).
\]

After passing to a further subsequence, all three convergences hold almost everywhere. On $\{\rho>0\}$,
\[
u_{f_\e} = \frac{j_\e}{\rho_\e} \longrightarrow  \frac{j}{\rho} = u_f.
\]

To treat the temperature, define
\[
q_\e :=
\begin{cases}
\displaystyle\frac{|j_\e|^2}{\rho_\e}, &\rho_\e>0,\\[6pt]
0,&\rho_\e=0.
\end{cases}
\]
By Cauchy--Schwarz,
\bq\label{eq/ q energy}
0\le q_\e\le e_\e.
\eq
On $\{\rho>0\}$,
\[
q_\e \longrightarrow  \frac{|j|^2}{\rho}.
\]
On $\{\rho=0\}$, we have $f=0$ for almost every velocity and hence $e=0$. Since $e_\e\to e$ almost everywhere after the above extraction, \eqref{eq/ q energy} gives $q_\e\to0$ there. Thus $q_\e$ converges almost everywhere to
\[
q:=
\begin{cases}
\displaystyle\frac{|j|^2}{\rho}, &\rho>0,\\[6pt]
0,&\rho=0.
\end{cases}
\]
Since $(e_\e)$ is strongly convergent in $L^1$, it is uniformly integrable. The bound \eqref{eq/ q energy} and Vitali's theorem then give
\[
q_\e \longrightarrow  q \quad\text{strongly in }L^1((0,T)\times I).
\]
Using $3\rho_\e\theta_{f_\e} = e_\e-q_\e$ and $3\rho\theta_f = e-q$, we obtain
\[
\rho_\e\theta_{f_\e} \longrightarrow  \rho\theta_f \quad\text{strongly in }L^1((0,T)\times I).
\]
In particular,
\[
\theta_{f_\e}\longrightarrow \theta_f \quad\text{a.e. on }\{\rho>0\}.
\]
This proves the macroscopic convergences in \eqref{eq/ ae}.

Since $\lambda_\e(\rho_\e) = \frac{\rho_\e}{1+\e\rho_\e}$, the almost everywhere convergence $\rho_\e\to\rho$ gives
\[
\lambda_\e(\rho_\e)\longrightarrow \rho \quad\text{a.e. in }(0,T)\times I.
\]
Moreover,
\[
0\le\lambda_\e(\rho_\e)\le\rho_\e.
\]
The strong $L^1$ convergence of $\rho_\e$ implies uniform integrability, and hence Vitali's theorem gives \eqref{eq/ rate str}.

It remains to prove \eqref{eq/ M str}. Since $f(t,x,\cdot)$ is an integrable function, $\rho>0$ implies $\theta_f>0$: otherwise $f$ would be supported at the single velocity $u_f$. On $\{\rho>0\}$, the almost everywhere convergence of $\rho_\e$, $u_{f_\e}$, and $\theta_{f_\e}$ yields
\bq\label{eq/ novac ae}
M[f_\e](t,x,v) \longrightarrow  M[f](t,x,v)
\eq
for almost every $(t,x,v)$ with $\rho(t,x)>0$.

Let
\[
\calV_0:=\{(t,x)\in(0,T)\times I:\rho(t,x)=0\}.
\]
Since $M[f]=0$ on $\calV_0\times\R^3$,
\[
\iiint_{\calV_0\times\R^3} |M[f_\e]-M[f]|\,\dt\dx\dv = \iiint_{\calV_0\times\R^3} M[f_\e]\,\dt\dx\dv = \iint_{\calV_0}\rho_\e\,\dt\dx \longrightarrow 0.
\]
After a further extraction,
\[
M[f_\e]\longrightarrow 0=M[f] \quad\text{a.e. on }\calV_0\times\R^3.
\]
Together with \eqref{eq/ novac ae}, this gives
\bq\label{eq/ max ae}
M[f_\e]\longrightarrow  M[f] \quad\text{a.e. in }(0,T)\times\calO.
\eq

Finally, the Maxwellian moment identities and Lemma \ref{lem/ mom2} give
\begin{align*}
\iiint_{(0,T)\times\calO} \la v\ra^2M[f_\e]\,\dt\dx\dv &= \iiint_{(0,T)\times\calO} \la v\ra^2f_\e\,\dt\dx\dv \\
&\longrightarrow  \iiint_{(0,T)\times\calO} \la v\ra^2f\,\dt\dx\dv = \iiint_{(0,T)\times\calO} \la v\ra^2M[f]\,\dt\dx\dv.
\end{align*}
Applying Scheff\'e's lemma with respect to the measure $\la v\ra^2\,\dt\dx\dv$ and using \eqref{eq/ max ae}, we obtain \eqref{eq/ M str}. This completes the proof.
\end{proof}

We now fix the subsequences used in the remaining argument. Let $(\delta_j)$ be a decreasing sequence such that $\delta_j\downarrow0$. By a diagonal extraction in $\e$, we may assume that, for every $\delta=\delta_j$, the following weak$\star$ limits exist along the same $\e$-subsequence:
\[
\beta_\delta(f_\e)  \weakstarto\ov{\beta}_\delta,\quad  \beta_\delta'(f_\e)  \weakstarto\frka_\delta,\quad  f_\e\beta_\delta'(f_\e)  \weakstarto\frkb_\delta \quad\text{in }   L^\infty((0,\infty)\times \calO) .
\]
For the first and third sequences the convergence also holds weakly in $L^1$ on every finite time interval whenever needed. The weak$\star$ compactness follows from
\[
0\le\beta_\delta(f_\e)\le\frac1\delta, \quad 0\le\beta_\delta'(f_\e)\le1, \quad 0\le f_\e\beta_\delta'(f_\e)\le\frac1\delta.
\]
In what follows, $\delta\to0$ always refers to the fixed sequence $(\delta_j)$.

We shall also use the entropy-production measures. By \eqref{eq/ reg diss}, after a further diagonal extraction,
\bq\label{eq/ mu}
D_\e[f_\e]\,\dt\dx\dv \weakstarto\nu \quad\text{in }\calM_{\rm loc} ([0,\infty)\times\ov{\calO}),
\eq
where
\[
D_\e[f_\e] = \lambda_\e(\rho_{f_\e}) (f_\e-M[f_\e]) (\log f_\e-\log M[f_\e]).
\]
The measure $\nu$ is nonnegative.

We next identify the limit of the collision term in \eqref{eq/ betadelta}.

\begin{lemma}\label{lem/ Q weak}
For every fixed $\delta=\delta_j$, the family $(Q_{\e,\delta})_\e$ defined in \eqref{eq/ Q reg} is weakly compact in $L^1_{\rm loc}([0,T]\times\ov{\calO})$.
\end{lemma}

\begin{proof}
We use the entropy production splitting of \cite[Proposition 3.1]{GSR05}. Fix $L>0$ and set $K_L=[0,T]\times\ov I\times\ov B_L$. For $R>2$, write
\[
Q_{\e,\delta} = Q_{\e,\delta}\one_{\{M[f_\e]>Rf_\e\}} + Q_{\e,\delta}\one_{\{M[f_\e]\le Rf_\e\}}.
\]
On $\{M[f_\e]>Rf_\e\}$, we find
\[
\log\frac{M[f_\e]}{f_\e}\ge\log R,
\]
and hence
\[
\iiint_{K_L} |Q_{\e,\delta}| \one_{\{M[f_\e]>Rf_\e\}} \,\dt\dx\dv \le \frac1{\log R} \int_0^T\scrD_\e[f_\e]\,\dt \le \frac{C_T}{\log R}.
\]
On the complementary set, $|M[f_\e]-f_\e| \le(R-1)f_\e$ due to $R>2$. Hence
\[
|Q_{\e,\delta}| \one_{\{M[f_\e]\le Rf_\e\}} \le \frac{R}{\delta} \rho_{f_\e}.
\]
Consequently, for every compact $K\subset[0,T]\times\ov{\calO}$ and every measurable $A\subset K$,
\bq\label{eq/ Q UI}
\sup_{\e>0}\iiint_A|Q_{\e,\delta}|\,\dt\dx\dv \le \frac{C_T}{\log R} + \frac{R}{\delta} \sup_{\e>0}\iiint_A\rho_{f_\e}\,\dt\dx\dv.
\eq

The functions $\rho_{f_\e}$ are strongly convergent in $L^1((0,T)\times I)$ and hence uniformly integrable there. Their lifts to a fixed compact velocity set are also uniformly integrable. Indeed, for $K\subset(0,T)\times I\times B_L$ and $N>0$,
\[
\iiint_A\rho_{f_\e}\,\dt\dx\dv \le N|A| + |B_L| \iint_{\{(t,x)\in(0,T)\times I:\rho_{f_\e}>N\}}\rho_{f_\e}\,\dt\dx.
\]
We first choose $R$ large in \eqref{eq/ Q UI}, then $N$ large, and finally $|A|$ small. This proves the local uniform integrability of $(Q_{\e,\delta})_\e$, and the Dunford--Pettis theorem gives the claim. This completes the proof.
\end{proof}

We next establish separate compactness of the gain and loss terms.

\begin{lemma}\label{lem/ split weak}
Define
\[
Q_{\e,\delta}^+ := \beta_\delta'(f_\e) \lambda_\e(\rho_{f_\e})M[f_\e], \quad Q_{\e,\delta}^- := \beta_\delta'(f_\e) \lambda_\e(\rho_{f_\e})f_\e.
\]
For every fixed $\delta=\delta_j$, both $(Q_{\e,\delta}^+)_\e$ and $(Q_{\e,\delta}^-)_\e$ are weakly compact in $L^1_{\rm loc}([0,T]\times\ov{\calO})$.
\end{lemma}

\begin{proof}
Since
\[
0\le f_\e\beta_\delta'(f_\e) = \frac{f_\e}{(1+\delta f_\e)^2} \le\frac1\delta,
\]
we obtain
\[
0\le Q_{\e,\delta}^- \le \frac1\delta \lambda_\e(\rho_{f_\e}).
\]
The strong convergence \eqref{eq/ rate str} implies local uniform integrability of the right-hand side after lifting it to compact velocity sets. Hence $(Q_{\e,\delta}^-)_\e$ is locally weakly compact.

Since $Q_{\e,\delta} = Q_{\e,\delta}^+-Q_{\e,\delta}^-$, we also have
\[
0\le Q_{\e,\delta}^+ \le |Q_{\e,\delta}|+Q_{\e,\delta}^-.
\]
Lemma \ref{lem/ Q weak} then yields the weak compactness of the gain term. This completes the proof.
\end{proof}

\begin{proposition}\label{prop/ qplus}
For every fixed $\delta=\delta_j$,
\bq\label{eq/ qplus weak}
Q_{\e,\delta}^+ \weakto \rho_f\frka_\delta M[f] \quad\text{weakly in } L^1_{\rm loc}([0,T]\times\ov{\calO}).
\eq

\end{proposition}

\begin{proof}
Let
\[
G_\e := \lambda_\e(\rho_{f_\e})M[f_\e], \quad G:=\rho_fM[f].
\]
Fix a compact set $K\subset[0,T]\times\ov{\calO}$. Along the subsequence fixed above, Lemma \ref{lem/ max L1} yields
\[
G_\e\xrightarrow[\e\to0]{}G\quad\text{almost everywhere on }K.
\]
For $R>0$, set
\[
G_{\e,R}:=\min\{G_\e,R\}, \quad G_R:=\min\{G,R\}.
\]
Bounded convergence and the weak$\star$ convergence of $\beta_\delta'(f_\e)$ give
\bq\label{eq/ gain trunc}
\beta_\delta'(f_\e)G_{\e,R} \weakto \frka_\delta G_R \quad\text{weakly in }L^1(K).
\eq

Since $0\le\beta_\delta'(f_\e)G_{\e,R}\le Q_{\e,\delta}^+$, Lemma \ref{lem/ split weak} and \eqref{eq/ gain trunc} yield
\[
\iiint_K\frka_\delta G_R\,\dt\dx\dv\le C_{K,\delta}.
\]
Monotone convergence as $R\to\infty$ gives $\frka_\delta G\in L^1(K)$ and $\frka_\delta G_R\to\frka_\delta G$ in $L^1(K)$.

We compare the truncated product with $Q_{\e,\delta}^+$. Since $0\le\beta_\delta'\le1$,
\[
0 \le Q_{\e,\delta}^+ - \beta_\delta'(f_\e)G_{\e,R} = \beta_\delta'(f_\e)(G_\e-R)_+ \le Q_{\e,\delta}^+ \one_{\{G_\e>R\}}.
\]
The family $(Q_{\e,\delta}^+)_\e$ is uniformly integrable on $K$. Moreover,
\[
\{G_\e>R\} \subset \lt\{G>\frac R2\rt\} \cup \lt\{|G_\e-G|>\frac R2\rt\}.
\]
Since $G$ is finite almost everywhere (because both $\rho_f$ and $M[f]$ are $L^1_{\rm loc}((0,\infty)\times \ov{\calO}$)) and $G_\e\to G$ almost everywhere on $K$, the above inclusion gives
\[
\lim_{R\to\infty}\limsup_{\e\to0}|K\cap\{G_\e>R\}|=0.
\]
Consequently,
\bq\label{eq/ gain err}
\lim_{R\to\infty} \limsup_{\e\to0} \iiint_K \lt| Q_{\e,\delta}^+ - \beta_\delta'(f_\e)G_{\e,R} \rt|  \dt\dx\dv =0.
\eq

Combining the strong convergence $\frka_\delta G_R \to \frka_\delta G$ with \eqref{eq/ gain trunc} and \eqref{eq/ gain err} proves \eqref{eq/ qplus weak}. This completes the proof.
\end{proof}

\begin{proposition}\label{prop/ qminus}
For every fixed $\delta=\delta_j$,
\[
Q_{\e,\delta}^- \weakto \rho_f\frkb_\delta \quad\text{weakly in } L^1_{\rm loc}([0,T]\times\ov{\calO}).
\]
\end{proposition}

\begin{proof}
We write $Q_{\e,\delta}^- = (f_\e\beta_\delta'(f_\e) ) \lambda_\e(\rho_{f_\e})$. On every compact velocity set, $f_\e\beta_\delta'(f_\e)$ converges weakly$\star$ in $L^\infty$ to $\frkb_\delta$, while $\lambda_\e(\rho_{f_\e})$ converges strongly in $L^1$ to $\rho_f$ by \eqref{eq/ rate str}. The product convergence gives the claim. This completes the proof.
\end{proof}

We may now pass to the limit in the interior Green formula. Since
\[
\beta_\delta(f_{\e0}) \longrightarrow  \beta_\delta(f_0) \quad\text{strongly in }L^1(\calO)
\]
and Propositions \ref{prop/ qplus}--\ref{prop/ qminus} identify the collision term, we obtain
\bq\label{eq/ ovbeta}
\p_t\ov{\beta}_\delta +v_1\p_x\ov{\beta}_\delta = \rho_f ( \frka_\delta M[f]-\frkb_\delta  ) \quad\text{in }\calD'((0,T)\times\calO),
\eq
with initial trace
\[
\ov{\beta}_\delta|_{t=0} = \beta_\delta(f_0).
\]

%
%
%
%
%
%
%
%
%
%
%

\subsection{Passage to the limit $\delta\to0$}\label{subsec/ delta limit}
A direct passage to the limit $\delta\to0$ in \eqref{eq/ ovbeta} is not suitable, since the limiting collision term need not be locally integrable without an additional renormalization. We thus set
\[
\gamma(z):=\log(1+z).
\]
Since the right-hand side of \eqref{eq/ ovbeta} belongs to $L^1_{\rm loc}$ for fixed $\delta$, the renormalization theorem for transport equations \cite{Mis10} gives
\bq\label{eq/ beta twice}
\p_t\gamma(\ov{\beta}_\delta) +v_1\p_x\gamma(\ov{\beta}_\delta) = \frkQ_\delta \quad\text{in }\calD'((0,T)\times\calO),
\eq
where
\[
\frkQ_\delta := \gamma'(\ov{\beta}_\delta) \rho_f ( \frka_\delta M[f]-\frkb_\delta  ).
\]

We first identify the limits of $\ov{\beta}_\delta$, $\frka_\delta$, and $\frkb_\delta$.

\begin{lemma}\label{lem/ n delta}
Let $(X,\mu)$ be a $\sigma$-finite measure space. Suppose $g^n,\eta_\delta^n,\ov\eta_\delta,g\in L^1(X)$ satisfy
\[
g^n\weakto g, \quad \eta_\delta^n\weakto\ov\eta_\delta \quad\text{weakly in }L^1(X), \quad \text{and} \quad \lim_{\delta\to0} \sup_{n \in \bbN} \|\eta_\delta^n-g^n\|_{L^1(X)} =0.
\]
Then, we have
\[
\ov\eta_\delta \longrightarrow  g \quad\text{strongly in }L^1(X).
\]

\end{lemma}

\begin{proof}
Since the $L^1$-norm is weakly lower semicontinuous,
\[
\|\ov\eta_\delta-g\|_{L^1(X)} \le \liminf_{n\to\infty} \|\eta_\delta^n-g^n\|_{L^1(X)} \le \sup_{n \in \bbN} \|\eta_\delta^n-g^n\|_{L^1(X)}.
\]
The right-hand side tends to zero as $\delta\to0$. This completes the proof.
\end{proof}

\begin{lemma}\label{lem/ delta lim}
For every $T>0$, as $\delta\downarrow0$ along the fixed sequence $(\delta_j)$,
\begin{align}
\ov\beta_\delta&\nearrow f \quad \text{strongly in }L^1((0,T)\times\calO), \nonumber \\
\frka_\delta&\nearrow1 \quad \text{strongly in }L^1_{\rm loc}([0,T]\times\ov\calO),\label{eq/ a to 1}\\
\frkb_\delta&\nearrow f \quad \text{strongly in }L^1((0,T)\times\calO).\label{eq/ b to f}
\end{align}
\end{lemma}

\begin{proof}
Since $0\le\beta_\delta(f_\e)\le f_\e$, the family $(\beta_\delta(f_\e))_\e$ is weakly compact in $L^1((0,T)\times\calO)$. Its weak $L^1$ limit coincides with $\ov{\beta}_\delta$ by uniqueness of distributional limits.

By \eqref{eq/ beta err unif}, the approximation is uniform in $\e$ in $L^1((0,T)\times\calO)$. Thus Lemma \ref{lem/ n delta} gives
\[
\ov{\beta}_\delta \longrightarrow  f \quad\text{strongly in }L^1((0,T)\times\calO).
\]

For the monotonicity, if $0<\delta_1<\delta_2$, then $\beta_{\delta_1}(f_\e) \ge \beta_{\delta_2}(f_\e)$. Passing to the weak limit gives $\ov{\beta}_{\delta_1} \ge \ov{\beta}_{\delta_2}$ almost everywhere. Hence $\ov{\beta}_\delta\nearrow f$ as $\delta\downarrow0$.

For the first convergence,
\[
\lt| \beta_\delta'(f_\e)-1 \rt| = 1-\frac1{(1+\delta f_\e)^2} = \frac{ 2\delta f_\e+\delta^2f_\e^2 }{ (1+\delta f_\e)^2 } \le 3\delta f_\e.
\]
Thus, on every finite time interval,
\[
\sup_{\e>0} \iiint_{(0,T)\times\calO} |\beta_\delta'(f_\e)-1| \,\dt\dx\dv \le C_T\delta.
\]
Applying Lemma \ref{lem/ n delta} on compact subsets, with $g^\e\equiv1$, gives \eqref{eq/ a to 1}.

For $\frkb_\delta$, let $L>1$. We compute
\[
|f_\e\beta_\delta'(f_\e)-f_\e| = \frac{ 2\delta f_\e^2+\delta^2f_\e^3 }{ (1+\delta f_\e)^2 } \le (2\delta L+\delta^2L^2)f_\e + 3f_\e\one_{\{f_\e>L\}} \le (2\delta L+\delta^2L^2)f_\e + \frac{3}{\log L} f_\e|\log f_\e|.
\]
Using \eqref{eq/ bdd}, we first choose $L$ large and then $\delta$ small to obtain
\[
\lim_{\delta\to0} \sup_{\e>0} \|f_\e\beta_\delta'(f_\e)-f_\e\|_ {L^1((0,T)\times\calO)} =0.
\]
Lemma \ref{lem/ n delta} gives \eqref{eq/ b to f}.

Finally, both $\beta_\delta'(z)$ and $z\beta_\delta'(z)$ increase pointwise as $\delta\downarrow0$. Passing to the weak limits as above shows that $\frka_\delta$ and $\frkb_\delta$ have the same monotonicity. This completes the proof.
\end{proof}

We now identify the limit of the collision term in \eqref{eq/ beta twice}. Write
\[
\frkQ_\delta^+ := \gamma'(\ov{\beta}_\delta) \rho_f\frka_\delta M[f]  \quad \text{and} \quad \frkQ_\delta^- := \gamma'(\ov{\beta}_\delta) \rho_f\frkb_\delta.
\]

\begin{lemma}\label{lem/ delta dom}
Let $\nu=F_\nu\,\dt\dx\dv+\nu_s$ be the Lebesgue decomposition of the measure in \eqref{eq/ mu}, with $0 \le F_\nu \in L^1_{\rm loc}([0,\infty)\times \ov{\calO}$). Then for every $\delta>0$,
\bq\label{eq/ Q dom}
0\le\frkQ_\delta^-\le\rho_f,\quad 0\le\frkQ_\delta^+\le2\frkQ_\delta^-+\frac2{\log2}F_\nu.
\eq
Consequently, both $(\frkQ_\delta^-)_\delta$ and $(\frkQ_\delta^+)_\delta$ are locally uniformly integrable.
\end{lemma}

\begin{proof}
For every $\e$,
\[
f_\e\beta_\delta'(f_\e) = \frac{f_\e}{(1+\delta f_\e)^2} \le \frac{f_\e}{1+\delta f_\e} = \beta_\delta(f_\e).
\]
Passing to the weak$\star$ limit gives $0\le\frkb_\delta\le\ov{\beta}_\delta$ almost everywhere. Since $\gamma'(\ov{\beta}_\delta) = \frac1{1+\ov{\beta}_\delta}$, we obtain
\[
\frkQ_\delta^- = \rho_f \frac{\frkb_\delta} {1+\ov{\beta}_\delta} \le\rho_f.
\]
On compact velocity sets, $\rho_f$ is integrable, which proves the local uniform integrability.

The gain term is controlled by the entropy production. We first derive an estimate at the $\e$-level. On $\{M[f_\e]<2f_\e\}$, we get $Q_{\e,\delta}^+ \le 2Q_{\e,\delta}^-$. On $\{M[f_\e]\ge2f_\e\}$, we estimate
\[
D_\e[f_\e] = \lambda_\e(\rho_{f_\e}) (M[f_\e]-f_\e) \log\frac{M[f_\e]}{f_\e} \ge \frac{\log2}{2} \lambda_\e(\rho_{f_\e}) M[f_\e].
\]
Since $\beta_\delta'(f_\e)\le1$, this yields
\[
Q_{\e,\delta}^+ \one_{\{M[f_\e]\ge2f_\e\}} \le \frac{2}{\log2} D_\e[f_\e].
\]
Hence
\bq\label{eq/ qvd est}
Q_{\e,\delta}^+ \le 2Q_{\e,\delta}^- + \frac{2}{\log2} D_\e[f_\e].
\eq

Passing to the limit $\e\to0$ in \eqref{eq/ qvd est}, using Propositions \ref{prop/ qplus}--\ref{prop/ qminus} and \eqref{eq/ mu}, gives the inequality of nonnegative Radon measures
\[
\rho_f\frka_\delta M[f]\,\dt\dx\dv \le 2\rho_f\frkb_\delta\,\dt\dx\dv + \frac{2}{\log2}\nu.
\]
Comparing the absolutely continuous parts in the Lebesgue decomposition, we obtain
\[
\rho_f\frka_\delta M[f] \le 2\rho_f\frkb_\delta + \frac{2}{\log2}F_\nu \quad\text{a.e.}
\]
Multiplication by
\[
0\le \gamma'(\ov{\beta}_\delta) = \frac1{1+\ov{\beta}_\delta} \le1
\]
and the definition of $\frkQ_\delta^\pm$ give \eqref{eq/ Q dom}. This completes the proof.
\end{proof}

\begin{proposition}\label{prop/ collision delta}
As $\delta\to0$ along $(\delta_j)$,
\begin{align}
\frkQ_\delta^+ &\longrightarrow  \gamma'(f)\rho_fM[f] \quad \text{strongly in }L^1_{\rm loc}([0,T]\times\ov\calO), \label{eq/ Qplus delta}\\
\frkQ_\delta^- &\longrightarrow  \gamma'(f)\rho_ff \quad \text{strongly in }L^1_{\rm loc}([0,T]\times\ov\calO). \label{eq/ Qminus delta}
\end{align}
\end{proposition}

\begin{proof}
Lemma \ref{lem/ delta lim} implies
\[
\ov{\beta}_\delta\to f, \quad \frka_\delta\to1, \quad \frkb_\delta\to f \quad \text{a.e.}
\]
Hence
\[
\frkQ_\delta^+ \longrightarrow  \gamma'(f)\rho_fM[f], \quad \frkQ_\delta^- \longrightarrow  \gamma'(f)\rho_ff
\]
almost everywhere.

Lemma \ref{lem/ delta dom} gives local uniform integrability of both families. Vitali's theorem gives \eqref{eq/ Qplus delta}--\eqref{eq/ Qminus delta}. This completes the proof.
\end{proof}

Passing to the limit $\delta\to0$ in \eqref{eq/ beta twice}, using Lemma \ref{lem/ delta lim} and Proposition \ref{prop/ collision delta}, gives
\[
\p_t\gamma(f) +v_1\p_x\gamma(f) = \gamma'(f)\rho_f(M[f]-f) \quad\text{in }\calD'((0,T)\times\calO).
\]

%
%
%
%
%
%
%
%
%
%
%

\subsection{Boundary traces and conclusion}

We now return to the spatial boundary terms in \eqref{eq/ reg green}. We first extract the limiting outgoing traces.

\begin{lemma}
For every $T>0$, after passing to a further subsequence, there exist nonnegative functions $f_l$ and $f_r$ such that
\begin{align}
f_{\e l} &\weakto f_l \quad\text{weakly in } L^1((0,T)\times\R^3_-, |v_1|\,\dt\dv), \label{eq/ fl weak}\\
f_{\e r} &\weakto f_r \quad\text{weakly in } L^1((0,T)\times\R^3_+, v_1\,\dt\dv). \label{eq/ fr weak}
\end{align}

\end{lemma}

\begin{proof}
The entropy part of \eqref{eq/ out bdd} gives uniform integrability on sets of bounded velocity, while its energy part gives tightness in velocity. The Dunford--Pettis theorem thus yields \eqref{eq/ fl weak}--\eqref{eq/ fr weak}. This completes the proof.
\end{proof}

For every $\delta=\delta_j$, after a diagonal extraction from the same $\e$-subsequence, define
\bq\label{eq/ delta traces}
\beta_\delta(f_{\e l}) \weakstarto\frkl_\delta, \quad \beta_\delta(f_{\e r}) \weakstarto\frkr_\delta
\eq
in the corresponding weighted boundary spaces. Since $0\le\beta_\delta(z)\le z$, these convergences also hold weakly in $L^1$ with respect to the outgoing flux measures.

We first pass to the limit $\e\to0$ for fixed $\delta$. Using \eqref{eq/ delta traces}, Propositions \ref{prop/ qplus}--\ref{prop/ qminus}, and the strong convergence of $\beta_\delta(f_{\e0})$, the Green formula \eqref{eq/ reg green} yields
\bq\label{eq/ ovbeta bdy}
\p_t\ov{\beta}_\delta +v_1\p_x\ov{\beta}_\delta = \rho_f ( \frka_\delta M[f]-\frkb_\delta  )
\eq
up to the spatial boundary, with traces
\[
\begin{aligned}
\ov{\beta}_\delta|_{t=0} &= \beta_\delta(f_0),\\
\ov{\beta}_\delta|_{x=-\frac12,\ v_1<0} &= \frkl_\delta, \quad \ov{\beta}_\delta|_{x=-\frac12,\ v_1>0}  = \beta_\delta(g_l),\\
\ov{\beta}_\delta|_{x=\frac12,\ v_1>0} &= \frkr_\delta, \quad \ov{\beta}_\delta|_{x=\frac12,\ v_1<0}  = \beta_\delta(g_r).
\end{aligned}
\]
Equivalently, \eqref{eq/ ovbeta bdy} is understood through its Green formula for test functions in $\calD([0,T)\times\ov{\calO})$.

Applying the renormalization theorem to $\gamma(\ov{\beta}_\delta)$ gives
\bq\label{eq/ beta twice bdy}
\p_t\gamma(\ov{\beta}_\delta) +v_1\p_x\gamma(\ov{\beta}_\delta) = \frkQ_\delta
\eq
up to the spatial boundary, with traces
\[
\begin{aligned}
\gamma(\ov{\beta}_\delta)|_{t=0} &= \gamma(\beta_\delta(f_0)),\\
\gamma(\ov{\beta}_\delta)|_{x=-\frac12,\ v_1<0} &= \gamma(\frkl_\delta), \quad \gamma(\ov{\beta}_\delta)|_{x=-\frac12,\ v_1>0}  = \gamma(\beta_\delta(g_l)),\\
\gamma(\ov{\beta}_\delta)|_{x=\frac12,\ v_1>0} &= \gamma(\frkr_\delta), \quad \gamma(\ov{\beta}_\delta)|_{x=\frac12,\ v_1<0}  = \gamma(\beta_\delta(g_r)).
\end{aligned}
\]

It remains to identify the limits of the outgoing renormalized traces.

\begin{lemma}\label{lem/ in tr lim}
As $\delta\to0$ along $(\delta_j)$,
\begin{align}
\frkl_\delta &\nearrow f_l \quad \text{strongly in } L^1((0,T)\times\R^3_-, |v_1|\,\dt\dv), \label{eq/ ldelta str}\\
\frkr_\delta &\nearrow f_r \quad \text{strongly in } L^1((0,T)\times\R^3_+, v_1\,\dt\dv). \label{eq/ rdelta str}
\end{align}
\end{lemma}

\begin{proof}
We give the argument at $x=-\frac12$. Since $0\le\beta_\delta(f_{\e l})\le f_{\e l}$ and $(f_{\e l})_\e$ is weakly compact in
\[
L^1((0,T)\times\R^3_-, |v_1|\,\dt\dv),
\]
the sequence $(\beta_\delta(f_{\e l}))_\e$ is weakly compact in the same space. Its weak limit is necessarily $\frkl_\delta$.

For $L>1$,
\[
0\le f_{\e l}-\beta_\delta(f_{\e l}) \le \delta Lf_{\e l} + \frac1{\log L} f_{\e l}|\log f_{\e l}|.
\]
Hence \eqref{eq/ out bdd} gives
\[
\lim_{\delta\to0} \sup_{\e>0} \int_0^T\int_{v_1<0} |f_{\e l}-\beta_\delta(f_{\e l})| |v_1|\,\dv\dt =0.
\]
Applying Lemma \ref{lem/ n delta} with the boundary measure $|v_1|\,\dt\dv$ proves \eqref{eq/ ldelta str}. The proof of \eqref{eq/ rdelta str} is identical.

The monotonicity follows from the pointwise monotonicity of $\beta_\delta$ with respect to $\delta$ and passage to the weak limit. This completes the proof.
\end{proof}

We can now let $\delta\to0$ in the Green formula associated with \eqref{eq/ beta twice bdy}. Since $\gamma$ is $1$-Lipschitz, Lemma \ref{lem/ delta lim} and Lemma \ref{lem/ in tr lim} give
\[
\gamma(\ov{\beta}_\delta) \longrightarrow  \gamma(f)
\]
strongly in the bulk and
\[
\gamma(\frkl_\delta) \longrightarrow  \gamma(f_l), \quad \gamma(\frkr_\delta) \longrightarrow  \gamma(f_r)
\]
strongly in the corresponding outgoing boundary spaces. Likewise,
\[
\gamma(\beta_\delta(f_0)) \longrightarrow  \gamma(f_0) \quad\text{in }L^1(\calO),
\]
and
\[
\gamma(\beta_\delta(g_l)) \longrightarrow  \gamma(g_l), \quad \gamma(\beta_\delta(g_r)) \longrightarrow  \gamma(g_r)
\]
in the corresponding incoming flux spaces. Proposition \ref{prop/ collision delta} handles the source term.

We conclude that
\bq\label{eq/ gamma green}
\p_t\gamma(f) +v_1\p_x\gamma(f) = \gamma'(f)\rho_f(M[f]-f)
\eq
in the renormalized Green formulation up to the spatial boundary, with traces
\begin{align*}
\gamma(f)|_{t=0} &= \gamma(f_0), \\
\gamma(f)|_{x=-\frac12,\ v_1<0} &= \gamma(f_l), \quad \gamma(f)|_{x=-\frac12,\ v_1>0} = \gamma(g_l), \\
\gamma(f)|_{x=\frac12,\ v_1>0} &= \gamma(f_r), \quad \gamma(f)|_{x=\frac12,\ v_1<0} = \gamma(g_r).
\end{align*}
The temporal trace at the terminal time is recovered from the classical trace theorem for transport equations; see \cite[Theorems 4.4--4.5]{Mis10}. Hence the Green formula holds on every finite time interval.

We next establish the strong time continuity of $f$. Set
\[
w:=\gamma(f)=\log(1+f).
\]
The trace theorem gives a representative such that
\[
w\in C([0,T];L^1_{\rm loc}(\calO)), \quad w(0)=\log(1+f_0).
\]
The energy and entropy bounds, together with Proposition \ref{prop/ super}, pass to the weak limit by velocity truncation and weak lower semicontinuity, applied with nonnegative test functions of time. In particular,
\[
\operatorname*{ess\,sup}_{0<t<T} \iint_{\calO} \lt( \la v\ra^2f(t)+\Phi(\la v\ra^2)f(t)+f(t)\log(1+f(t)) \rt)\,\dx\dv \le C_T.
\]

For every $t\in[0,T]$, define $f(t):=e^{w(t)}-1$. Approximating $t$ by times outside the exceptional null set and using the local continuity of $w$, followed by Fatou's lemma, extends the above bound to every $t\in[0,T]$.

Let $t_n\to t$. The continuity of $w$ implies that $f(t_n)\to f(t)$ locally in measure. By taking compact subintervals increasing to $I$, the same convergence holds in measure on $I\times B_R$ for every fixed $R>0$. The entropy bound gives uniform integrability of $(f(t_n))$ on this set; thus Vitali's theorem yields
\[
\iint_{I\times B_R}|f(t_n)-f(t)|\,\dx\dv \longrightarrow 0.
\]
On the other hand,
\[
\iint_{I\times\{|v|>R\}} |f(t_n)-f(t)|\,\dx\dv \le \frac{2C_T}{R^2}.
\]
Letting first $n\to\infty$ and then $R\to\infty$, we obtain
\[
f\in C([0,T];L^1(\calO)).
\]
An interpolation, now using the uniform bound on $\Phi(\la v\ra^2)f$ to control the energy tails, gives in a similar way
\[
f\in C([0,T];L^1(\calO,\la v\ra^2\,\dx\dv)).
\]
The initial trace of $w$ gives $f(0)=f_0$.

Finally, let $\beta\in\frkB$ be arbitrary and define
\[
\Gamma(s) := \beta(e^s-1), \quad s\ge0.
\]
Since
\[
\Gamma'(s) = e^s\beta'(e^s-1) = (1+z)\beta'(z) \big|_{z=e^s-1},
\]
the definition of $\frkB$ implies
\[
|\Gamma'(s)|\le C_\beta, \quad |\Gamma(s)|\le C_\beta s, \quad s\ge0.
\]
To justify the renormalization with $\Gamma$, let $(\Gamma_N)_{N\ge1}$ be a sequence of bounded-derivative truncations such that
\[
\Gamma_N\to\Gamma, \quad \Gamma_N'\to\Gamma' \quad\text{pointwise on }[0,\infty),
\]
and
\[
|\Gamma_N'(s)|\le C_\beta, \quad |\Gamma_N(s)|\le C_\beta s
\]
uniformly in $N$. Applying the renormalization theorem to $\gamma(f)=\log(1+f)$ and \eqref{eq/ gamma green} with $\Gamma_N$, we obtain
\[
\p_t\Gamma_N(\gamma(f)) +v_1\p_x\Gamma_N(\gamma(f)) = \Gamma_N'(\gamma(f)) \gamma'(f)\rho_f(M[f]-f).
\]
The source term is dominated by $C_\beta  | \gamma'(f)\rho_f(M[f]-f)|$, which belongs to $L^1_{\rm loc}$ by Proposition \ref{prop/ collision delta}. Moreover,
\[
|\Gamma_N(\gamma(f))| \le C_\beta\gamma(f) \le C_\beta f,
\]
and the same bound holds for the initial and boundary traces, since $\gamma(f_l)\le f_l$, $\gamma(f_r)\le f_r$, and similarly for the inflow data. Hence we may pass to the limit $N\to\infty$ in the Green formula.

Since $\Gamma(\gamma(f)) = \beta(f)$ and
\[
\Gamma'(\gamma(f)) \gamma'(f) \rho_f(M[f]-f)  = (1+f)\beta'(f) \frac{1}{1+f} \rho_f(M[f]-f) = \beta'(f)\rho_f(M[f]-f),
\]
we obtain, for every $\beta\in\frkB$,
\[
\p_t\beta(f) +v_1\p_x\beta(f) = \beta'(f)\rho_f(M[f]-f)
\]
with the corresponding initial and inflow boundary traces. Consequently, $f$ satisfies the Green formula \eqref{eq/ green} for every $\beta\in\frkB$, and hence is a renormalized solution of the BGK equation with the inflow boundary condition.

%
%
%
%
%
%
%
%
%
%
%

\section{Construction of a renormalized solution in the Maxwell case}\label{sec/ maxwell limit}

We now treat the Maxwell boundary condition \textbf{(Max)}. The interior compactness arguments and the passages to the limits $\e\to0$ and $\delta\to0$ developed in Sections \ref{subsec/ int lim} and \ref{subsec/ delta limit} are unchanged. The remaining issue is to identify the limiting boundary traces and to show that they satisfy the Maxwell boundary condition.

The boundary compactness mechanism is closely related to the general results of \cite[Theorems 2.4--2.5]{Mis10}, where the Darroz\`es--Guiraud information is combined with biting weak compactness. In the present slab setting the argument can be given directly, and we provide the details below.

Throughout this section, $T>0$ is fixed. All subsequences are extracted from the $\e$-subsequence already fixed in Section \ref{sec/ inflow limit}; in particular, all the interior limits obtained there remain valid.

%
%
%
%
%
%
%
%
%
%
%

\subsection{Compactness of the Maxwell boundary traces}

The arguments at $x=-\frac12$ and $x=\frac12$ are symmetric. We first consider the right boundary $x=\frac12$. Since $\alpha_r>0$ is fixed, \eqref{eq/ darr bdd}, Lemma \ref{lem/ j bounds}, and \eqref{eq/ bdy E reg} give
\bq\label{eq/ nice bounds}
\sup_{\e>0} \int_0^T ( \scrE_{\e r}(t)+J_{\e r}(t)  )\,\dt \le C_T,\quad \sup_{\e>0} \int_0^T\int_{v_1>0} \la v\ra^2f_{\e r}(t,v)v_1\,\dv\dt \le C_T.
\eq
Here and below $C_T$ may depend on $\alpha_r$.

Define the probability measure
\[
\dd\mu_+(v) := \frkM(v)v_1\one_{\{v_1>0\}}\,\dv,
\]
where the normalization follows from \eqref{eq/ M norm}, and set
\[
F_{\e r}(t,v) := \frac{f_{\e r}(t,v)}{\frkM(v)} \quad (v_1>0).
\]
Then
\bq\label{eq/ J mu}
J_{\e r}(t) = \int_{\R^3_+} F_{\e r}(t,v)\,\dd\mu_+(v).
\eq
Moreover,
\[
\scrE_{\e r}(t) = \int_{\R^3_+} h(F_{\e r})\,\dd\mu_+ - h(J_{\e r}(t)), \quad h(z):=z\log z.
\]

For $a,b\ge0$, define the relative entropy
\bq\label{eq/ rel ent}
h_{\rm rel}(b|a) :=
\begin{cases}
\displaystyle b\log\frac{b}{a}-(b-a), &a>0,\\[6pt]
0, &a=b=0,\\
+\infty, &a=0<b.
\end{cases}
\eq
If $J_{\e r}(t)>0$, then
\[
\scrE_{\e r}(t) = \int_{\R^3_+} \Big[ h(F_{\e r}) -h(J_{\e r}) -h'(J_{\e r}) (F_{\e r}-J_{\e r}) \Big] \,\dd\mu_+ = \int_{\R^3_+} h_{\rm rel} ( F_{\e r}|J_{\e r}  ) \,\dd\mu_+,
\]
where we used
\[
\int_{\R^3_+} (F_{\e r}-J_{\e r})\,\dd\mu_+ =0.
\]
If $J_{\e r}(t)=0$, then \eqref{eq/ J mu} and $F_{\e r}\ge0$ imply $F_{\e r}(t,\cdot)=0$ $\mu_+$-a.e.; thus the same identity holds with the convention in \eqref{eq/ rel ent}. Hence, for almost every $t$,
\bq\label{eq/ DG rel}
\scrE_{\e r}(t) = \int_{\R^3_+} h_{\rm rel} ( F_{\e r}|J_{\e r}  ) \,\dd\mu_+.
\eq

We first apply the biting lemma to the outgoing mass flux. By \eqref{eq/ nice bounds} and \cite[Theorem 2.4]{Mis10}, after passing to a further subsequence there exist $J_r^\ast\in L^1(0,T)$ and an increasing sequence of measurable sets
\[
\calG_k\subset(0,T), \quad |(0,T)\setminus\calG_k|<\frac1k,
\]
such that
\bq\label{eq/ J weak}
J_{\e r} \weakto J_r^\ast \quad\text{weakly in }L^1(\calG_k) \quad \text{for every }k\ge1.
\eq
Since $\mu_+$ is a probability measure, \eqref{eq/ J weak} also implies
\bq\label{eq/ J prod}
J_{\e r} \weakto J_r^\ast \quad\text{weakly in } L^1(\calG_k\times\R^3_+,\dt\,\dd\mu_+).
\eq
Indeed, for $\psi\in L^\infty(\calG_k\times\R^3_+)$,
\[
\iint_{\calG_k\times\R^3_+} J_{\e r}(t)\psi(t,v) \,\dt\,\dd\mu_+ = \int_{\calG_k} J_{\e r}(t) \lt( \int_{\R^3_+}\psi(t,v)\,\dd\mu_+ \rt)\dt.
\]
Consequently, $(J_{\e r})_\e$ is uniformly integrable in the product space appearing in \eqref{eq/ J prod}.

We next use the Darroz\`es--Guiraud information to obtain the corresponding uniform integrability of $F_{\e r}$.

\begin{lemma}
For every $L>0$,
\bq\label{eq/ young ineq}
F_{\e r} \le \frac1L h_{\rm rel} ( F_{\e r}|J_{\e r}  ) + \frac{e^L-1}{L} J_{\e r} \quad \text{a.e.}
\eq
\end{lemma}

\begin{proof}
Let $\Gamma(z):=z\log z-z+1$, $z\ge0$. Its Fenchel conjugate is $\Gamma^\ast(L)=e^L-1$, and thus
\[
Lz\le\Gamma(z)+e^L-1 \quad (L>0,\ z\ge0).
\]
When $J_{\e r}>0$, take $z=\frac{F_{\e r}}{J_{\e r}}$ and multiply by $J_{\e r}$. Since
\[
J_{\e r} \Gamma\lt( \frac{F_{\e r}}{J_{\e r}} \rt) = h_{\rm rel} ( F_{\e r}|J_{\e r}  ),
\]
we obtain \eqref{eq/ young ineq}. When $J_{\e r}=0$, we have $F_{\e r}=0$ $\mu_+$-a.e., and the inequality is trivial. This completes the proof.
\end{proof}

\begin{lemma}\label{lem/ weak for k}
After passing to a further subsequence, there exists
\[
F_r \in L^1((0,T)\times\R^3_+,\dt\,\dd\mu_+)
\]
such that, for every $k\ge1$,
\[
F_{\e r} \weakto F_r \quad\text{weakly in } L^1(\calG_k\times\R^3_+,\dt\,\dd\mu_+).
\]

\end{lemma}

\begin{proof}
Fix $k$. By \eqref{eq/ J prod}, the family $(J_{\e r})_\e$ is uniformly integrable in $L^1(\calG_k\times\R^3_+,\dt\,\dd\mu_+)$. Hence there exists a modulus $\omega_k$ with $\omega_k(s)\to0$ as $s\to0$ such that
\[
\sup_{\e>0} \iint_A J_{\e r}(t) \,\dt\,\dd\mu_+ \le \omega_k(|A|)
\]
for every measurable $A\subset\calG_k\times\R^3_+$, where
\[
|A| := \iint_A \dt\,\dd\mu_+.
\]
Integrating \eqref{eq/ young ineq} over $A$ and using \eqref{eq/ DG rel} and \eqref{eq/ nice bounds}, we obtain
\[
\sup_{\e>0} \iint_A F_{\e r}\,\dt\,\dd\mu_+ \le \frac{C_T}{L} + \frac{e^L-1}{L} \omega_k(|A|).
\]
We first choose $L$ large and then $|A|$ small. This proves uniform integrability of $(F_{\e r})_\e$ on $\calG_k\times\R^3_+$.

The Dunford--Pettis theorem gives weak compactness for each $k$. Using that the sets $\calG_k$ are increasing, a diagonal extraction gives a single subsequence for which the weak limits on the different $\calG_k$ agree on their intersections. We denote the resulting measurable function by $F_r$.

Finally,
\[
\iint_{\calG_k\times\R^3_+} F_r\,\dt\,\dd\mu_+ = \lim_{\e\to0} \iint_{\calG_k\times\R^3_+} F_{\e r}\,\dt\,\dd\mu_+ = \lim_{\e\to0} \int_{\calG_k} J_{\e r}(t)\,\dt \le C_T.
\]
Letting $k\to\infty$ and using monotone convergence gives
\[
F_r \in L^1((0,T)\times\R^3_+,\dt\,\dd\mu_+).
\]
This completes the proof.
\end{proof}

Define the outgoing limiting trace at $x=\frac12$ by
\[
f_r(t,v) := \frkM(v)F_r(t,v), \quad v_1>0.
\]
Lemma \ref{lem/ weak for k} is then equivalent to
\bq\label{eq/ out wk}
f_{\e r} \weakto f_r \quad\text{weakly in } L^1(\calG_k\times\R^3_+, v_1\,\dt\dv)
\eq
for every $k$.

We next identify the limit of the outgoing flux. Define
\[
J_r(t) := \int_{v_1>0}f_r(t,v)v_1\,\dv.
\]

\begin{lemma}\label{lem/ identify j}
The limit in \eqref{eq/ J weak} satisfies
\[
J_r^\ast=J_r \quad\text{a.e. in }(0,T).
\]
\end{lemma}

\begin{proof}
Let $\phi\in L^\infty(\calG_k)$. Then
\[
\int_{\calG_k} J_{\e r}(t)\phi(t)\,\dt = \iint_{\calG_k\times\R^3_+} f_{\e r}(t,v)\phi(t)v_1\,\dt\dv.
\]
Passing to the limit using \eqref{eq/ J weak} and \eqref{eq/ out wk}, we obtain
\[
\int_{\calG_k} J_r^\ast(t)\phi(t)\,\dt = \int_{\calG_k} J_r(t)\phi(t)\,\dt.
\]
Thus $J_r^\ast=J_r$ almost everywhere on $\calG_k$. Since $\calG_k\uparrow(0,T)$ up to a null set, the identity holds almost everywhere on $(0,T)$. This completes the proof.
\end{proof}

We now pass to the limit in the Maxwell boundary condition. For $v_1<0$,
\[
f_{\e r}(t,v) = (1-\alpha_r) f_{\e r}(t,\scrV v) + \alpha_r\frkM(v)J_{\e r}(t).
\]
By \eqref{eq/ out wk},
\[
f_{\e r}(t,\scrV v) \weakto f_r(t,\scrV v) \quad \text{weakly in } L^1(\calG_k\times\R^3_-, |v_1|\,\dt\dv).
\]
Moreover, \eqref{eq/ J weak} and Lemma \ref{lem/ identify j} give
\[
\frkM(v)J_{\e r}(t) \weakto \frkM(v)J_r(t)
\]
in the same space. Hence
\[
f_{\e r} \weakto (1-\alpha_r)f_r(t,\scrV v) + \alpha_r\frkM(v)J_r(t)
\]
weakly in $L^1(\calG_k\times\R^3_-, |v_1|\,\dt\dv)$.

We thus extend $f_r$ to the incoming velocities by
\bq\label{eq/ max r}
f_r(t,v) := (1-\alpha_r)f_r(t,\scrV v) + \alpha_r\frkM(v)J_r(t), \quad v_1<0.
\eq
With this definition,
\bq\label{eq/ r weak}
f_{\e r} \weakto f_r \quad\text{weakly in } L^1(\calG_k\times\R^3, |v_1|\,\dt\dv)
\eq
for every $k$, and
\[
f_r\in L^1((0,T)\times\R^3, |v_1|\,\dt\dv).
\]

Repeating the same argument at $x=-\frac12$, and passing to a further subsequence if necessary, we obtain a trace $f_l\in L^1((0,T)\times\R^3, |v_1|\,\dt\dv)$ satisfying
\bq\label{eq/ max l}
f_l(t,v) = (1-\alpha_l)f_l(t,\scrV v) + \alpha_l\frkM(v)J_l(t), \quad v_1>0,
\eq
where
\[
J_l(t) := \int_{v_1<0}f_l(t,v)|v_1|\,\dv.
\]
By replacing the good sets obtained at the two endpoints by their intersections and relabeling them, we may assume that the same increasing sequence $(\calG_k)$ satisfies
\[
\begin{aligned}
f_{\e l} &\weakto f_l &&\text{weakly in } L^1(\calG_k\times\R^3, |v_1|\,\dt\dv),\\
f_{\e r} &\weakto f_r &&\text{weakly in } L^1(\calG_k\times\R^3, |v_1|\,\dt\dv)
\end{aligned}
\]
for every $k$, while
\[
|(0,T)\setminus\calG_k|\longrightarrow 0.
\]

%
%
%
%
%
%
%
%
%
%
%

\subsection{Passage to the limit at the boundary}

We now return to the renormalized equations considered in Section \ref{sec/ inflow limit}. Recall that a decreasing sequence $(\delta_j)$ with $\delta_j\downarrow0$ and a common $\e$-subsequence were fixed there. All quantities indexed by $\delta$ below are understood for $\delta\in\{\delta_j:j\ge1\}$.

After a further diagonal extraction in $\e$, preserving all the interior limits from Sections \ref{subsec/ int lim} and \ref{subsec/ delta limit}, we may assume that, for every such $\delta$,
\bq\label{eq/ max tr}
\beta_\delta(f_{\e l}) \weakstarto \frkl_\delta,\quad  \beta_\delta(f_{\e r}) \weakstarto \frkr_\delta \quad \text{in }L^\infty((0,T)\times\R^3).
\eq
This is possible since $0\le\beta_\delta(f_{\e l}), \beta_\delta(f_{\e r}) \le\frac1\delta$. 

Let $\varphi\in \calD([0,T)\times\ov{\calO})$. Since $v_1\varphi(\cdot,\pm\frac12,\cdot)$ belongs to $L^1((0,T)\times\R^3)$, \eqref{eq/ max tr} gives
\begin{align*}
&\int_0^T\int_{\R^3} \beta_\delta(f_{\e r}) \varphi(t,\frac12,v)v_1\,\dv\dt \longrightarrow  \int_0^T\int_{\R^3} \frkr_\delta \varphi(t,\frac12,v)v_1\,\dv\dt, \\
&\int_0^T\int_{\R^3} \beta_\delta(f_{\e l}) \varphi(t,-\frac12,v)v_1\,\dv\dt \longrightarrow  \int_0^T\int_{\R^3} \frkl_\delta \varphi(t,-\frac12,v)v_1\,\dv\dt.
\end{align*}

Together with the interior convergence established in Section \ref{subsec/ int lim}, these limits show that $\ov{\beta}_\delta$ satisfies
\[
\p_t\ov{\beta}_\delta +v_1\p_x\ov{\beta}_\delta = \rho_f ( \frka_\delta M[f]-\frkb_\delta  )
\]
in the Green formulation up to the spatial boundary, with traces
\[
\ov{\beta}_\delta|_{x=-\frac12} = \frkl_\delta, \quad \ov{\beta}_\delta|_{x=\frac12} = \frkr_\delta.
\]

Set, as in Section \ref{subsec/ delta limit},
\[
\gamma(z):=\log(1+z).
\]
The trace-renormalization property for transport equations \cite[Theorem 4.5]{Mis10} gives the Green formulation
\bq\label{eq/ max beta twice bdy}
\p_t\gamma(\ov{\beta}_\delta) +v_1\p_x\gamma(\ov{\beta}_\delta) = \frkQ_\delta
\eq
up to the spatial boundary, with traces
\[
\gamma(\ov{\beta}_\delta)|_{x=-\frac12} = \gamma(\frkl_\delta), \quad \gamma(\ov{\beta}_\delta)|_{x=\frac12} = \gamma(\frkr_\delta).
\]

We next identify the limits of these boundary traces as $\delta\to0$.

\begin{lemma}\label{lem/ gk str}
For every $k\ge1$,
\begin{align}
\frkr_\delta &\nearrow f_r \quad \text{strongly in } L^1(\calG_k\times\R^3, |v_1|\,\dt\dv), \label{eq/ rdelta max str}\\
\frkl_\delta &\nearrow f_l \quad \text{strongly in } L^1(\calG_k\times\R^3, |v_1|\,\dt\dv) \nonumber
\end{align}
as $\delta\to0$ along $(\delta_j)$.
\end{lemma}

\begin{proof}
We give the proof at the right endpoint. By \eqref{eq/ r weak}, $(f_{\e r})_\e$ is weakly compact in $L^1(\calG_k\times\R^3, |v_1|\,\dt\dv)$. Since $0\le\beta_\delta(f_{\e r})\le f_{\e r}$, the family $(\beta_\delta(f_{\e r}))_\e$ is weakly compact in the same space. Its weak limit is $\frkr_\delta$, by uniqueness of distributional limits and \eqref{eq/ max tr}.

For $L>1$,
\bq\label{eq/ max tr err}
0 \le f_{\e r}-\beta_\delta(f_{\e r}) = \frac{\delta f_{\e r}^2} {1+\delta f_{\e r}} \le \delta Lf_{\e r} + f_{\e r} \one_{\{f_{\e r}>L\}}. 
\eq
The weak compactness of $(f_{\e r})_\e$ in this $L^1$ space implies uniform integrability. Consequently,
\[
\lim_{L\to\infty} \sup_{\e>0} \iint_{\calG_k\times\R^3} f_{\e r} \one_{\{f_{\e r}>L\}} |v_1|\,\dt\dv =0.
\]
Using also the uniform $L^1$ bound, we obtain from \eqref{eq/ max tr err}
\[
\lim_{\delta\to0} \sup_{\e>0} \| \beta_\delta(f_{\e r})-f_{\e r} \|_{L^1(\calG_k\times\R^3, |v_1|\,\dt\dv)} =0.
\]
Lemma \ref{lem/ n delta} then gives the convergence stated in \eqref{eq/ rdelta max str}.

If $0<\delta_1<\delta_2$, then $\beta_{\delta_1}(f_{\e r}) \ge \beta_{\delta_2}(f_{\e r})$. Passing to the weak limit gives $\frkr_{\delta_1} \ge \frkr_{\delta_2}$ almost everywhere. Hence the convergence is monotone. The proof at the left endpoint is identical. This completes the proof.
\end{proof}

Since the convergence in Lemma \ref{lem/ gk str} is monotone and the sets $\calG_k$ cover $(0,T)$ up to a null set, we have
\[
\frkr_\delta\nearrow f_r, \quad \frkl_\delta\nearrow f_l \quad\text{a.e. on }(0,T)\times\R^3.
\]
In particular,
\[
0\le\frkr_\delta\le f_r, \quad 0\le\frkl_\delta\le f_l.
\]
Since $\gamma$ is increasing and $0\le\gamma(z)\le z$, the dominated convergence theorem gives
\begin{align}
\gamma(\frkr_\delta) &\longrightarrow  \gamma(f_r) \quad \text{strongly in } L^1((0,T)\times\R^3, |v_1|\,\dt\dv), \label{eq/ log r str}\\
\gamma(\frkl_\delta) &\longrightarrow  \gamma(f_l) \quad \text{strongly in } L^1((0,T)\times\R^3, |v_1|\,\dt\dv). \label{eq/ log l str}
\end{align}

We may now pass to the limit $\delta\to0$ in the Green formula associated with \eqref{eq/ max beta twice bdy}. The bulk convergence and the convergence of the collision terms are given by Lemma \ref{lem/ delta lim} and Proposition \ref{prop/ collision delta}, while \eqref{eq/ log r str}--\eqref{eq/ log l str} handle the spatial boundary terms. We obtain
\[
\p_t\gamma(f) +v_1\p_x\gamma(f) = \gamma'(f)\rho_f(M[f]-f)
\]
in the Green formulation, with spatial traces
\[
\gamma(f)|_{x=-\frac12} = \gamma(f_l), \quad \gamma(f)|_{x=\frac12} = \gamma(f_r).
\]
The limiting traces satisfy the Maxwell boundary conditions \eqref{eq/ max l} and \eqref{eq/ max r} almost everywhere. Thus the Maxwell boundary condition is realized by the limiting solution.

Finally, the last renormalization argument of Section \ref{sec/ inflow limit} applies without change. Indeed, for $\beta\in\frkB$ and $\Gamma(s):=\beta(e^s-1)$, we have
\[
|\Gamma'(s)|\le C_\beta, \quad |\Gamma(s)|\le C_\beta s.
\]
Applying the renormalization theorem first to bounded-derivative truncations of $\Gamma$ and then passing to the truncation limit yields
\[
\p_t\beta(f) +v_1\p_x\beta(f) = \beta'(f)\rho_f(M[f]-f)
\]
with traces $\beta(f_l)$ and $\beta(f_r)$. Consequently, $f$ satisfies the Green formula \eqref{eq/ green} for every $\beta\in\frkB$ and is a renormalized solution of the BGK equation with the Maxwell boundary condition.

The recovery of the terminal temporal trace and the strong continuity $f\in  C([0,T];L^1(\calO,\la v \ra^2 \dx\dv))$ proceeds exactly as in Section \ref{sec/ inflow limit}, since that argument depends only on the interior renormalized equation and the uniform bulk estimates. A diagonal argument in $T$ completes the construction on $[0,\infty)$.

%
%
%
%
%
%
%
%
%
%
%

\subsection{Entropy inequalities and global balances}

The compactness arguments also preserve the entropy inequalities. We record these properties separately from the renormalized equation, since the functions $z$ and $z\log z$ do not belong to $\frkB$.

\begin{proposition}\label{prop/ bal}
Let $f$ be the limit obtained in Sections \ref{sec/ inflow limit} and \ref{sec/ maxwell limit} from a family of entropy admissible approximate solutions as in Section \ref{sec/ approximate}, and set
\begin{equation}\label{eq/ diss def}
\scrD[f](t):=\iint_{\calO}\rho_f(f-M[f])(\log f-\log M[f])\,\dx\dv.
\end{equation}
The integrand is defined to be zero on $\{\rho_f=0\}$, with the usual lower semicontinuous extension elsewhere. Then $\scrD[f]\in L^1(0,T)$ for every $T>0$.

In case \textbf{(In)}, for every $t\ge0$,
\bq\label{eq/ ent in}
\begin{aligned}
&\scrH[f(t)]+\int_0^t\scrD[f](s)\,\ds+\int_0^t\int_{v_1<0}f_l\log f_l\,|v_1|\,\dv\ds+\int_0^t\int_{v_1>0}f_r\log f_r\,v_1\,\dv\ds \\
&\quad\le\scrH[f_0]+\int_0^t\int_{v_1>0}g_l\log g_l\,v_1\,\dv\ds+\int_0^t\int_{v_1<0}g_r\log g_r\,|v_1|\,\dv\ds.
\end{aligned}
\eq
Moreover, for each $\psi\in\{1,v_1,v_2,v_3,|v|^2\}$,
\bq\label{eq/ bal in}
\iint_{\calO}\psi(v)f(t)\,\dx\dv+\int_0^t\int_{\R^3}\psi(v)v_1(f_r-f_l)\,\dv\ds=\iint_{\calO}\psi(v)f_0\,\dx\dv.
\eq
Here the incoming parts of $f_l,f_r$ are $g_l,g_r$, respectively.

In case \textbf{(Max)}, the total mass is conserved and
\bq\label{eq/ ent max}
\scrH_{\frkM}[f(t)]+\int_0^t\scrD[f](s)\,\ds+\int_0^t(\alpha_l\scrE_l(s)+\alpha_r\scrE_r(s) )\,\ds\le\scrH_{\frkM}[f_0],
\eq
where $\scrH_{\frkM}[f]=\iint_{\calO}f\log\frac{f}{\frkM}\,\dx\dv$ and $\scrE_l,\scrE_r$ are the Darroz\`es--Guiraud terms in Lemma \ref{lem/ every}, evaluated at the limiting traces. In particular, the mass and kinetic energy are bounded uniformly for $t\ge0$ in both cases. The thermal estimate also passes to the limit:
\[
\int_0^T\int_I\rho_f^2\theta_f\,\dx\dt\le C_T.
\]
\end{proposition}

\begin{proof}
We first pass to the limit in the entropy production. The function
\[
d(a,b):= 
\begin{cases}
(a-b)(\log a-\log b) & a,b>0, \\
+\infty &\text{exactly one of $a,b$ is zero,}\\
0 &a=b=0
\end{cases}
\]
is nonnegative, jointly convex, and lower semicontinuous on $[0,\infty) \times [0,\infty)$. We know from Lemma \ref{lem/ max L1} the almost everywhere convergence of the collision frequencies and Maxwellians, while $f_\e\weakto f$ in $L^1$.

The homogeneity relation $d(ca,cb)=c\,d(a,b)$ for $c \ge 0$ (with the convention $0\cdot\infty = 0$) allows us to rewrite the integrand of \eqref{eq/ diss def} as $d(\rho_f f, \rho_f M[f])$. We then consider the truncations $d(\min\{\rho_f,R\}f, \min\{\rho_f, R\}M[f])$, where $R>0$. From the product convergence theorem, for each $R>0$ we obtain that 
\begin{equation*}
    \Big(\min\{\lambda_\ve(\rho_{f_\ve}) , R\} f_\ve , \min\{\lambda_\ve(\rho_{f_\ve}), R \} M[f_\ve] \Big) \weakto \Big(\min\{\rho_f , R\} f,  \min\{\rho_f, R\} M[f]\Big)  
\end{equation*}
weakly in $[L^1_{\rm loc}([0,T]\times \ov{\calO})]^2$. It is then a standard exercise, using the properties of $d(\cdot,\cdot)$, Fatou and Mazur's lemmas, along with a localization procedure (followed by the monotone convergence theorem), that
\begin{align*}
    &\int_0^t \iint_{\calO} \min\{\rho_f, R\} d(f,M[f]) \,\dx\dv\ds \\
    &=\int_0^t \iint_{\calO} d(\min\{\rho_f , R\} f,  \min\{\rho_f, R\} M[f]) \,\dx\dv\ds  \\
    &\le \liminf_{\ve\to 0} \int_0^t \iint_{\calO} d(\min\{\lambda_\ve(\rho_{f_\ve}) , R\} f_\ve , \min\{\lambda_\ve(\rho_{f_\ve}), R \} M[f_\ve] ) \,\dx\dv\ds \\
    &= \liminf_{\ve\to 0} \int_0^t \iint_{\calO} \min\{\lambda_\ve(\rho_{f_\ve}),R\} d(f_\ve, M[f_\ve]) \,\dx\dv\ds \\
    &\le \liminf_{\ve\to 0} \int_0^t \iint_{\calO} \lambda_\ve(\rho_{f_\ve})  d(f_\ve, M[f_\ve]) \,\dx\dv\ds ,
\end{align*}
where the equalities above are due to homogeneity. The monotone convergence theorem with $R\to\infty$ then yields
\[
\int_0^t\scrD[f](s)\,\ds\le\liminf_{\e\to0}\int_0^t\scrD_\e[f_\e](s)\,\ds.
\]
The same argument applies with nonnegative bounded weights in time.
 
 Fix $T>0$. For \textbf{(In)}, the weak convergence in \eqref{eq/ fl weak}--\eqref{eq/ fr weak}, together with the uniform energy tails, gives
\[
\begin{aligned}
f_{\e l}&\weakto f_l &&\text{in }L^1((0,T)\times\R^3_-,\la v\ra^2|v_1|\,\dt\dv),\\
f_{\e r}&\weakto f_r &&\text{in }L^1((0,T)\times\R^3_+,\la v\ra^2v_1\,\dt\dv).
\end{aligned}
\]
The same truncation argument gives weak convergence of $f_\e$ in $L^1((0,T)\times\calO,\la v\ra^2\,\dt\dx\dv)$. The function $\eta_v(z):=z\log z+\la v\ra^2z+e^{-1-\la v\ra^2}$ is nonnegative and convex in $z$. Applying weak lower semicontinuity to $\eta_v$ and subtracting the convergent linear terms gives the required lower semicontinuity for the bulk and boundary entropies.

For brevity, write
\[
\begin{aligned}
\scrB_{\rm out}(s)&:=\int_{v_1<0}f_l\log f_l\,|v_1|\,\dv+\int_{v_1>0}f_r\log f_r\,v_1\,\dv,\\
\scrB_{\rm in}(s)&:=\int_{v_1>0}g_l\log g_l\,v_1\,\dv+\int_{v_1<0}g_r\log g_r\,|v_1|\,\dv.
\end{aligned}
\]
Let $\chi\in C_c^\infty(0,T)$ be nonnegative and set $\widehat\chi(s):=\int_s^T\chi(t)\,\dt$. Multiply \eqref{eq/ H in} by $\chi(t)$ and integrate over the terminal time $t$. Fubini's theorem, the lower semicontinuity just proved, and the convergence of the initial entropy yield
\[
\int_0^T\chi(t)\scrH[f(t)]\,\dt+\int_0^T\widehat\chi(s)\bigl(\scrD[f](s)+\scrB_{\rm out}(s)\bigr)\,\ds \le\scrH[f_0]\int_0^T\chi(t)\,\dt+\int_0^T\widehat\chi(s)\scrB_{\rm in}(s)\,\ds.
\]
Since $\chi$ is arbitrary, an application of Fubini's theorem shows that \eqref{eq/ ent in} holds for almost every $t\in(0,T)$. For an arbitrary $t\in[0,T)$, choose such times $t_n\to t$. The weighted time continuity established in Section \ref{sec/ inflow limit} gives
\[
\iint_{\calO}\la v\ra^2|f(t_n)-f(t)|\,\dx\dv \longrightarrow0,\quad \scrH[f(t)]\le\liminf_{n\to\infty}\scrH[f(t_n)].
\]
All cumulative time integrals in \eqref{eq/ ent in} are continuous. Hence this inequality holds for every $t\ge0$, since $T$ is arbitrary.

For the balance laws, fix $\psi\in\{1,v_1,v_2,v_3,|v|^2\}$. The Maxwellian moment identities give
\[
\int_{\R^3}\psi(v)\lambda_\e(\rho_{f_\e})(M[f_\e]-f_\e)\,\dv=0.
\]
Use $\zeta(t)\psi(v)\chi_R(v)$ in the approximate Green formula, where $\zeta\in C_c^\infty([0,T))$ and $\chi_R$ is a velocity cutoff. Letting $R\to\infty$ at fixed $\e$ eliminates the collision term. By \eqref{eq/ avg} and the uniform energy tails,
\[
\iint_{\calO}\psi(v)f_\e(t)\,\dx\dv\longrightarrow\iint_{\calO}\psi(v)f(t)\,\dx\dv\quad\text{in }L^1(0,T).
\]
The weighted weak convergence of the traces and the initial convergence then allow us to pass to the limit $\e\to0$, giving
\[
\int_0^T\zeta'(t)\iint_{\calO}\psi(v)f(t)\,\dx\dv\dt+\zeta(0)\iint_{\calO}\psi(v)f_0\,\dx\dv =\int_0^T\zeta(t)\int_{\R^3}\psi(v)v_1(f_r-f_l)\,\dv\dt.
\]
The boundary flux belongs to $L^1(0,T)$. Integrating this distributional identity and using the weighted time continuity of $f$ proves \eqref{eq/ bal in} for every $t\ge0$.

For \textbf{(Max)}, pass to the limit in the approximate mass identity after integration against $\chi$. This gives
\[
\int_0^T\chi(t)\iint_{\calO}f(t)\,\dx\dv\dt=\left(\iint_{\calO}f_0\,\dx\dv\right)\int_0^T\chi(t)\,\dt.
\]
The continuity of $f$ in $L^1(\calO)$ yields mass conservation for every $t\ge0$.

The bulk relative entropy is weakly lower semicontinuous, since
\[
z\log\frac{z}{\frkM(v)}+\frac{\frkM(v)}e\ge0
\]
is convex in $z$. At the right boundary, put $F_r:=\frac{f_r}{\frkM}$ and recall that $\dd\mu_+=\frkM v_1\one_{\{v_1>0\}}\,\dv$. The weak convergence of $F_{\e r}$ and $J_{\e r}$ on $\calG_k$, together with the joint convexity of $h_{\rm rel}$ in \eqref{eq/ DG rel}, gives
\[
\int_{\calG_k}\widehat\chi(s)\scrE_r(s)\,\ds=\int_{\calG_k}\int_{\R^3_+}\widehat\chi(s)h_{\rm rel}(F_r|J_r)\,\dd\mu_+\,\ds \le\liminf_{\e\to0}\int_{\calG_k}\widehat\chi(s)\scrE_{\e r}(s)\,\ds.
\]
The same argument applies at the left boundary. Since the information terms are nonnegative and $\calG_k$ increases to $(0,T)$ up to a null set, monotone convergence yields
\[
\int_0^T\widehat\chi(s)\bigl(\alpha_l\scrE_l(s)+\alpha_r\scrE_r(s)\bigr)\,\ds \le\liminf_{\e\to0}\int_0^T\widehat\chi(s)\bigl(\alpha_l\scrE_{\e l}(s)+\alpha_r\scrE_{\e r}(s)\bigr)\,\ds.
\]
We may now integrate \eqref{eq/ H max} against $\chi(t)$ and pass to the limit. Combining the bulk, collision, and boundary lower semicontinuity estimates gives
\[
\int_0^T\chi(t)\scrH_{\frkM}[f(t)]\,\dt +\int_0^T\widehat\chi(s)\bigl(\scrD[f](s)+\alpha_l\scrE_l(s)+\alpha_r\scrE_r(s)\bigr)\,\ds\le\scrH_{\frkM}[f_0]\int_0^T\chi(t)\,\dt.
\]
The terminal time argument used above proves \eqref{eq/ ent max} for every $t\ge0$.

In case \textbf{(In)}, summing \eqref{eq/ bal in} for $\psi=1$ and $\psi=|v|^2$, and discarding the nonnegative outgoing flux, gives a bound on the total mass and energy that is uniform in time, by \eqref{eq/ in data}. In case \textbf{(Max)}, \eqref{eq/ H coerc} and \eqref{eq/ ent max} give
\[
a\iint_{\calO}|v|^2f(t)\,\dx\dv\le\scrH_{\frkM}[f(t)]+C_{\frkM,a}\le\scrH_{\frkM}[f_0]+C_{\frkM,a},
\]
with $a>0$ chosen as in \eqref{eq/ H coerc}. Together with mass conservation, this proves the asserted global bound.

Finally, Lemma \ref{lem/ max L1} and Fatou's lemma applied to \eqref{eq/ therm reg} give
\[
\int_0^T\int_I\rho_f^2\theta_f\,\dx\dt\le\liminf_{\e\to0}\int_0^T\int_I\lambda_\e(\rho_{f_\e})\rho_{f_\e}\theta_{f_\e}\,\dx\dt\le C_T,
\]
where $\rho_f^2\theta_f$ is defined to be zero on $\{\rho_f=0\}$. This completes the proof.
\end{proof}

\begin{remark}
The Maxwell boundary condition describes exchange of momentum and energy with the boundary; thus these quantities are not generally conserved. The convergence proved here directly yields mass conservation and the relative entropy inequality \eqref{eq/ ent max}. An identity for the full momentum or energy flux at the Maxwell boundary would additionally require ruling out concentration in time of the unrenormalized boundary fluxes; biting weak convergence alone does not justify such an identity. Similarly, a local energy balance would require control of the third velocity moment in the transport flux, which is not assumed here.
\end{remark}

%
%
%
%
%
%
%
%
%
%
%

\appendix
%
%
%
%
%
%
%
%
%
%
%
\section{Existence for the truncated collision frequency}\label{app/ approx}

We supply the approximate solutions used in Section \ref{sec/ tails}. The construction follows the compactness method for BGK equations in \cite{Per89,PD91}, with the transport traces understood as in \cite{Mis00T,Mis10}. The parameter $\e$ is fixed throughout this appendix.

\begin{proposition}\label{prop/ approx}
Let $\e\in(0,1)$ and let $f_{\e0}\ge0$ satisfy
\[
\iint_{\calO}f_{\e0}(\la v\ra^2+|\log f_{\e0}| )\,\dx\dv<\infty.
\]
Under either \textbf{(In)} with \eqref{eq/ in data}, or \textbf{(Max)} with $\alpha_l,\alpha_r>0$, problem \eqref{eq/ regbgk} has a nonnegative global weak solution $f_\e$. For every $T>0$, it belongs to $C([0,T];L^1(\calO,\la v\ra^2\,\dx\dv))$, has the weighted traces in \eqref{eq/ tr wt}, and satisfies the entropy inequality \eqref{eq/ H in} or \eqref{eq/ H max}, respectively. In particular, it has all the properties required in Section \ref{sec/ approximate}.
\end{proposition}

\begin{proof}
Write $X=L^1(\calO,\la v\ra^2\,\dx\dv)$. We first recall a linear transport estimate. If $h$ solves $\p_t h+v_1\p_xh=q$ with zero inflow or with \textbf{(Max)}, then
\bq\label{eq/ lin tr}
\sup_{0\le t\le T}\|h(t)\|_X+\sum_{b\in\{l,r\}}\int_0^T\int_{\R^3}\la v\ra^2|h_b|\,|v_1|\,\dv\dt\le C_T\lt(\|h(0)\|_X+\int_0^T\|q(t)\|_X\,\dt\rt).
\eq
For zero inflow this follows from the characteristic formula. For \textbf{(Max)}, the positive linear solution is constructed by iterating the incoming trace, starting with zero incoming trace. The required bound follows by testing with
\[
W(x,v)=\la v\ra^2+c_{\frkM}\lt(\frac12+x\,\operatorname{sgn}v_1\rt),\quad c_{\frkM}=\int_{v_1>0}|v|^2\frkM(v)v_1\,\dv.
\]
Indeed, $W$ is equivalent to $\la v\ra^2$, $|v_1\p_xW|\le c_{\frkM}W$, and the outgoing weight minus the reflected average of the incoming weight is $\alpha_b|v|^2+(1-\alpha_b)c_{\frkM}\ge0$ at boundary $b$. This gives the bulk bound by Gr\"onwall's inequality. The tests $xv_1$ and $\la v\ra^2$, as in Lemma \ref{lem/ j bounds}, control the boundary fluxes; the additional source terms are bounded by $C\int_0^T\|q\|_X\,\dt$. These bounds justify the monotone construction. Positivity gives \eqref{eq/ lin tr} for signed data by comparison. The characteristic formula and approximation of the data give continuity in $X$. Prescribed inflow is included through its free transport solution.

For $F\ge0$, write $j_F=\int_{\R^3}vF\,\dv$ and $E_F=\int_{\R^3}|v|^2F\,\dv$, and set
\[
\calU=\lt\{(r,j,E)\in(0,\infty)\times\R^3\times\R:E>\frac{|j|^2}{r}\rt\}.
\]
Choose $\chi_n\in C_c^\infty(\calU)$ with $0\le\chi_n\le1$, equal to one on each compact subset of $\calU$ for all sufficiently large $n$. Extend it by zero outside $\calU$, and define
\[
a_n[F]=\lambda_\e(\rho_F)\chi_n(\rho_F,j_F,E_F),\quad Q_n(F)=a_n[F](M[F]-F).
\]
The moment vector belongs to $\calU$ wherever $\rho_F>0$, since a nonzero integrable velocity density has positive temperature. On the support of $\chi_n$, the density and temperature are bounded away from zero and all macroscopic parameters are bounded. Smooth dependence of the Gaussian on these parameters shows that $Q_n$ is globally Lipschitz on the nonnegative cone of $X$; for the loss term, use also the bound on the local weighted mass $\rho_F+E_F$ on this support. Moreover, $Q_n(F)+\frac1\e F\ge0$. The Duhamel fixed point argument, using \eqref{eq/ lin tr}, thus gives a global nonnegative solution $f_n\in C([0,T];X)$ of
\[
\p_t f_n+v_1\p_xf_n=Q_n(f_n),\quad f_n(0)=f_{\e0},
\]
with the prescribed boundary condition. We write $a_n=a_n[f_n]$.

The cutoff is independent of $v$ and multiplies the entire collision operator. Consequently, its moments against $1,v,|v|^2$ vanish, and its entropy production is $a_n(f_n-M[f_n])(\log f_n-\log M[f_n])\ge0$. Approximation of the initial and incoming data, followed by convex regularizations of $z\log z$ in the Green formula, justifies the entropy inequalities computed in Lemma \ref{lem/ every}. That proof and the multiplier argument of Lemma \ref{lem/ j bounds} give the bulk energy, absolute entropy, and boundary bounds uniformly in $n$.

Choose $\Phi$ by Lemma \ref{lem/ dlVP} and Corollary \ref{cor/ in Phi}, applied to the fixed initial datum and, in case \textbf{(In)}, the incoming data. The algebraic estimate \eqref{eq/ Phi Q} gives
\[
\int_0^T\iint_{\calO}\Phi(\la v\ra^2)Q_n(f_n)\,\dx\dv\dt\le C\int_0^T\int_I a_n\rho_{f_n}\theta_{f_n}\,\dx\dt\le\frac{C}{3\e}\int_0^T\iint_{\calO}|v|^2f_n\,\dx\dv\dt.
\]
The boundary calculation and truncation in Proposition \ref{prop/ super} yield its weighted energy estimate uniformly in $n$, with a constant $C_{\e,T}$. Thus this construction uses only the bounded rate to propagate the energy weight.

Extract $f_n\weakto f$ in $L^1((0,T)\times\calO)$. Since
\[
\int_0^T\|Q_n(f_n)\|_X\,\dt\le\frac2\e\int_0^T\|f_n\|_X\,\dt,
\]
velocity averaging applied to $\beta_\delta(f_n)$, followed by the entropy truncation estimate \eqref{eq/ beta err}, gives compactness of the bounded velocity moments; see \cite[Section III, Theorem 6]{DLM91}. The energy tails supplied by $\Phi$ then give strong convergence of $(\rho_{f_n},j_{f_n},E_{f_n})$ in $L^1((0,T)\times I)$, exactly as in Lemma \ref{lem/ mom2}. The argument of Lemma \ref{lem/ max L1} gives $M[f_n]\to M[f]$ in $L^1(0,T;X)$. After extraction,
\[
a_n\to a:=\lambda_\e(\rho_f)\quad\text{a.e.},\quad G_n:=a_nM[f_n]\to G:=aM[f]\quad\text{in }L^1(0,T;X).
\]
On $\{\rho_f>0\}$, the cutoff is eventually one along the converging moment vectors; on $\{\rho_f=0\}$, use $0\le a_n\le\rho_{f_n}$. The gain convergence follows from $a_n\le\frac1\e$.

To identify the full solution and its boundary traces, let $F$ solve the linear problem $\p_tF+v_1\p_xF+aF=G$ with the same initial and boundary data. Its existence follows from the linear construction above and the boundedness of $a$. Set $H_n=G_n-G-(a_n-a)F$. Dominated convergence gives $H_n\to0$ in $L^1(0,T;X)$, while
\[
\p_t(f_n-F)+v_1\p_x(f_n-F)=H_n-a_n(f_n-F),\quad (f_n-F)(0)=0.
\]
Estimate \eqref{eq/ lin tr} and Gr\"onwall's inequality imply $f_n\to F$ in $C([0,T];X)$ and convergence of both traces in $L^1((0,T)\times\R^3,\la v\ra^2|v_1|\,\dt\dv)$. Hence $F=f$, and the limit satisfies \eqref{eq/ regbgk} and the exact boundary condition. The collision term converges in $L^1(0,T;X)$, which also justifies the moment identities and the transport renormalizations.

Finally, lower semicontinuity gives \eqref{eq/ H in} or \eqref{eq/ H max}. The weighted strong convergence controls the negative entropy parts through \eqref{eq/ H neg}; Fatou's lemma applies to the nonnegative bulk production and to the Darroz\`es--Guiraud integrand \eqref{eq/ DG rel}. The outgoing fluxes converge strongly as a consequence of the trace convergence. Continuity in $X$ and lower semicontinuity of the entropy extend the inequalities to every terminal time. A diagonal extraction in $T$ gives the global solution $f_\e$. This completes the proof.
\end{proof}

%
%
%
%
%
%
%
%
%
%
%
\section{The mild formula}\label{app/ mild}

In this appendix, we derive the mild formulation along each free flight segment inside the slab. The argument applies to both \textbf{(In)} and \textbf{(Max)}. The difference lies only in the value prescribed at the backward endpoint of a characteristic.

Recall that $\gamma(z):=\log(1+z)$. In the inflow case, Section \ref{sec/ inflow limit} gives
\[
\p_t\gamma(f)+v_1\p_x\gamma(f) = \gamma'(f)\rho_f(M[f]-f)
\]
in the Green formulation, with incoming traces
\[
\gamma(g_l) \quad\text{at }x=-\frac12,\ v_1>0, \quad \gamma(g_r) \quad\text{at }x=\frac12,\ v_1<0.
\]
In the Maxwell case, Section \ref{sec/ maxwell limit} gives the same equation, with the incoming traces determined by the limiting Maxwell boundary condition \eqref{eq/ max l} and \eqref{eq/ max r}. The initial trace is $\gamma(f_0)$ in either case.

For $(x,v)\in I\times\R^3$, define the backward exit time
\[
\tau_b(x,v) :=
\begin{cases}
\displaystyle\frac{x+\frac12}{v_1}, &v_1>0,\\[8pt]
\displaystyle\frac{\frac12-x}{|v_1|}, &v_1<0,\\[8pt]
+\infty, &v_1=0,
\end{cases}
\]
and set
\[
s_b=s_b(t,x,v) := \max\{0,t-\tau_b(x,v)\}, \quad X(s):=x-v_1(t-s).
\]
Then $X(t)=x$, and, whenever $t>\tau_b(x,v)$,
\[
X(s_b) =
\begin{cases}
-\frac12,&v_1>0,\\
\frac12,&v_1<0.
\end{cases}
\]
The set $\{t=\tau_b(x,v)\}$ is negligible and will be disregarded.

For the inflow boundary condition, define
\bq\label{eq/ in datum}
F_b^{\rm in}(t,x,v) :=
\begin{cases}
f_0(X(0),v), &t<\tau_b(x,v),\\
g_l(s_b,v), &t>\tau_b(x,v),\ v_1>0,\\
g_r(s_b,v), &t>\tau_b(x,v),\ v_1<0.
\end{cases}
\eq
For the Maxwell boundary condition, define
\bq\label{eq/ max datum}
F_b^{\rm max}(t,x,v) :=
\begin{cases}
f_0(X(0),v), &t<\tau_b(x,v),\\
f_l(s_b,v), &t>\tau_b(x,v),\ v_1>0,\\
f_r(s_b,v), &t>\tau_b(x,v),\ v_1<0,
\end{cases}
\eq
where $f_l$ and $f_r$ are the limiting traces constructed in Section \ref{sec/ maxwell limit}. In particular, their incoming parts satisfy
\begin{align*}
f_l(t,v) &= (1-\alpha_l)f_l(t,\scrV v) +\alpha_l\frkM(v)J_l(t), \quad v_1>0,\\
f_r(t,v) &= (1-\alpha_r)f_r(t,\scrV v) +\alpha_r\frkM(v)J_r(t), \quad v_1<0.
\end{align*}

We write $F_b$ for either \eqref{eq/ in datum} or \eqref{eq/ max datum}, according to the boundary condition. Then, for almost every $(t,x,v)$, we claim
\bq\label{eq/ slab mild}
\begin{aligned}
f(t,x,v) ={}& F_b(t,x,v) \exp\lt( -\int_{s_b}^t \rho_f(\sigma,X(\sigma)) \,\tnd\sigma \rt)\\
&+ \int_{s_b}^t \exp\lt( -\int_s^t \rho_f(\sigma,X(\sigma)) \,\tnd\sigma \rt) \rho_f(s,X(s)) M[f](s,X(s),v) \,\ds.
\end{aligned}
\eq
The formula is justified as follows. Set
\[
w:=\gamma(f)=\log(1+f), \quad G:= \gamma'(f)\rho_f(M[f]-f) = \frac{\rho_f(M[f]-f)}{1+f}.
\]
By Proposition \ref{prop/ collision delta}, $G\in L^1((0,T)\times I\times B_R)$ for every $T,R<\infty$.

We first verify the integrability of $G$ and $\rho_f$ along almost every free flight segment. For fixed $T,R>0$, the change of variables $y=X(s)=x-v_1(t-s)$ has unit Jacobian in the $x$-variable. Since $s_b(t,x,v)\le s\le t$ implies $X(s)\in I$, we obtain
\[
\int_0^T \iint_{I\times B_R} \int_{s_b}^t |G(s,X(s),v)| \,\ds\dx\dv\dt  \le T \int_0^T \iint_{I\times B_R} |G(s,y,v)| \,\dy\dv\ds <\infty,
\]
and similarly
\[
\int_0^T \iint_{I\times B_R} \int_{s_b}^t \rho_f(s,X(s)) \,\ds\dx\dv\dt  \le T|B_R| \int_0^T\int_I \rho_f(s,y)\,\dy\ds <\infty.
\]
Equivalently, the integrands may be extended by zero outside the slab before performing the change of variables. We emphasize that the estimates above use only that the translated integration domain is contained in $I$.

It follows from Fubini's theorem that, for almost every $(t,x,v)$,
\bq\label{eq/ char int}
\int_{s_b}^t \lt( |G(s,X(s),v)| + \rho_f(s,X(s)) \rt) \ds <\infty.
\eq

The trace and renormalization theory for transport equations \cite[Theorems 4.4--4.5]{Mis10} now implies that, for almost every such free flight segment, the map $s\mapsto w(s,X(s),v)$ has an absolutely continuous representative and satisfies
\[
\frac{\dd}{\ds} w(s,X(s),v) = G(s,X(s),v)
\]
for almost every $s\in(s_b,t)$.

The value at the backward endpoint is
\[
w(s_b,X(s_b),v) = \log(1+F_b(t,x,v)).
\]
Indeed, if $s_b=0$, this is the initial trace. If $s_b>0$, it is the incoming boundary trace identified in Section \ref{sec/ inflow limit} for \textbf{(In)} and in Section \ref{sec/ maxwell limit} for \textbf{(Max)}. Hence
\bq\label{eq/ log char}
w(s,X(s),v) = \log(1+F_b(t,x,v)) + \int_{s_b}^s G(\sigma,X(\sigma),v) \,\tnd\sigma
\eq
for $s_b\le s\le t$.

By \eqref{eq/ char int}, the right-hand side of \eqref{eq/ log char} is bounded as a function of $s$ on $[s_b,t]$. Consequently,
\[
f(s,X(s),v) = e^{w(s,X(s),v)}-1
\]
is also bounded on the same segment. The ordinary chain rule for absolutely continuous functions gives
\bq\label{eq/ f char}
\frac{\dd}{\ds} f(s,X(s),v) = (1+f(s,X(s),v) ) G(s,X(s),v)  = \rho_f(s,X(s)) ( M[f](s,X(s),v) - f(s,X(s),v)  ) 
\eq
for almost every $s\in(s_b,t)$.

We also note that every term in \eqref{eq/ f char} is integrable on the free flight segment. Indeed, if
\[
C_{t,x,v} := \sup_{s_b\le s\le t} (1+f(s,X(s),v) ) <\infty,
\]
then
\[
\int_{s_b}^t \lt| \rho_f(M[f]-f) \rt|(s,X(s),v) \,\ds \le C_{t,x,v} \int_{s_b}^t |G(s,X(s),v)| \,\ds <\infty,
\]
while
\[
\int_{s_b}^t \rho_f(s,X(s)) f(s,X(s),v) \,\ds \le \lt( \sup_{s_b\le s\le t} f(s,X(s),v) \rt) \int_{s_b}^t \rho_f(s,X(s)) \,\ds <\infty.
\]
Thus
\[
\rho_f(s,X(s)) M[f](s,X(s),v) \in L^1(s_b,t).
\]

Equation \eqref{eq/ f char} may now be written as
\[
\frac{\dd}{\ds} f(s,X(s),v) + \rho_f(s,X(s)) f(s,X(s),v) = \rho_f(s,X(s)) M[f](s,X(s),v).
\]
Multiplying by the integrating factor
\[
\exp\lt( \int_{s_b}^s \rho_f(\sigma,X(\sigma)) \,\tnd\sigma \rt)
\]
and integrating from $s_b$ to $t$ gives exactly \eqref{eq/ slab mild}.

Thus the mild formula is valid on the last free flight segment for almost every $(t,x,v)$. In the inflow case, its backward datum is completely prescribed by $f_0$, $g_l$, and $g_r$. In the Maxwell case, the backward datum at a wall is the incoming trace $f_l$ or $f_r$, which has already been identified in Section \ref{sec/ maxwell limit} and satisfies the Maxwell boundary condition. Hence the mild formula is a consequence of the trace identification; it is not used to establish the Maxwell boundary law.

%
%
%
%
%
%
%
%
%
%
%

\section*{Acknowledgments}
The work was supported by NRF grant no. 2022R1A2C1002820 and RS-2024-00406821.

%
%
%
%
%
%
%
%
%
%
%

\bibliographystyle{abbrv} 
\bibliography{bgkslab_CS}

\end{document}